\documentclass[10pt,a4paper]{article}
\pdfoutput=1 

\usepackage{lineno,hyperref}
\modulolinenumbers[5]

\usepackage[english]{babel}
\usepackage{graphicx}
\usepackage{marginnote}

\usepackage{tikz}
\usetikzlibrary{patterns}
\usepackage{bbm}
\usetikzlibrary{arrows}
\usepackage{mathtools}
 \usepackage{multirow}
\usepackage{amssymb}
\usepackage{amsfonts}
\usepackage{amsthm}
\usepackage{subfig}
\usepackage[scr]{rsfso} 
\usepackage{algorithm}     
\usepackage{algpseudocode} 

\usepackage{geometry}
\newcommand{\email}[1]{\hspace*{\stretch{1}}\emph{\texttt{#1}}}

\makeatletter
\def\blfootnote{\xdef\@thefnmark{$\star$}\@footnotetext}
\makeatother
\newenvironment{Authors}%
  {\begin{center}\begin{bfseries}}%
  {\end{bfseries}\end{center}}
\newenvironment{Addresses}%
  {\begin{flushleft}\begin{itshape}}%
  {\end{itshape}\end{flushleft}}

 \usepackage{fancyhdr}  
  
\fancypagestyle{plain}{
	\fancyhead{}
	\fancyhead[C]{\hfill Accepted for publication in 
ESAIM: Mathematical Modelling and Numerical Analysis (M2AN), October 2020}
}
\usepackage{pgfplots}
\usepackage{float}
\usepgfplotslibrary{external}
\DeclareRobustCommand{\tikzcaption}[1]{\tikzset{external/export next=false}#1}
\DeclareRobustCommand{\tikzref}[1]{\tikzcaption{\resizebox{!}{\refsize}{\ref{#1}}}}

\newtheorem{theorem}{Theorem}[section]
\newtheorem{theorem1}{Theorem}[section]
\newtheorem{theorem2}{Theorem}[section]
\newtheorem{proposition}[theorem]{Proposition}
\newtheorem{lemma}[theorem1]{Lemma}
\newtheorem{remark}[theorem2]{Remark}

  \newcommand{\vertiii}[1]{{\left\vert\kern-0.25ex\left\vert\kern-0.25ex\left\vert #1 
    \right\vert\kern-0.25ex\right\vert\kern-0.25ex\right\vert}}
  
\begin{document}

\thispagestyle{plain}

\title{ Space-time registration-based model reduction of parameterized one-dimensional hyperbolic PDEs}
 \date{}
 
 \maketitle
\vspace{-50pt} 
 
\begin{Authors}
Tommaso Taddei$^{1}$, Lei Zhang$^{1}$
\end{Authors}

\begin{Addresses}
$^1$
IMB, UMR 5251, Univ. Bordeaux;  33400, Talence, France.
Inria Bordeaux Sud-Ouest, Team MEMPHIS;  33400, Talence, France, \email{tommaso.taddei@inria.fr,lei.a.zhang@inria.fr} \\
\end{Addresses}

\begin{abstract}
We propose a model reduction procedure for rapid and reliable solution of parameterized hyperbolic partial differential equations. Due to the presence of parameter-dependent shock waves and contact discontinuities, these problems are extremely challenging for traditional model reduction approaches based on linear approximation spaces. 
The main ingredients of the proposed approach are (i) an adaptive space-time registration-based data compression procedure to align local features in a fixed reference domain, (ii) a space-time Petrov-Galerkin (minimum residual) formulation for the computation of the mapped solution, and (iii) a hyper-reduction procedure to speed up online computations. 
We present numerical results for 
a Burgers model problem and a 
shallow water model problem, to empirically demonstrate the potential of the method. 
\end{abstract}

\emph{Keywords:} 
parameterized  hyperbolic partial differential equations; model order reduction; data compression.
 
 \section{Introduction}
\label{sec:intro}
Several studies have demonstrated the 
inaccuracy of 
linear approximation spaces to deal with parameter-dependent hyperbolic partial differential equations (PDEs) with parameter-dependent shocks:
this challenge hinders the application of 
parameterized model order reduction (pMOR) techniques to this class of problems. 
To  address the slow decay of the Kolmogorov $N$-width 
of the solution manifold 
associated with the  problem of interest \cite{ohlberger2015reduced},
several authors have proposed to resort to nonlinear approximations.  
The goal of this paper is to develop 
a Lagrangian nonlinear compression method, and associated reduced-order model (ROM) for  
one-dimensional (systems of) conservation laws: the key element of the approach is  
 a space-time registration procedure to improve the linear reducibility of the solution manifold.
 In computer vision and pattern recognition, registration refers to the process of finding a  transformation that aligns two datasets; in this paper, registration refers to the process of finding a  parametric spatio-temporal  transformation that improves the linear compressibility of  the solution manifold.

We denote by $\mu$ the vector of model parameters in the parameter region $\mathcal{P} \subset \mathbb{R}^P$, we denote by $\Omega \subset \mathbb{R}^2$ the spatio-temporal domain over which the PDE is defined,  and we define the Hilbert space 
$\mathcal{X} = [L^2(\Omega)]^D$, 
where $D\geq 1$ denotes the number of state variables, and 
the Banach space
 ${\rm Lip}(\Omega)$  of Lipschitz functions over $\Omega$.
Then, we introduce the solution $U_{\mu}$  to the PDE for a given  $\mu$, 
$U_{\mu}: \Omega \to \mathbb{R}^D$,   and   the solution manifold
$\mathcal{M} := \{ U_{\mu} : \mu \in \mathcal{P} \} \subset \mathcal{X}$.
\emph{Linear compression} methods rely on approximations of the form 
$U_{\mu} \approx \widehat{U}_{\mu} = Z_N \widehat{\boldsymbol{\alpha}}_{\mu}$, where 
$Z_N: \mathbb{R}^N \to \mathcal{X}$ is a linear parameter-independent operator and $\widehat{\boldsymbol{\alpha}}: \mathcal{P} \to \mathbb{R}^N$ is a function of the  parameters.
On the other hand, we might distinguish between
\emph{Eulerian} and \emph{Lagrangian} 
\emph{nonlinear compression} methods.
We do not provide here a comprehensive survey of 
nonlinear compression methods, but rather cite a few representative approaches.
\begin{itemize}
\item
Eulerian approaches  \cite{lee2020model,reiss2018shifted,rim2019manifold,welper2017interpolation} consider approximations of the form $\widehat{U}_{\mu}:= Z_{N,\mu} (\widehat{\boldsymbol{\alpha}}_{\mu} )$, where
$Z_{N}: \mathbb{R}^N  \times \mathcal{P} \to \mathcal{X}$ is a suitably-chosen operator which might depend on the parameter $\mu$ and might also be nonlinear in  the first argument.
\item
Lagrangian approaches 
\cite{iollo2014advection,ohlberger2013nonlinear,taddei2015reduced,taddei2020registration}
rely on linear compression methods to approximate the mapped solution  $\widetilde{U}_{\mu}:= U_{\mu} \circ \boldsymbol{\Phi}_{\mu}$, 
where $\boldsymbol{\Phi} : \Omega \times \mathcal{P} \to \Omega$ is a suitably-chosen bijection from $\Omega$ into itself: the mapping $\boldsymbol{\Phi}$ should be chosen to make the mapped solution manifold $\widetilde{\mathcal{M}} = \{ \widetilde{U}_{\mu} : \mu \in \mathcal{P} \}$ more amenable for linear approximations.
\end{itemize}
Note that any  Lagrangian method is equivalent to an Eulerian method with 
$Z_{N,\mu} (\boldsymbol{\alpha}) := 
(\widetilde{Z}_N  \boldsymbol{\alpha})\; 
\circ \boldsymbol{\Phi}_{\mu}^{-1} 
$ for some linear operator 
$\widetilde{Z}_N: \mathbb{R}^N \to \mathcal{X}$, while the converse is not true.

{
Given a low-dimensional representation of the solution field $U_{\mu}$ for all $\mu \in \mathcal{P}$,  we might distinguish between \emph{non-intrusive} (data-fitted) and \emph{intrusive} (projection-based) approaches for the  prediction of the reduced coefficients for out-of-sample parameter values: the former rely on multi-target regression algorithms; the latter rely on Galerkin/Petrov-Galerkin projection to devise a ROM for  online predictions. To guarantee fast online evaluations of projection-based ROMs,  hyper-reduction techniques need to be applied: these techniques are designed to reduce assembling costs associated with residual evaluations. In the framework of linear compression methods, we refer to \cite{chakir2009methode,gallinari2018reduced,guo2018reduced} for representative examples of non-intrusive techniques, and to the reduced basis literature 
(e.g., \cite{hesthaven2016certified,quarteroni2015reduced}) for a thorough discussion about hyper-reduced projection schemes. 
Non-intrusive techniques can be trivially extended to nonlinear compression methods; on the other hand, the extension of projection-based schemes  is more challenging: 
for Lagrangian approaches, following
\cite{ohlberger2013nonlinear,taddei2020registration},
 we might perform projection and  hyper-reduction  in the mapped configuration;
 for Eulerian approaches, specialized techniques need to be proposed to ensure rapid ROM evaluations 
(see \cite{lee2020model,rim2019manifold}).}

%

In this paper, we  present a Lagrangian projection-based pMOR technique for conservation laws. Given $\mu \in \mathcal{P}$, 
we shall consider approximations of the form:
\begin{equation}
\label{eq:mapping_representation}
U_{\mu} \approx
\widehat{U}_{\mu} \circ  \boldsymbol{\Phi}_{\mu}^{-1},
\quad
{\rm with} \;
\widehat{U}_{\mu} = Z_N \widehat{\boldsymbol{\alpha}}_{\mu},
\quad
\boldsymbol{\Phi}_{\mu} =\texttt{id} +   W_M \widehat{\mathbf{a}}_{\mu}.
\end{equation}
Here, 
$\texttt{id} (\mathbf{x}) \equiv \mathbf{x}$ is the identity map,
$Z_N: \mathbb{R}^N \to \mathcal{X}$ and
$W_M: \mathbb{R}^M \to  [ {\rm Lip}(\Omega) ]^2$ are suitable linear operators, and 
$\widehat{\boldsymbol{\alpha}}: \mathcal{P} \to \mathbb{R}^N$  and
$\widehat{\mathbf{a}}: \mathcal{P} \to \mathbb{R}^M$ 
are  functions of the parameter $\mu$.
The key features of the present work are 
(i) a Lagrangian data compression technique for the construction of a low-dimensional representation of the solution field of the form \eqref{eq:mapping_representation},  
(ii)  a kernel-based regression algorithm for the online computation of the mapping coefficients $\widehat{\mathbf{a}}_{\mu}$,
and
(iii) a space-time  hyper-reduced Petrov-Galerkin (minimum residual) 
ROM for the online computation of 
the solution coefficients $\widehat{\boldsymbol{\alpha}}_{\mu}$.

Given the space-time snapshots  
$\{  U^k = U_{\mu^k}\}_{k=1}^{n_{\rm train}} \subset \mathcal{M}$, our data compression procedure returns 
(i) the linear operators $Z_N: \mathbb{R}^N \to \mathcal{X}$ and
$W_M: \mathbb{R}^M \to  [ {\rm Lip}(\Omega) ]^2$ in
\eqref{eq:mapping_representation}, and 
(ii) the coefficients $\{  \boldsymbol{\alpha}^k \}_{k=1}^{n_{\rm train}} \subset \mathbb{R}^N$ and  $\{  \mathbf{a}^k \}_{k=1}^{n_{\rm train}} \subset \mathbb{R}^M$ such that
$U^k \approx
\widehat{U}^k \circ (\boldsymbol{\Phi}^k)^{-1}$ where
$\widehat{U}^k = Z_N \boldsymbol{\alpha}^k$ and
$\boldsymbol{\Phi}^k =\texttt{id} +   W_M \mathbf{a}^k$.  
We develop an adaptive registration algorithm --- which is an extension of the approach in \cite{taddei2020registration} --- to construct the mappings  
 $\{ \boldsymbol{\Phi}^k \}_k$;
 on the other hand, 
we resort to 
proper orthogonal decomposition (POD, \cite{berkooz1993proper,volkwein2011model}) 
to generate the low-dimensional linear approximation operators  $Z_N, W_M$. 
Since the procedure can be viewed as a generalization of POD, we here refer to our approach as to 
\texttt{RePOD} (Registered POD): as rigorously showed below, our approach is \emph{general}, that is it does not depend on the underlying mathematical model.

The registration approach in \cite{taddei2020registration} relies on 
(i) a nonlinear non-convex optimization statement that aims at reducing the difference between a properly-chosen template $\bar{U} = U_{\bar{\mu}}$ and the mapped field $\widetilde{U}_{\mu} =U_{\mu} \circ \boldsymbol{\Phi}_{\mu}$ for $\mu \in \{ \mu^k \}_{k=1}^{n_{\rm train}}$, and on 
(ii) a generalization procedure based on kernel regression to extend the mapping to the whole parameter domain. In this work, we modify the optimization statement to penalize the distance from a low-dimensional space --- here referred to as \emph{template space} --- and we propose a greedy procedure to adaptively build the template space. We remark that several authors have proposed template-fitting strategies to deal with transport 
\cite{mendible2019dimensionality,mowlavi2018model,rowley2000reconstruction}:
here, by enforcing  the bijectivity of  $\boldsymbol{\Phi}$   from $\Omega$ in itself for all $\mu \in \mathcal{P}$, we might consider a standard projection-based ROM in the mapped  configuration.

{
Registration exploits the presence of local features (e.g., shock waves) that are \emph{topologically equivalent} for all $\mu \in \mathcal{P}$. These structures are present in a variety of physically-relevant parameter-dependent problems, including the two model problems considered in this paper.
In this respect, the use of a space-time formulation is instrumental to capture the interaction between multiple shock waves: in Appendix \ref{sec:data_compression_vis}, we show that space-only registration is not appropriate  to deal with the interaction between two shock waves, for a Burgers model problem.
} We observe that, starting with the seminal work  \cite{urban2014improved} , space-time formulations have been extensively considered in the pMOR literature: we refer to \cite{brunken2019parametrized,glas2019reduced} for  applications to  hyperbolic PDEs. 
However,
while in \cite{brunken2019parametrized,glas2019reduced}, space-time formulations are motivated by 
their superior stability and variational construction, which facilitate the error analysis, here
the space-time setting  is motivated by approximation considerations.
In this respect, our work shares important features with  the recent space-time adaptive discontinuous Galerkin (DG) method proposed in 
\cite{zahr2018optimization,zahr2019implicit}.

Given a new value of the parameter $\mu \in \mathcal{P}$, we resort to kernel-based regression to estimate the mapping coefficients $\widehat{\mathbf{a}}_{\mu}$, while we resort to a space-time minimum residual projection-based ROM to compute the solution coefficients $\widehat{\boldsymbol{\alpha}}_{\mu}$ in \eqref{eq:mapping_representation}. To reduce the costs of the minimum residual ROM, we first introduce a $J$-dimensional empirical test space \cite{taddei2018offline} $\mathcal{Y}_J$ 
and then we resort to an empirical quadrature procedure 
(EQP, \cite{yano2019discontinuous}) to reduce the online assembling and memory costs. 
In Appendix \ref{sec:AMR_linear_theory}, we present mathematical justifications for linear problems of the procedure used to construct the test space $\mathcal{Y}_J$; we further present numerical results to illustrate the superiority of minimum residual ROMs compared to   Galerkin ROMs.

Several authors have considered minimum residual 
 ROMs  for structural and fluid mechanics applications,  \cite{carlberg2013gnat,maday2001blackbox,yano2014space}.
We observe that in  \cite{carlberg2013gnat} the authors resort to Gappy-POD 
\cite{bui2003proper,everson1995karhunen}
to provide hyper-reduction of projection-based ROMs; on the other hand, 
similarly to Grimberg et al. \cite{grimberg2020stability}, 
we here resort to an EQP (see \cite{farhat2015structure,patera2017lp}).  
EQPs recast the problem of hyper-reduction as a suitable sparse representation problem and then rely on approximate techniques originally developed in the signal processing and optimization literature to approximate the solution. 
Here, we   adapt the procedure first presented in
\cite{yano2019discontinuous} for Galerkin ROMs to minimum residual ROMs to derive the sparse representation problem of interest; then, 
as in \cite{farhat2015structure}, 
we resort to a nonnegative linear squares method
(cf. \cite{lawson1974solving})  to find an approximate solution.

The paper is organized as follows.
In section \ref{sec:formulation}, we introduce the space-time variational formulation in the mapped configuration, and we introduce the two model problems considered for numerical assessment;
in section \ref{sec:registration}, we 
present the data compression procedure
based on space-time registration;
in section \ref{sec:projection_based_ROM},
we introduce  the hyper-reduced Petrov-Galerkin ROM;
finally, in section \ref{sec:numerics},
we present several numerical results to demonstrate the effectiveness of the proposed approach.
Several appendices complete the paper.

\subsection*{Notation}

By way of preliminaries, we introduce notation used throughout the paper.
We denote by $x$ a generic element of the spatial interval $(0,L)$ with $L>0$, and by $t$ a time instant in $(0,T)$ with $T>0$. In view of the space-time formulation,  we introduce   the spatio-temporal domain  $\Omega = (0,L) \times (0,T)$ and the gradient 
 $\nabla := [\partial_x, \partial_t]^T$. We denote by $\mathbf{x}=(x,t)$  a generic element of $\Omega$, and by $\mathbf{n}$   the outward normal to $\partial \Omega$.
 
 Given the reference domain  $\widetilde{\Omega} \subset \mathbb{R}^2$,  we introduce the parameterized mapping  
$\boldsymbol{\Phi}:  \widetilde{\Omega} \times \mathcal{P} \to \Omega$; we denote by $\mathbf{X}$ a generic element of $\widetilde{\Omega}$ and we define the mapped gradient
$\widetilde{\nabla} = [\partial_{X_1}, \partial_{X_2}]^T$. 
We further define the Jacobian matrix
$\mathbf{G}_{\mu} := \widetilde{\nabla}  \boldsymbol{\Phi}_{\mu}$ and determinant 
$g_{\mu} := {\rm det} ( \widetilde{\nabla}  \boldsymbol{\Phi}_{\mu} )$, which is assumed to be strictly positive over $\widetilde{\Omega}$.
  In this work, we consider bijections from $\Omega$ into itself: for this reason, we replace 
$\widetilde{\Omega}$  with $\Omega$.

We define the Hilbert space $\mathcal{X} = [L^2(\Omega)]^D$ where $D$ is the number of state variables. We denote by
$(\cdot, \cdot)$ the $L^2(\Omega)$ inner product,
$(w,v) = \int_{\Omega} w \cdot v \, d \mathbf{x} $, and by
$\| \cdot \| = \sqrt{(\cdot, \cdot)}$ the corresponding induced norm. Given the linear space $\mathcal{W} \subset \mathcal{X}$, 
$\Pi_{\mathcal{W}}: \mathcal{X} \to \mathcal{W}$ denotes the projection operator onto $\mathcal{W}$. 
We  further  denote by $\mathbf{e}_1,\ldots,\mathbf{e}_N$ the canonical basis in $\mathbb{R}^N$ and by $\|  \cdot \|_2$ the Euclidean norm. 
Given the linear operator $Z_N: \mathbb{R}^N \to \mathcal{X}$, we define $\zeta_n=
Z_N \mathbf{e}_n$ for $n=1,\ldots,N$ and we define the space
$\mathcal{Z}_N : = {\rm span} \{ \zeta_n \}_{n=1}^N$; in the remainder, we shall assume that $\{ \zeta_n \}_{n=1}^N$ is an orthonormal basis of $\mathcal{Z}_N$.

We further introduce the relative best-fit error  $E^{\rm bf}$, which is used to assess the performance of data compression. Given the reduced space $\mathcal{Z}_N \subset \mathcal{X}$ and $\mu \in \mathcal{P}$, we define
\begin{subequations}
\label{eq:best_fit_error}
\begin{equation}
E^{\rm bf}(\mu, \mathcal{Z}_N) :=
\frac{1}{ \| U_{\mu}  \|}
\min_{\zeta \in \mathcal{Z}_N} \| U_{\mu} - \zeta \|.
\end{equation}
We further define the "registered" best-fit error as
\begin{equation}
E^{\rm bf}(\mu, \mathcal{Z}_N, \boldsymbol{\Phi}) :=
\frac{1}{ \| U_{\mu}  \|}
\min_{\zeta \in \mathcal{Z}_N} 
\| U_{\mu} - \zeta \circ \boldsymbol{\Phi}_{\mu}^{-1}  \|
\;  = \;  
\frac{1}{\|  U_{\mu}  \|  }
\;\;
\min_{\zeta \in \mathcal{Z}_N} 
\sqrt{ \int_{\Omega} \, 
\, \left( U_{\mu} \circ \boldsymbol{\Phi}_{\mu}  - \zeta  \right)^2 \,  g_{\mu}  \;   d \mathbf{X}  }.
\end{equation}
Note that the latter expression is convenient for finite element calculations --- since it avoids the computation of the inverse map $\boldsymbol{\Phi}_{\mu}^{-1}$ in all quadrature points
---
and is used in the numerical results.
\end{subequations}

We use the method of snapshots (cf. \cite{sirovich1987turbulence}) to compute POD eigenvalues and eigenvectors. Given the snapshot set $\{ U^k \}_{k=1}^{n_{\rm train}} \subset \mathcal{M}$ and the inner product $(\cdot, \cdot)_{\rm pod}$, we define the Gramian matrix
$\mathbf{C} \in \mathbb{R}^{n_{\rm train}, n_{\rm train}}$,
$\mathbf{C}_{k,k'}= (U^k, U^{k'})_{\rm pod}$, 
and we define the POD eigenpairs    
$\{(\lambda_n, \zeta_n)  \}_{n=1}^{n_{\rm train}}$
as
$$
\mathbf{C} \boldsymbol{\zeta}_n = \lambda_n  \, \boldsymbol{\zeta}_n,
\quad
\zeta_n:=   \sum_{k=1}^{n_{\rm train}} \, \left(   \boldsymbol{\zeta}_n  \right)_k \, U^k,
\quad
n=1,\ldots,n_{\rm train},
$$
with $\lambda_1 \geq \lambda_2 \geq \ldots \lambda_{n_{\rm train}} \geq 0$. In our implementation, we orthonormalize the modes, that is
$( \zeta_n, \zeta_n)_{\rm pod}= 1$ for $n=1,\ldots,n_{\rm train}$.
To stress dependence of the POD space on the choice of the inner product, we use notation $L^2$-POD if $(\cdot, \cdot )_{\rm pod} = (\cdot, \cdot)_{L^2(\Omega)}$,  and
 $\|  \cdot \|_2$-POD if $(\cdot, \cdot)_{\rm pod}$ is the Euclidean inner product. Finally, we shall choose the size $N$ of the POD space based on the criterion
\begin{equation}
\label{eq:POD_cardinality_selection}
N := \min \left\{
N': \, \sum_{n=1}^{N'} \lambda_n \geq  \left(1 - tol_{\rm pod} \right) 
\sum_{i=1}^{n_{\rm train}} \lambda_i
\right\},
\end{equation} 
where $tol_{\rm pod} > 0$ is a given tolerance.

\subsection*{High-fidelity discretization}
 
 Our reduced-order formulation relies on a 
high-fidelity (hf)  DG 
finite element (FE) discretization; we refer to the textbook  \cite{hesthaven2007nodal} for an introduction to DG methods for conservation laws. We denote by $\mathcal{T}_{\rm hf} = \{  \texttt{D}^k \}_{k=1}^{N_{\rm e}} $
a non-overlapping  triangulation of $\Omega$,
{
we denote by $\{  \mathbf{x}_{i,k}^{\rm hf}: \, i=1,\ldots,n_{\rm lp}, k=1,\ldots, N_{\rm e} \}$ the nodes of the mesh, where $n_{\rm lp}$ is the number of degrees of freedom in each element,
we denote by
$\boldsymbol{\Psi}_k: \widehat{\texttt{D}} \to \texttt{D}^k$ the FE mapping between the reference element 
$\widehat{\texttt{D}}$ and the $k$-th element of the mesh}, and  
 we denote by 
$\partial \mathcal{T}_{\rm hf}  = \{  \texttt{f}^i \}_{i=1}^{N_{\rm f}}$ the set of facets of the mesh.   We denote by $\mathbf{N}^+$ the positive normal to a given facet in $\partial \mathcal{T}_{\rm hf}$: $\mathbf{N}^+$  coincides with the outward normal on $\partial \Omega$ and is chosen arbitrarily  for  interior facets.  We further define  the negative normal $\mathbf{N}^- = - \mathbf{N}^+$. Then, we define the DG FE space of order $p$,
\begin{equation}
\label{eq:DG_space}
\mathcal{X}_{\rm hf} = \left\{
v \in [L^2(\Omega)]^D \; : \;
v    \circ \boldsymbol{\Psi}_k   \in [\mathbb{P}^p (\widehat{\texttt{D}})  ]^D,
\quad
k=1,\ldots, N_{\rm e} 
\right\},
\end{equation}
where $\mathbb{P}^p$ denotes the space of two-dimensional polynomials of total degree at most $p$.
Given $w \in \mathcal{X}_{\rm hf}$ and 
$\mathbf{X} \in \partial \mathcal{T}_{\rm hf}$,
 we define  $w^{\pm}(\mathbf{X}) = \lim_{\epsilon \to 0^+} w(\mathbf{X} - \epsilon \mathbf{N}^{\pm} )$. 

We denote by $\{  \varphi_{i,k}^d =
\varphi_{i,k} \mathbf{e}_d    \}_{i,k,d}$ the Lagrangian basis of the space $\mathcal{X}_{\rm hf}$, with
$i=1,\ldots, n_{\rm lp} = \frac{p(p+1)}{2}$,
$k=1,\ldots,N_{\rm e}$,
$d=1,\ldots,D$, and we define  $N_{\rm hf} = {\rm dim} (\mathcal{X}_{\rm hf}) = n_{\rm lp} N_{\rm e} D$;
to shorten notation, we might also use the linear indexing 
$\varphi_j :=  \varphi_{i,k}^d =
\varphi_{i,k} \mathbf{e}_d$ where 
$j=j_{i,k,d} = i + (k-1) n_{\rm lp}  + (d-1) n_{\rm lp} N_{\rm e}$.
Given $w \in \mathcal{X}_{\rm hf}$, we denote by 
$\mathbf{w} \in \mathbb{R}^{N_{\rm hf}}$ the corresponding  FE vector,
$w (\cdot)= \sum_j ( \mathbf{w}   )_j \varphi_j (\cdot)$. With some abuse of notation, given the functional $F \in \mathcal{X}_{\rm hf}'$, we denote by 
$\mathbf{F} \in \mathbb{R}^{N_{\rm hf}}$ the 
corresponding  FE vector such that
$( \mathbf{F} )_j = F(  \varphi_j  )$, for 
$j=1,\ldots, N_{\rm hf}$.

In view of the definition of the ROM, we introduce the discrete $L^2$ and $H^1$ norms
$\mathbf{X}_{\rm hf},\mathbf{Y}_{\rm hf} \in \mathbb{R}^{N_{\rm hf}, N_{\rm hf}}$ such that
\begin{equation}
\label{eq:DG_norms}
\left\{
\begin{array}{ll}
\displaystyle{
\left(
\mathbf{X}_{\rm hf}
\right)_{j,j'}
=
\sum_{k=1}^{N_{\rm e}} \; 
\int_{\texttt{D}^k } \; \varphi_{j'} \cdot \varphi_j \, d \mathbf{X},
}
&
\\[3mm]
\displaystyle{
\left(
\mathbf{Y}_{\rm hf}
\right)_{j,j'}
=
\sum_{k=1}^{N_{\rm e}} \; 
\int_{\texttt{D}^k } \; \varphi_{j'} \cdot \varphi_j 
+
\nabla \varphi_{j'} \cdot  \nabla \varphi_j 
\, d \mathbf{X},
-
\int_{\partial \texttt{D}^k } \;
\mathcal{H}^{\rm d}(    \varphi_{j'}, 
\varphi_j ;  \mathbbm{1}, \mathbf{N}) \;   d \mathbf{X}
}
&
\\
\end{array}
\right.
\end{equation}
where  
$\mathcal{H}^{\rm d}(  \cdot, \cdot ;  \mathbbm{1}, \mathbf{N})$ is 
a suitable diffusion flux, here associated with the diffusion matrix $\mathbf{K} = \mathbbm{1}$. In this work, 
we consider   the
 BR2  flux  introduced in \cite{bassi1997high} (see also 
\cite{bassi2005discontinuous}).
We refer to   \cite{arnold2002unified} for a detailed discussion concerning the analysis of DG formulations for second-order elliptic problems.
In the following, with some abuse of notation, we denote by 
$\mathcal{X}_{\rm hf}$ the linear space \eqref{eq:DG_space} equipped with the discrete $L^2$ norm
$\| v  \| := \sqrt{ \mathbf{v}^T \mathbf{X}_{\rm hf} \mathbf{v} }$, and
we denote by 
$\mathcal{Y}_{\rm hf}$ the linear space \eqref{eq:DG_space} equipped with the discrete $H^1$ norm
$\vertiii{v} := \sqrt{ \mathbf{v}^T \mathbf{Y}_{\rm hf} \mathbf{v} }$: note that 
$\mathcal{X}_{\rm hf}$  and $\mathcal{Y}_{\rm hf}$  coincide as linear spaces but are different  Hilbert spaces.
 We denote by $\texttt{R}_{\mathcal{Y}_{\rm hf}}: \mathcal{Y}_{\rm hf}' \to \mathcal{Y}_{\rm hf}$ the Riesz operator in $\mathcal{Y}_{\rm hf}$.

{
In view of registration, given the mesh $\mathcal{T}_{\rm hf}$, we define the mapped mesh $\boldsymbol{\Phi}(\mathcal{T}_{\rm hf})$ associated with the mapping 
 $\boldsymbol{\Phi}: \Omega \to \Omega$ such that  $\boldsymbol{\Phi}(\mathcal{T}_{\rm hf})$ shares with $\mathcal{T}_{\rm hf}$ the same connectivity matrix and the nodes are given by 
$\{ \boldsymbol{\Phi} (  \mathbf{x}_{i,k}^{\rm hf} ) : \, i=1,\ldots,n_{\rm lp}, k=1,\ldots, N_{\rm e} \}$.
Furthermore, we define the FE space of order $p$ such that  
\begin{equation}
\label{eq:mappedFE_space}
\mathcal{X}_{\rm hf, \Phi} = 
 \left\{
v \in [L^2(\Omega)]^D \; : \;
v    \circ \boldsymbol{\Psi}_{k,\Phi}   \in [\mathbb{P}^p (\widehat{\texttt{D}})  ]^D,
\quad
k=1,\ldots, N_{\rm e} 
\right\},
\end{equation}
where
$\boldsymbol{\Psi}_{k,\Phi}$ is the FE mapping between the reference element $\widehat{\texttt{D}}$ and the $k$-th element of $\boldsymbol{\Phi}(\mathcal{T}_{\rm hf})$. Given $w \in \mathcal{X}_{\rm hf}$,  we can define the FE field 
$w_{\Phi}$ in $\mathcal{X}_{\rm hf, \Phi}$ such that
$w_{\Phi}( \boldsymbol{\Phi} (  \mathbf{x}_{i,k}^{\rm hf} ) ) = w(   \mathbf{x}_{i,k}^{\rm hf} )$: note that 
$w_{\Phi}$ and $w$ share the same FE vector and 
$w_{\Phi}$ is an  approximation of $ w\circ \boldsymbol{\Phi} ^{-1}$ ---
and $w_{\Phi} \equiv w\circ \boldsymbol{\Phi} ^{-1}$ if $\boldsymbol{\Phi}$ is piecewise linear.
}

\section{Formulation}
\label{sec:formulation}
\subsection{Space-time formulation of conservation laws}
\label{sec:space_time}

In this section, we omit dependence on the parameter $\mu$ for notational brevity.
We consider a general system of one-dimensional conservation laws:
\begin{equation}
\label{eq:1Dconslaws}
\left\{
\begin{array}{ll}
\partial_t U + \partial_x f(U) = S(U) & {\rm in} \; \Omega \\[3mm]
U(\cdot; 0) = U_{\rm D,0}(\cdot) & {\rm on} \; 
\Gamma_{\rm in,0} :=
 (0,L) \times \{ 0 \}  \\
\end{array}
\right.
\end{equation}
where $U:\Omega \to \mathbb{R}^D$ is the 
vector of conserved variables,
$f: \mathbb{R}^D \to \mathbb{R}^D$ is the physical flux, and $S: \mathbb{R}^D \to \mathbb{R}^D$ is the source term.
The problem is completed with suitable boundary conditions, which depend on the number of incoming characteristics.
We provide two examples of problems of the form \eqref{eq:1Dconslaws} at the end of this section. We remark that the solution $U$ may contain discontinuities and might not be unique: we here seek $U$ satisfying \eqref{eq:1Dconslaws} away from discontinuities and satisfying suitable Rankine-Hugoniot and entropy conditions at discontinuities, \cite{leveque1992numerical}.

We recast \eqref{eq:1Dconslaws} as
\begin{equation}
\label{eq:steady_conslaw}
\nabla \cdot F(U) \, = \, S(U)  \quad {\rm in} \; \Omega
\end{equation}
where $F(U) = [f(U), U]$,
$F:\mathbb{R}^{D} \to \mathbb{R}^{D,2}$.
 The problem is completed with suitable boundary conditions on $\{0,L \} \times (0,T)$, which depend on the number of incoming characteristics. Note that 
at  $\Gamma_{\rm in,0}$ all $D$ characteristics are incoming and
at $(0,L) \times \{T  \}$  all $D$ characteristics are outward: 
 this implies that \eqref{eq:steady_conslaw} requires the prescription of the full initial datum at $t=0$ and does not require any condition at $t=T$, consistently with \eqref{eq:1Dconslaws}.
Note that, for a proper choice of $F$, \eqref{eq:steady_conslaw} encapsulates two-dimensional steady conservation laws and unsteady one-dimensional conservation laws.

We recast \eqref{eq:steady_conslaw} on a reference domain: this will lead to the variational formulation exploited in section \ref{sec:projection_based_ROM} for model reduction. 
Given   the mapping 
$\boldsymbol{\Phi}:  \Omega  \to \Omega$, 
recalling  standard change-of-variable formulas (cf. Appendix \ref{sec:change_variable}), we obtain that the mapped solution field $\widetilde{U} = U \circ \boldsymbol{\Phi}$ satisfies 
\begin{subequations}
\label{eq:mapped_strong_form}
\begin{equation}
\widetilde{\nabla}   \cdot F_{\Phi}(\widetilde{U}) \, = \, S_{\Phi}(\widetilde{U}) \quad  {\rm in} \; \Omega,
\end{equation}
where  
\begin{equation}
\label{eq:mapped_strong_form_b}
F_{\Phi}(\cdot) = g  F(\cdot) \mathbf{G}^{-T} ,
\quad
S_{\Phi}(\cdot) = g  S(\cdot).
\end{equation}
We recall that $\mathbf{G},g$ are the Jacobian matrix and determinant defined in the introduction.
\end{subequations}

\begin{remark}
\label{remark:ALE}
Arbitrary Lagrangian Eulerian (ALE, \cite{hirt1974arbitrary,donea2017arbitrary}) methods involve the use of space-time mappings of the form
$\boldsymbol{\Phi}(X,t) = [    \Phi_1(X,t), t]$: this class of mappings allows the use of time-marching schemes to solve \eqref{eq:steady_conslaw}. Note, however, that  the
 deformation that can be achieved using this class of mappings is relatively modest: in particular, for any given time $t>0$, the mapped solution $\widetilde{U}$ has the same number of discontinuities as $U$, possibly at different locations. In the numerical results
of section \ref{sec:numerics} and Appendix \ref{sec:data_compression_vis}, we empirically show that the possibility of ``moving" shock waves in space and time is key to improve the linear reducibility of the mapped manifold.
\end{remark}

\subsection{High-fidelity space-time formulation}
We discretize \eqref{eq:mapped_strong_form} using a high-order nodal DG method.
In presence of shocks and other discontinuities, 
high-order schemes for hyperbolic PDEs require specific stabilization techniques to avoid instabilities. In this work, we resort to the sub-cell shock capturing method based on artificial viscosity proposed in 
\cite{persson2006sub}. More in detail, we consider the piecewise-constant artificial viscosity
\begin{subequations}
\label{eq:artificial_viscosity}
\begin{equation}
\varepsilon(U)\big|_{\texttt{D}^k} \, = \, \varepsilon_0 + 
\varepsilon_{\rm pp} \left( \log_{10} S_k \right),
\quad
S_k \, = \, \frac{\| s(U)  - \Pi_{p-1} s(U)  \|_{L^2(\texttt{D}^k)}}{  \| s(U) \|_{L^2(\texttt{D}^k)} },
\end{equation}
where $\varepsilon_0>0$ is a positive constant,
$s(U)$ is a suitable scalar function of the state, 
$\Pi_{p-1}:  \mathbb{P}^p  \to  \mathbb{P}^{p-1}$ is the projection onto the space of polynomials of total degree $p-1$,   and
\begin{equation}
\varepsilon_{\rm pp}(s) \, = \,
\left\{
\begin{array}{ll}
0 & s < s_0 - \kappa \\
\displaystyle{
\frac{\epsilon_0}{2} \left(1 + \sin \left(  
\frac{\pi (s - s_0)}{2 \kappa} \right)   \right)}
&
s_0 - \kappa \leq s < s_0 + \kappa \\
\epsilon_0
&
s \geq  s_0 + \kappa \\
\end{array}
\right.
\end{equation}
In this work, we consider $s(U)=U$ for Burgers equation and 
$s(U)=h$ for the shallow water equations (here, $h$ is the flow height, see  section \ref{sec:saintvenant_model_problem}); representative values of the constants in \eqref{eq:artificial_viscosity} considered in the numerical simulations are
$s_0=-2.5, \kappa=1.5, \epsilon_0 = 10^{-2},
\varepsilon_0=5 \cdot 10^{-4}$.
\end{subequations}

We have now the elements to introduce the  DG discretization of \eqref{eq:mapped_strong_form}: find
$\widetilde{U}_{\rm hf} \in \mathcal{X}_{\rm hf}$ such that
\begin{equation}
\label{eq:DG_pb}
R_{\Phi} \left(  \widetilde{U}_{\rm hf}, v \right)
=
R_{\Phi}^{\rm c} \left(  \widetilde{U}_{\rm hf}, v \right)
+
R^{\rm d} \left(  \widetilde{U}_{\rm hf}, v \right)
= 0,
\quad
\forall \; v \in \mathcal{X}_{\rm hf},
\end{equation}
where the variational discrete operator $R_{\Phi}: \mathcal{X}_{\rm hf} \times \mathcal{X}_{\rm hf} \to \mathbb{R}$ is the sum of the convection and diffusion contributions. Here, $R_{\Phi}^{\rm c} $ is given by
\begin{subequations}
\label{eq:DG_convective}
\begin{equation}
\begin{array}{l}
\displaystyle{
R_{\Phi}^{\rm c}(w,v)  =
\sum_{k=1}^{ N_{\rm e} } \;  
r_k^{\rm c}(w,v)
=
\sum_{k=1}^{ N_{\rm e} } \; 
\int_{\partial \texttt{D}^k} 
v \cdot \mathcal{H}_{\Phi}(w^+, w^-, \mathbf{N})  d \mathbf{X} \, - \,
\int_{  \texttt{D}^k} 
\widetilde{\nabla}  v\cdot  
F_{\Phi}(w )
 \; d \mathbf{X} \, 
-
\int_{ \texttt{D}^k} 
v\cdot S_{\Phi}(w) d \mathbf{X},
}
\\
\end{array}
\end{equation}
where $\mathcal{H}_{\Phi}$ is the numerical convective flux. Following \cite{zahr2019implicit}, we choose $\mathcal{H}_{\Phi}$  such that
\begin{equation}
\mathcal{H}_{\Phi}(\widetilde{U}_{\rm hf}^+, \widetilde{U}_{\rm hf}^-, \mathbf{N} )
=
\| g \mathbf{G}^{-T} \mathbf{N}  \|_2 
\mathcal{H}( \widetilde{U}_{\rm hf}^+, \widetilde{U}_{\rm hf}^-, \mathbf{n} ),
\quad
{\rm with} \; \; 
\mathbf{n} = \frac{\mathbf{G}^{-T} \mathbf{N} }{\| \mathbf{G}^{-T} \mathbf{N} \|_2},
\end{equation}
where $\mathcal{H}$ is a standard numerical flux in the physical domain. Note that for piecewise-linear maps \eqref{eq:DG_convective} is equivalent to the "Eulerian"  DG convective term  associated with the numerical flux  $\mathcal{H}$, and with the triangulation
$\mathcal{T}_{\rm hf}^{\star} = \{ \boldsymbol{\Phi} ( 
\texttt{D}^k) \}_{k=1}^{N_{\rm e}}$. In the numerical examples of this paper, we resort to the local  Lax-Friedrichs (Rusanov) flux:
\begin{equation}
\mathcal{H}(U^+, U^-, \mathbf{n} )
=
\frac{1}{2} \left( F(U^+) + F(U^-) \right) 
-
\frac{\tau}{2} \mathbf{n}^+ ( U^+ - U^- )^T,
\quad
\tau :=
\max \{  \big|\partial_U F(U^+) \mathbf{n}  \big|, \;
\big|\partial_U F(U^-) \mathbf{n}  \big|  \}.
\end{equation}
\end{subequations}
The diffusion form 
$R^{\rm d} $ is defined as 
\begin{equation}
\label{eq:DG_diffusion}
R^{\rm d}(w,v)  =
\sum_{k=1}^{ N_{\rm e} } \;  
r_k^{\rm d}(w,v)
=
\sum_{k=1}^{ N_{\rm e} } \; 
\int_{\partial \texttt{D}^k} 
\mathcal{H}^{\rm d}( w, v;  \varepsilon(U) \mathbbm{1}, \mathbf{N}) \;   d \mathbf{X} \, - \,
\int_{  \texttt{D}^k} 
\; \varepsilon(U) \; 
\widetilde{\nabla}  w  \cdot  
\widetilde{\nabla}  v
 \; d \mathbf{X},
\end{equation}
where  $\mathcal{H}^{\rm d}$ is  the BR2 diffusion flux 
associated with the diffusion matrix
$\mathbf{K} = \varepsilon(U) \mathbbm{1}$. Note that we consider an   artificial-diffusion form that is independent of the mapping $\boldsymbol{\Phi}$.

\begin{remark}
 {In this work, we resort to the  space-time   solver  discussed here to generate the hf snapshots;}
 the space-time   formulation of the conservation law \eqref{eq:1Dconslaws} also  provides the foundations for the projection-based ROM proposed in section 
\ref{sec:projection_based_ROM}.
We envision, however, that
space-time solvers might not be feasible for large-scale two-dimensional and three-dimensional problems: 
 {in this case, we might employ a third-party time-marching solver to generate the space-time snapshots and then use the space-time formulation  exclusively for ROM calculations.}
We refer to a future work for the integration between an external {time-marching} solver and the space-time ROM.
We further remark that other choices for the convection and diffusion fluxes and for the artificial viscosity are available: we refer to the   DG literature for thorough discussions and comparisons.
\end{remark}

\subsection{Model problems}
\label{sec:model_problems}
\subsubsection{A Burgers model problem}
\label{sec:burgers_model_problem}

We consider the Burgers equation:
\begin{subequations}
\label{eq:burgers_strong}
\begin{equation}
\left\{
\begin{array}{ll}
\displaystyle{
\partial_t U_{\mu}  + \frac{1}{2}  \partial_x U_{\mu}^2 \; = \; 0 }
&
\displaystyle{
(x,t) \in \Omega = (0,L) \times (0,T)
}
\\[3mm]
\displaystyle{
U_{\mu}(x,0) \; = \;  U_{\rm D, \mu}(x) }
&
\displaystyle{
x \in   (0,L)  
}
\\[3mm]
\displaystyle{
U_{\mu}(0,t) \; = \;  U_{\rm D, \mu}(0) }
&
\displaystyle{
t \in   (0,T)  
}
\\
\end{array}
\right.
\end{equation}
where $L=1$, $T=0.8$, and
\begin{equation}
U_{\rm D,\mu}(x)
= \mu_1 
\left(
2 - H_{\nu}(x - \mu_2) 
- H_{\nu}(x - \frac{1}{2}) 
\right)
+
0.3 \sin \left( \pi x \right),
\quad
H_{\nu}(s)  = \frac{1}{1+ e^{-\nu s}},
\;\;\nu=260.
\end{equation}
Here, we consider the parameter domain $\mathcal{P} = [1,1.3] \times [0.25,0.35]$. 
{We observe that a similar model problem has been considered in \cite[section 4.2]{peherstorfer2020model}.}
\end{subequations}

\begin{figure}[h!]
\centering
\subfloat[] 
{  \includegraphics[width=0.4\textwidth]
 {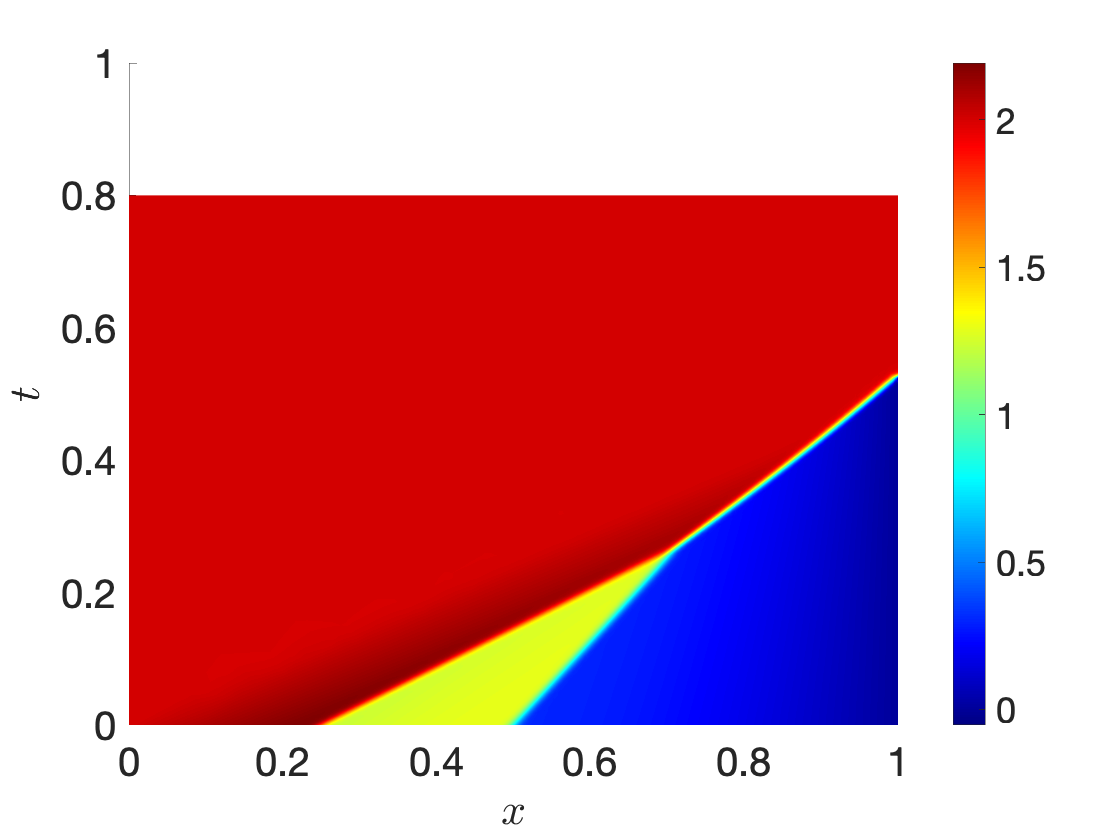}}
  ~~
 \subfloat[] 
{  \includegraphics[width=0.4\textwidth]
 {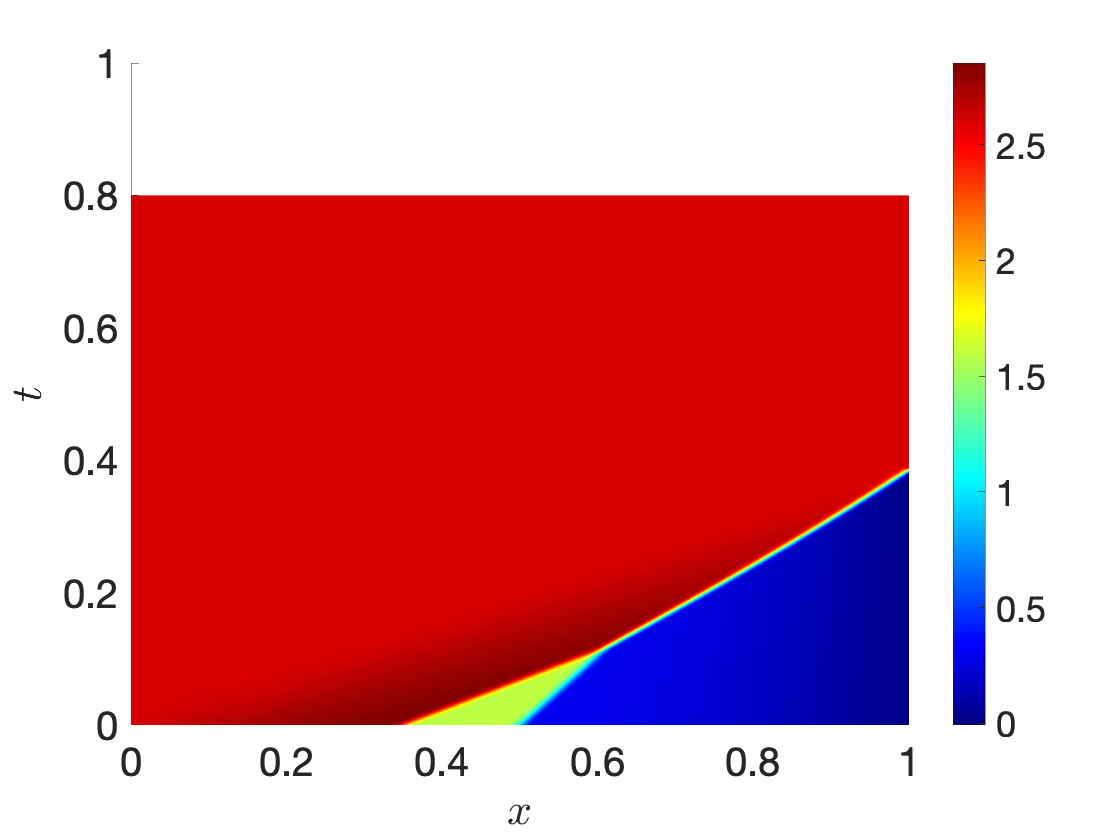}}
 
 \caption{solution to Burgers equation. (a) $\mu=[1,0.25]$. (b)
 $\mu=[1.35,0.35]$.}
 \label{fig:vis_burgers}
  \end{figure}  

Figure \ref{fig:vis_burgers} shows the behavior of $U$ over $\Omega$ for two values of $\mu$.
The solution is characterized by the transition between the three ``prototypical" behaviors depicted in Figure \ref{fig:vis_burgers_2}:
for small values of $t$, the solution exhibits two shock waves (cf. 
Figure \ref{fig:vis_burgers_2}(a));
for intermediate values of $t$, the solution exhibits one shock wave 
(cf. 
Figure \ref{fig:vis_burgers_2}(b));
for  large values of $t$, the solution is nearly constant over $(0,L)$  
(cf. 
Figure \ref{fig:vis_burgers_2}(c)).  
Despite its simplicity, this problem is extremely challenging for model reduction techniques: linear methods (i.e., methods based on linear approximation spaces) require a large number of modes to correctly represent the solution; on the other hand, 
to our knowledge, the transition from two shocks to one shock and from one shock to the smooth solution poses fundamental challenges for several nonlinear proposals: see  the discussion in \cite{rim2019manifold}.

\begin{figure}[h!]
\centering
\subfloat[$t=0.05$] 
{  \includegraphics[width=0.32\textwidth]
 {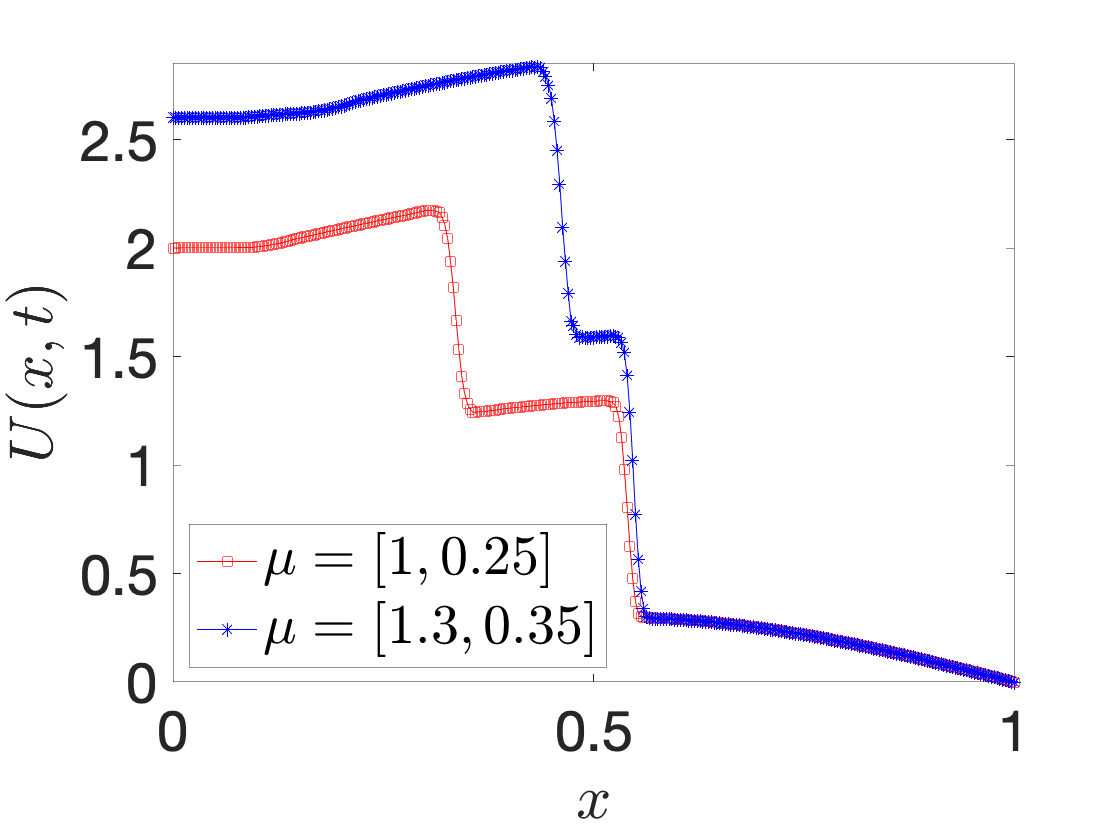}}
  ~~
 \subfloat[$t=0.3$] 
{  \includegraphics[width=0.32\textwidth]
 {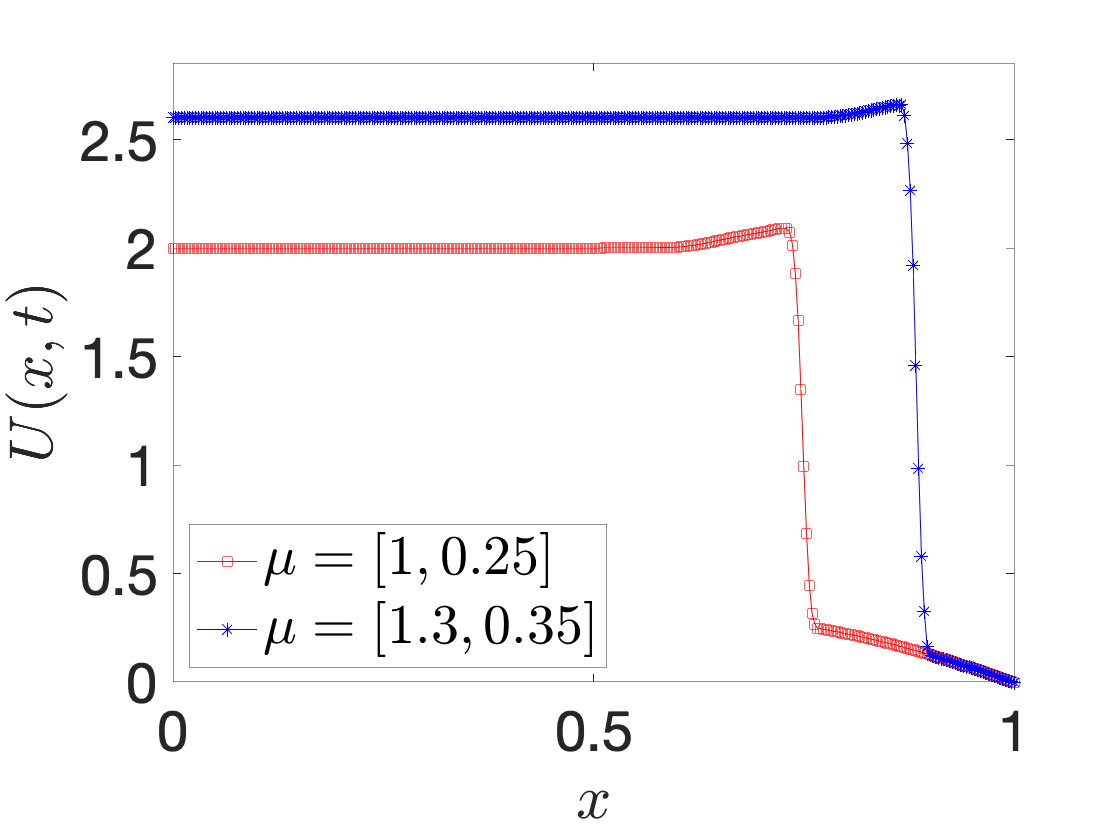}}
   ~~
 \subfloat[$t=0.8$] 
{  \includegraphics[width=0.32\textwidth]
 {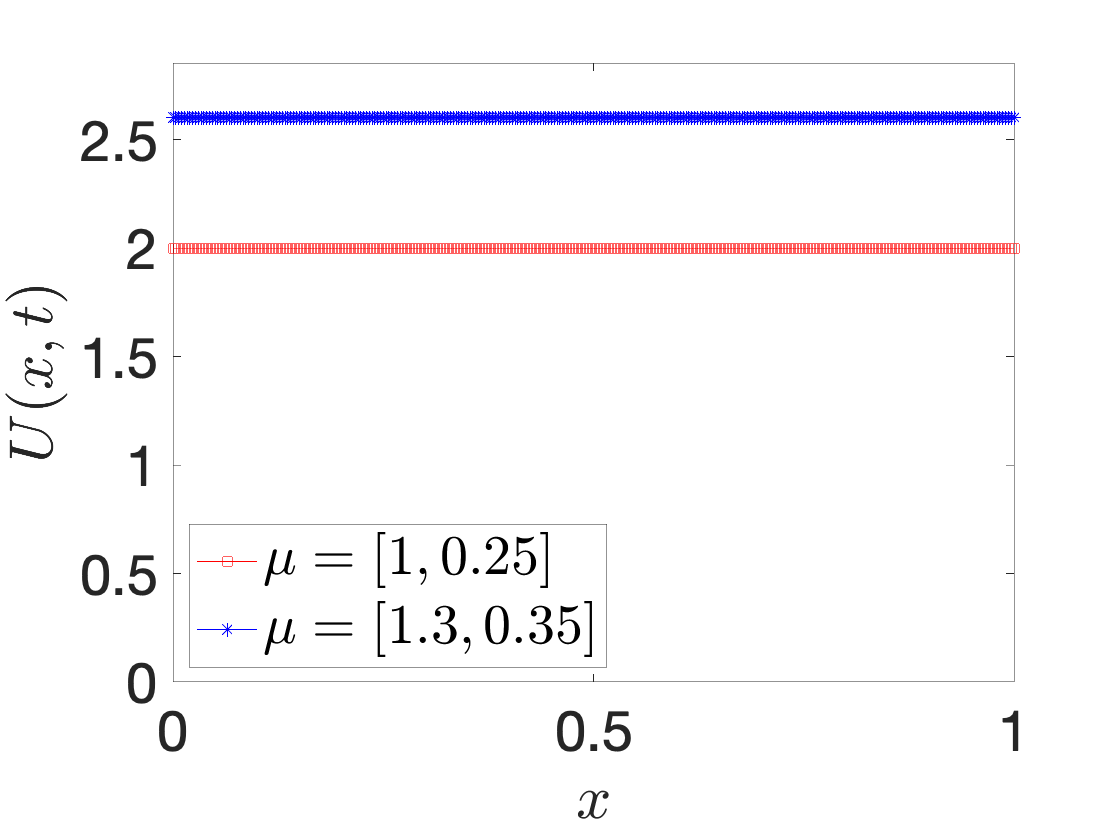}}
 
 \caption{solution to Burgers equation for   $\mu=[1,0.25]$ and 
 $\mu=[1.35,0.35]$ at  three different time instants.}
 \label{fig:vis_burgers_2}
  \end{figure}  

\subsubsection{A shallow-water model problem}
\label{sec:saintvenant_model_problem}

We consider a transient  shallow water (Saint Venant)  flow over a bump in a frictionless channel.
We here denote by   $U=[h, q]^T$ the vector of conserved variables,  where $q=h u$ is the discharge, $h$ is the flow height, and $u$ is the flow $x$-velocity; 
then, we introduce the flux $f(U): =  [ q, h u^2 + \frac{g}{2}
h^2  ]^T$ where $g=9.81$ is the gravity acceleration;   we further  introduce the (parameter-independent) bathymetry  $b$ such that
\begin{subequations}
\label{eq:saint_venant_strong}
\begin{equation}
\label{eq:bathymetry_saint_venant}
b(x) :=  -0.2 \, + \, e^{-0.125 (x - 10)^4}.
\end{equation}
We consider the system of Saint-Venant equations:
\begin{equation}
\left\{
\begin{array}{ll}
\displaystyle{
\partial_t U_{\mu}  +   \partial_x f ( U_{\mu} )  \; = \; S(U_{\mu}) 
}
&
\displaystyle{
(x,t) \in \Omega = (0,L) \times (0,T);
}
\\[3mm]
\displaystyle{U_{\mu}(x,0) \; = \;  U_{\rm bf}(x) }
&
\displaystyle{x \in   (0,L)};
\\[3mm]
\displaystyle{
q_{\mu}(0,t) \; = \;  q_{\rm in, \mu}(t),
\;\;
h_{\mu}(L,t) \; = \;   h_{\infty}
}
&
\displaystyle{
t \in   (0,T) ;
}
\\
\end{array}
\right.
\end{equation}
where  
$L=25$, $T=3$, 
$h_{\infty}=2$, the source term satisfies
$S(U) = [0, - g  h \partial_x b]^T$ and the discharge $q_{\rm in, \mu}$ satisfies
\begin{equation}
q_{\rm in, \mu}(t)
\; =\;
q_0  \;  \left( 1+ \mu_1 \, t \, 
e^{  -\frac{1}{2 \mu_2^2}(t - 0.05)^2   } \right),
\quad
q_0 = 4.4.
\end{equation}
Note that the parameter $\mu_1$ influences the {peak} of the incoming discharge, while 
$\mu_2$ affects the time scale and the integral
$\int_0^T q_{\rm in, \mu}(t) \; dt$. Here, we consider
$\mu \in \mathcal{P} = [2,8] \times [0.1,0.2]$. Finally, $ U_{\rm bf}$ corresponds to the limit solution for $t \to \infty$ of the PDE:
\begin{equation}
\left\{
\begin{array}{ll}
\displaystyle{
\partial_t U   +   \partial_x f ( U )  \; = \; S(U) 
}
&
\displaystyle{
(x,t) \in \Omega = (0,L) \times (0,\infty);
}
\\[3mm]
\displaystyle{U(x,0) \; = \;  0}
&
\displaystyle{x \in   (0,L)};
\\[3mm]
\displaystyle{
q(0,t) \; = \;  q_0, 
\;\;
h(L,t) \; = \;   h_{\infty},
}
&
\displaystyle{
t \in   (0,\infty)  .
}
\\
\end{array}
\right.
\end{equation}
\end{subequations}

In Figures \ref{fig:vis_st_venant} and \ref{fig:vis_st_venant_2}, we show the 
 behavior of the free surface $z=h +b$ for two values of the parameter $\mu$ in $\mathcal{P}$.
We observe that for $t \approx 1.5$ the incoming wave associated with the inflow boundary condition interacts with the bump; for sufficiently large values of $\mu_1$, we also observe a backward-propagating wave generated by the interaction between the incoming wave and the bump.

\begin{figure}[h!]
\centering
\subfloat[] 
{  \includegraphics[width=0.4\textwidth]
 {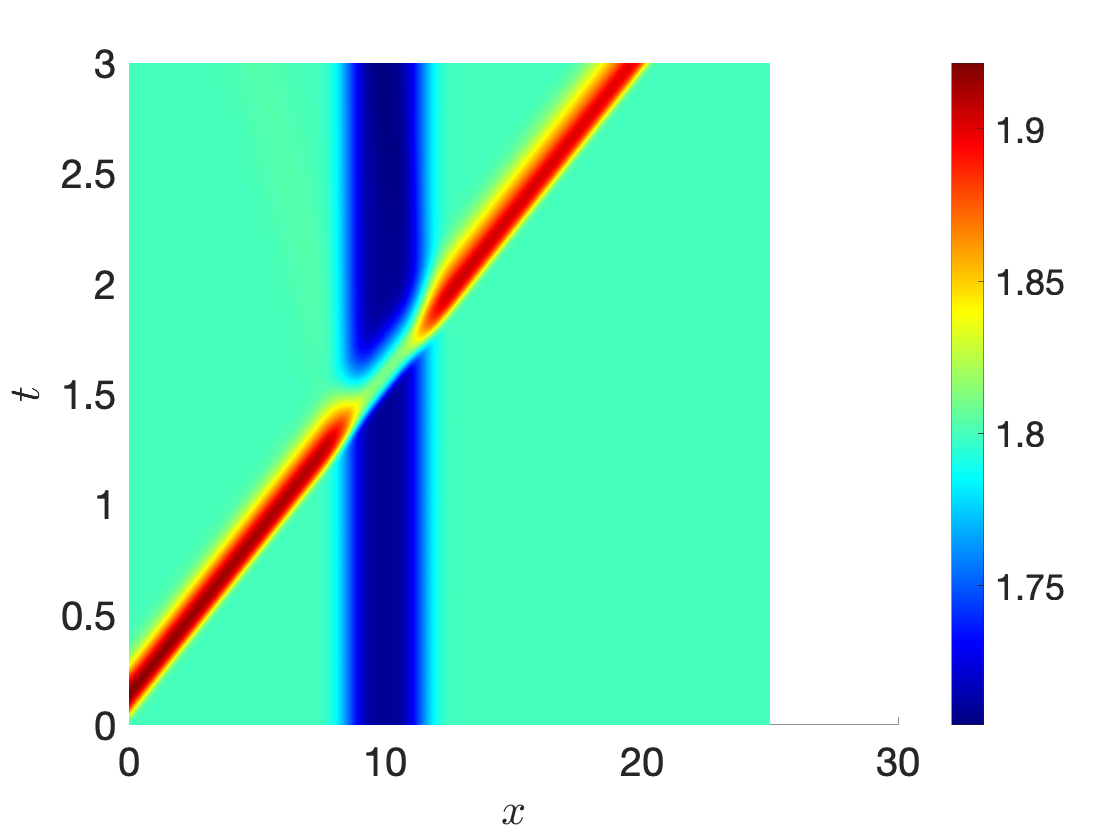}}
  ~~
 \subfloat[] 
{  \includegraphics[width=0.4\textwidth]
 {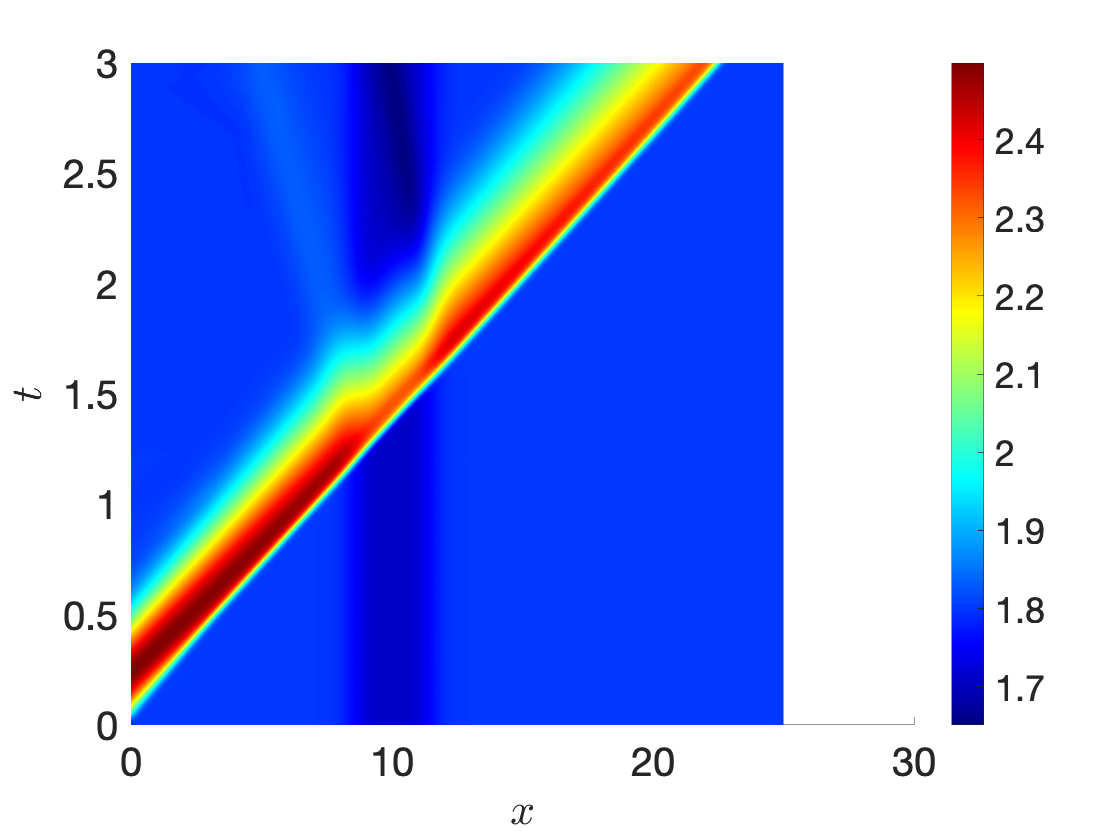}}
 
 \caption{solution to shallow water equations; behavior of the free surface $z=h +b$. (a) $\mu=[2,0.1]$. (b) $\mu=[8,0.2]$.}
\label{fig:vis_st_venant}
\end{figure}  

\begin{figure}[h!]
\centering
\subfloat[$t=0.4$] 
{  \includegraphics[width=0.32\textwidth]
 {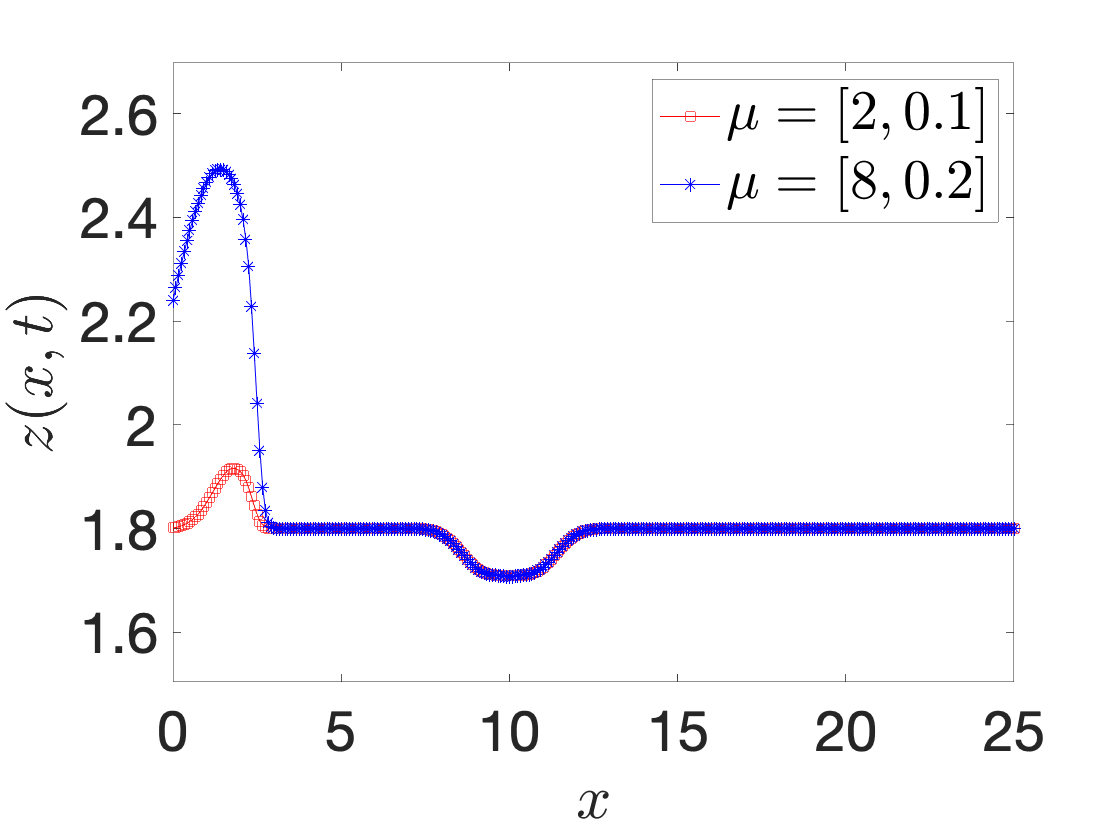}}
  ~~
 \subfloat[$t=1.5$] 
{  \includegraphics[width=0.32\textwidth]
 {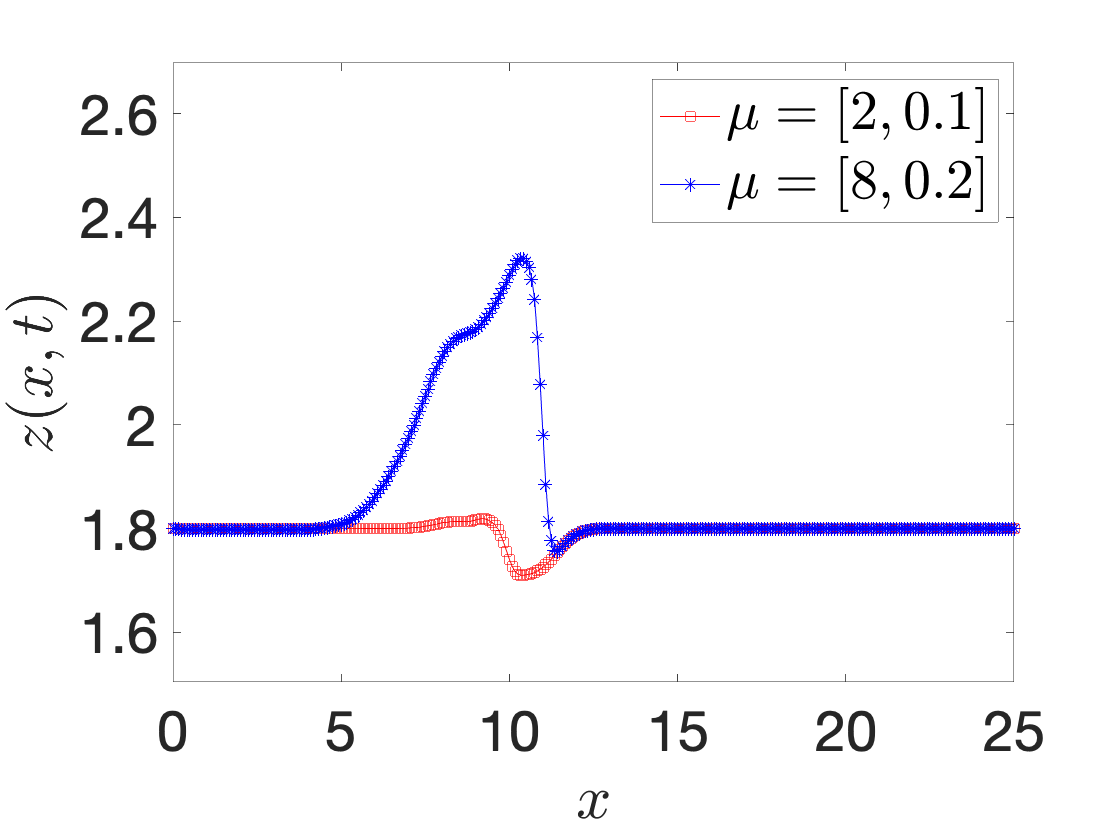}}
   ~~
 \subfloat[$t=3$] 
{  \includegraphics[width=0.32\textwidth]
 {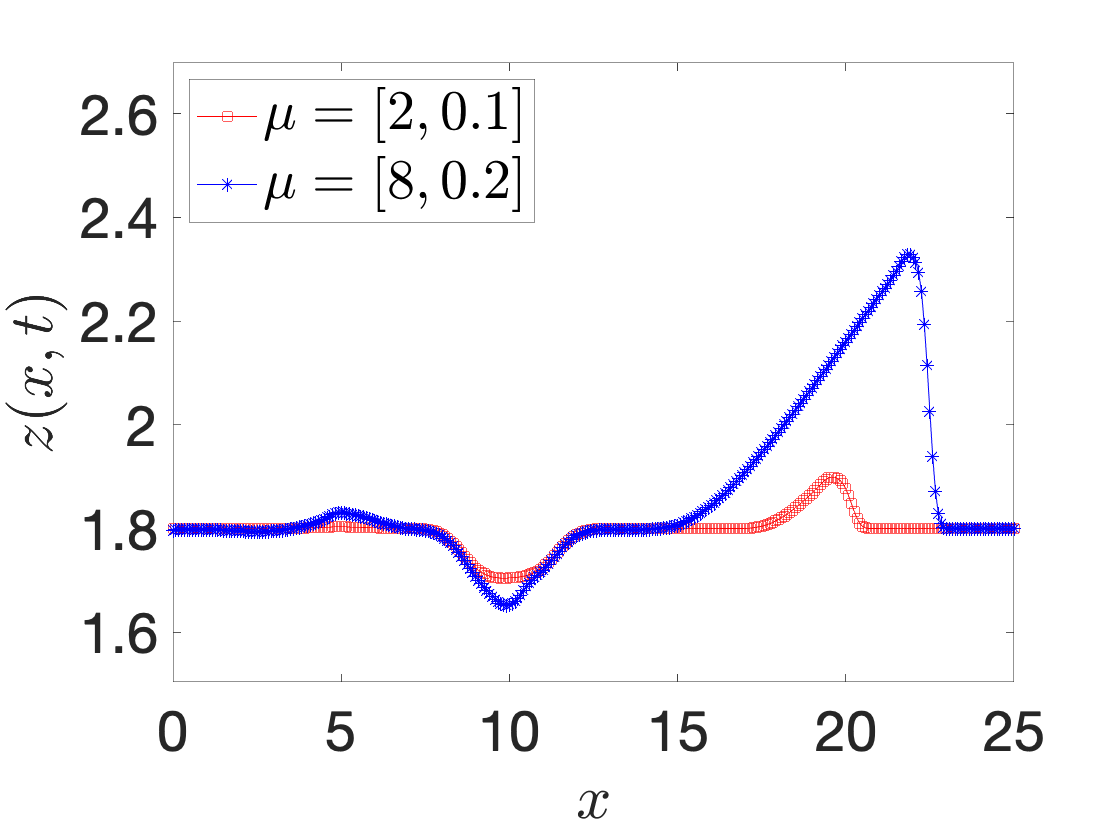}}
 
 \caption{solution to  shallow water equations; behavior of the free surface $z=h +b$  for   $\mu=[2,0.1]$ and 
 $\mu=[8,0.2]$ at  three different time instants.}
 \label{fig:vis_st_venant_2}
  \end{figure}

\section{Data compression (\texttt{RePOD})}
\label{sec:registration}
We denote by $\mathcal{W}_{\rm hf} = {\rm span} \{ \boldsymbol{\varphi}_m^{\rm hf}  \}_{m=1}^{M_{\rm hf}}$   a $M_{\rm hf}$-dimensional space contained in ${\rm Lip}(\Omega; \mathbb{R}^2)$; 
following \cite{taddei2020registration},  given the  identity function $\texttt{id} (\mathbf{X}) = \mathbf{X}$ for all $\mathbf{X} \in \Omega$,  we seek mappings of the form
\begin{equation}
\label{eq:mapping_general}
\boldsymbol{\Phi}_{\mu}(\mathbf{X}) = \texttt{id}  (\mathbf{X}) + 
\boldsymbol{\varphi}_{\mu}(\mathbf{X}),
\quad
\boldsymbol{\varphi}_{\mu} \in \mathcal{W}_{\rm hf},
\;\; \forall \; \mu \in \mathcal{P}.
\end{equation}
We refer to $\boldsymbol{\varphi}$ as to   parameterized displacement, and we denote by 
$\mathcal{C}_{\rm hf}$ a subset of $\mathcal{W}_{\rm hf} $ such that $\boldsymbol{\Phi} = \texttt{id} + 
\boldsymbol{\varphi}$ is bijective from $\Omega$ in itself for all $\boldsymbol{\varphi} \in \mathcal{C}_{\rm hf}$.

We devise a computational procedure that takes as input the snapshots $\{ U^k \}_{k=1}^{n_{\rm train}} \subset \mathcal{M}$ and returns 
(i) the linear operators $Z_N: \mathbb{R}^N \to \mathcal{X}$ and
$W_M: \mathbb{R}^M \to  [ {\rm Lip}(\Omega) ]^2$ in
\eqref{eq:mapping_representation}, and 
(ii) the coefficients $\{  \boldsymbol{\alpha}^k \}_{k=1}^{n_{\rm train}} \subset \mathbb{R}^N$ and  $\{  \mathbf{a}^k \}_{k=1}^{n_{\rm train}} \subset \mathbb{R}^M$ such that
$U^k \approx
\widehat{U}^k \circ (\boldsymbol{\Phi}^k)^{-1}$, with 
$\widehat{U}^k = Z_N \boldsymbol{\alpha}^k$ and 
$\boldsymbol{\Phi}^k = 
\texttt{id} + W_M \mathbf{a}^k$. 
Towards this end, we proceed as follows.

\begin{enumerate}
\item
In section \ref{sec:affine_mappings}, we present guidelines for the definition of the space $\mathcal{W}_{\rm hf}$ and we provide sufficient and computationally-feasible conditions for the bijectivity of $\boldsymbol{\Phi}$.
The discussion exploits results first presented in \cite{taddei2020registration} and is here reported for completeness.
\item
In section \ref{sec:optimization}, given a dictionary of functions --- here referred to as \emph{template space} --- $\mathcal{T}_N \subset L^2(\Omega)$ and the field $\mathfrak{s}^k $, we present a general optimization-based  procedure for the construction of a mapping $\boldsymbol{\Phi}^k$ such that 
$\tilde{\mathfrak{s}}^k =
{\mathfrak{s}}^k  \circ \boldsymbol{\Phi}^k$ is well-approximated by elements in
$\mathcal{T}_N$. 
Our approach  is a   generalization of the optimization-based technique  in \cite{taddei2020registration}; it exploits the theoretical results of section \ref{sec:affine_mappings} to effectively enforce the bijectivity of the mapping.
The field  $\mathfrak{s}^k$ is a suitable function of the $k$-th snapshot, $\mathfrak{s}^k = \mathfrak{s}(U^k)$, that will be introduced below.
\item
In section \ref{sec:generalization}, we present a greedy procedure for the adaptive construction of the template space $\mathcal{T}_N$. Our greedy approach returns the mappings $\{ \boldsymbol{\Phi}^k \}_k$ for all training points and the low-dimensional operator 
$W_M: \mathbb{R}^M \to 
\mathcal{W}_M \subset 
\mathcal{W}_{\rm hf}$ such that 
 $\boldsymbol{\Phi}^k = \texttt{id} + W_M \mathbf{a}^k$, for some $\mathbf{a}^1, \ldots,\mathbf{a}^{n_{\rm train}} \in \mathbb{R}^M$.
 \item
 In section \ref{sec:POD_compression}, we finally apply POD to the mapped snapshots $\{ \widetilde{U}^k = U^k \circ \boldsymbol{\Phi}^k \}_k$ to obtain the reduced operator $Z_N: \mathbb{R}^N \to \mathcal{Z}_N \subset \mathcal{X}$ and the solution coefficients
 $\boldsymbol{\alpha}^1, \ldots,\boldsymbol{\alpha}^{n_{\rm train}} \in \mathbb{R}^N$.
\end{enumerate}

\subsection{Affine mappings}
\label{sec:affine_mappings}
Next Proposition provides the mathematical foundations for the registration algorithm discussed below.

\begin{proposition}
\label{th:application_unit_square}
(\cite[Proposition 2.3]{taddei2020registration})
Given $\Omega = (0,L) \times (0,T)$, consider the mapping 
$\boldsymbol{\Phi} =  \texttt{id}+ \boldsymbol{\varphi}$, where
\begin{equation}
\label{eq:boundary_conditions_mapping}
\left\{
\begin{array}{ll}
\boldsymbol{\varphi}   \cdot \mathbf{e}_1 \, = \, 0 
&
{\rm on} \; \{ \mathbf{X}: \; X_1=0, {\rm or} \, X_1=L\},
\\[3mm]
\boldsymbol{\varphi}   \cdot \mathbf{e}_2 \, = \, 0 
&
{\rm on} \; \{ \mathbf{X}: \; X_2=0, {\rm or} \, X_2=T\}.
\\
\end{array}
\right.
\end{equation}
Then, $\boldsymbol{\Phi}$ is bijective from $\Omega$ into itself if 
\begin{equation}
\label{eq:condition_equivalent}
\min_{\mathbf{X} \in \overline{\Omega}} \,  
g(\mathbf{X}) = 
\,{\rm det} \left(
\widetilde{\nabla} \boldsymbol{\Phi}(\mathbf{X}) \right) \, > \, 0.
\end{equation}
\end{proposition}

It can be shown that mappings satisfying 
\eqref{eq:boundary_conditions_mapping} and 
\eqref{eq:condition_equivalent}
map each edge of the rectangle in itself and  each corner   in itself.
We here enforce condition \eqref{eq:boundary_conditions_mapping} for all elements of the search space  $\mathcal{W}_{\rm hf}$: more precisely, we 
consider $\mathcal{W}_{\rm hf} = {\rm span} \{ \boldsymbol{\varphi}_m^{\rm hf}  \}_{m=1}^{M_{\rm hf}}$
with 
\begin{equation}
\label{eq:tensorized_polynomials}
\left\{
\begin{array}{l}
\boldsymbol{\varphi}_{m=
i + (i'-1)\bar{M}}^{\rm hf}(\mathbf{X})
=
\ell_i  \left( \frac{X_1}{L} \right)     \ell_{i'} \left( \frac{X_2}{T} \right)  
 \, 
\frac{X_1}{L^2}  (L -  X_1 ) \, \mathbf{e}_1 \\[2mm]
\boldsymbol{\varphi}_{m= \bar{M}^2 + 
i + (i'-1)\bar{M}}^{\rm hf}(\mathbf{X})
=
\ell_i  \left( \frac{X_1}{L} \right)     \ell_{i'} \left( \frac{X_2}{T} \right)   \, 
\frac{X_2}{T^2}  (T -  X_2 )  \, \mathbf{e}_2 \\
\end{array}
\right.
\quad
i,i'=1,\ldots, \bar{M},
\end{equation}
where $\{  \ell_i \}_{i=1}^{\bar{M}}$ are the first $\bar{M}$ Legendre polynomials in $(0,1)$ and $M_{\rm hf} = 2 \bar{M}^2$. Clearly,   other choices 
satisfying \eqref{eq:boundary_conditions_mapping}
(e.g., Fourier expansions) might also be considered.
On  the other hand, 
condition \eqref{eq:condition_equivalent} 
is difficult to impose computationally and should be replaced by a computationally feasible surrogate.

We propose to replace
 \eqref{eq:condition_equivalent}  with the 
approximation
\begin{equation}
\label{eq:condition_equivalent_weak}
\int_{\Omega} \, 
{\rm exp} \left( \frac{\epsilon  - g(\mathbf{X})}{C_{\rm exp}} \right) \,  + \, 
{\rm exp} 
\left(  \frac{g(\mathbf{X}) - 1/\epsilon }{C_{\rm exp}} 
\right)
\, dX \leq \delta,
\end{equation}
where $\epsilon \in (0,1)$; we further define the subset $\mathcal{C}_{\rm hf}$ of   $\mathcal{W}_{\rm hf}$ as 
 \begin{equation}
 \label{eq:calC_hf}
 \mathcal{C}_{\rm hf} = \left\{
 \boldsymbol{\varphi} \in \mathcal{W}_{\rm hf} \,:\,
 \boldsymbol{\Phi} =  \texttt{id} + \boldsymbol{\varphi}
 \; {\rm satisfies} \; 
 \eqref{eq:condition_equivalent_weak}
  \right\}.
 \end{equation}
Provided that $\boldsymbol{\varphi}$ is sufficiently smooth, condition \eqref{eq:condition_equivalent_weak} enforces the bijectivity of the mapping: more precisely, 
given  $\epsilon \in (0,1)$ and $C >0$ there exist 
$\delta, C_{\rm exp}> 0$ such that
if $\boldsymbol{\varphi}$ belongs to 
$$
 \mathcal{C}_{\rm hf}   \cap \mathcal{B}_{C,\infty} = \left\{
 \boldsymbol{\varphi} \in \mathcal{C}_{\rm hf} \,:\,
\|  \widetilde{\nabla} g \|_{L^{\infty}(\Omega)}  \leq C   \right\},
$$
 then $\boldsymbol{\Phi}$ is a bijection from $\Omega$ in itself (see \cite[section 2.2]{taddei2020registration}).

The constant  $\epsilon \in (0,1)$ can be interpreted as the maximum allowed pointwise contraction induced by the mapping $\boldsymbol{\Phi}$ and by its inverse. 
On the other hand, we observe that   the constant 
$\delta$ should satisfy
\begin{equation}
\label{eq:delta_condition}
\delta \geq  |\Omega| \; \left(
{\rm exp} \left( \frac{\epsilon  - 1}{C_{\rm exp}} \right) \,  + \, 
{\rm exp} \left( \frac{1 - 1/\epsilon }{C_{\rm exp}} \right) \right),
\end{equation}
so that $\boldsymbol{\varphi}= \mathbf{0}$ is admissible.
In all our numerical examples, we choose 
\begin{equation}
\label{eq:choice_Cdelta}
\epsilon=0.1, \quad 
C_{\rm exp} = 0.025 \epsilon, \quad
\delta = |\Omega|.
\end{equation}


\subsection{Registration}

\subsubsection{Optimization-based registration}
\label{sec:optimization}

We first present the registration procedure for a single field. 
Given the set $ \mathcal{C}_{\rm hf} $ in \eqref{eq:calC_hf}, we introduce the template space
$\mathcal{T}_N = {\rm span} \{  \psi_n  \}_{n=1}^N \subset L^2(\Omega)$,
the $M$-dimensional space
$\mathcal{W}_M \subset \mathcal{W}_{\rm hf}$, and 
 the target snapshot $U^k = U_{\mu^k}$. Then, we propose to build $\boldsymbol{\Phi}^k = \texttt{id} +  \boldsymbol{\varphi}^k$ as the solution to
\begin{subequations} 
\label{eq:optimization_statement}
\begin{equation}
\label{eq:optimization_statement_a}
\min_{\psi \in   {\mathcal{T}}_N, \; 
\boldsymbol{\varphi} \in \mathcal{W}_M}
\,
\mathfrak{f} \left(\psi,  \texttt{id}  + \boldsymbol{\varphi} ,   \mu^k \right) \,  + \, 
\xi \big|   
\boldsymbol{\varphi}  \big|_{H^2(\Omega)}^2,  
\quad
{\rm s.t.} \;
\boldsymbol{\varphi} \in \mathcal{C}_{\rm hf}.
\end{equation}
The functional {$ \mathfrak{f} :  L^2(\Omega) \times [{\rm Lip}(\Omega)]^2 \times \mathcal{P} \to \mathbb{R}_+$  } --- which is here referred to as \emph{proximity measure} --- is given by
\begin{equation}
\label{eq:L2obj}
\mathfrak{f}\left(
\psi,  \boldsymbol{\Phi},  \mu \right)
:= \; \int_{\Omega } \;  
\left( \mathfrak{s} (U_{\mu}) \circ \boldsymbol{\Phi} -  
\psi    \right)^2 \, d \mathbf{X};
\end{equation} 
\end{subequations}
Here, $
\mathfrak{s} : \mathcal{X} \to 
L^2(\Omega)$ is a suitable \emph{registration sensor} that will be introduced below.
On the other hand, the $H^2$ seminorm is given by
$|  \mathbf{v}   |_{H^2(\Omega)}^2 
:= \sum_{i,j,k=1}^d \, \int_{\Omega} \, 
\left(   \widehat{\partial}_{i,j}^2 v_k  \,\right)^2 \, dX$ for all $\mathbf{v} \in 
H^2(\Omega; \mathbb{R}^d) $. 
The constraint 
$\boldsymbol{\varphi} \in \mathcal{C}_{\rm hf}$
in \eqref{eq:optimization_statement_a}, which was introduced in 
\eqref{eq:condition_equivalent_weak}, weakly enforces that 
$g \in [\epsilon, 1/\epsilon]$ and thus  that $\boldsymbol{\Phi}$ is a bijection from $\Omega$ into itself for all admissible solutions to \eqref{eq:optimization_statement_a}.

We observe that, compared to \cite{taddei2020registration}, we here optimize with respect to both displacement, $\boldsymbol{\varphi}$, and template $\psi$. 
Rather than minimizing  the distance from a template field $\bar{\psi}$ in the mapped configuration, we here minimize the best-fit error from a given linear space:
in our experience, the statement in \eqref{eq:optimization_statement_a} outperforms the one in \cite{taddei2020registration} for  fields
with several local extrema
 for which a one-dimensional template might not suffice.
 Note that the optimal solution 
$(\psi^k, \boldsymbol{\varphi}^k)$ to  \eqref{eq:optimization_statement} satisfies
$\psi^k = \Pi_{\mathcal{T}_N}  ( \mathfrak{s}(U_{\mu^k}) \circ     
\boldsymbol{\Phi}^k
)$: for moderate values of $N$, we empirically find that the cost of solving \eqref{eq:optimization_statement} is comparable to the cost of solving the statement in \cite{taddei2020registration}. 

Since $| v  |_{H^2(\Omega)}=0$ for all linear polynomials, we find that 
the penalty term measures deviations from linear maps; in particular, we find that mapping and displacement have the same $H^2$ seminorm, that is $\big| \texttt{id}  + \boldsymbol{\varphi}  \big|_{H^2(\Omega)}
=
\big|   \boldsymbol{\varphi}  \big|_{H^2(\Omega)}
$. 
Due to the condition
$\boldsymbol{\Phi}(\Omega) = \Omega$, we might further interpret the penalty as a measure of the deviations from the identity map.
The penalty in \eqref{eq:optimization_statement_a}  
can be further interpreted as a Tikhonov regularization, and 
has   the effect to control the gradient of the Jacobian determinant $g$
---
recalling the discussion in section \ref{sec:affine_mappings}, the latter is important to enforce  bijectivity. 
The hyper-parameter $\xi$ balances accuracy
 --- measured by $\mathfrak{f}$ --- and smoothness of the mapping. 
 If we write $\boldsymbol{\varphi} = \sum_{m=1}^{M_{\rm hf}} a_m  \boldsymbol{\varphi}_m^{\rm hf}$, we obtain that 
 $|  \boldsymbol{\Phi}   |_{H^2(\Omega)} =  \mathbf{a}^T \, \mathbf{A}^{\rm reg} \, \mathbf{a}$
with  $\mathbf{A}^{\rm reg} \in \mathbb{R}^{M_{\rm hf}, M_{\rm hf}}$
such that
$A_{m,m'}^{\rm reg} = ((\boldsymbol{\varphi}_{m'}^{\rm hf},      
\boldsymbol{\varphi}_{m}^{\rm hf} ))_{H^2(\Omega)}$ for $m,m'=1,\ldots, M_{\rm hf}$ 
---
$  ((\cdot,       \cdot ))_{H^2(\Omega)}$ is the bilinear form associated with 
$| \cdot |_{H^2(\Omega)}$.
Since $\boldsymbol{\varphi}_1^{\rm hf},\ldots, \boldsymbol{\varphi}_{M_{\rm hf}}^{\rm hf}$ are polynomials, computation of the entries of 
$\mathbf{A}^{\rm reg}$ is straightforward.
Finally, since the registration problem is non-convex in $\boldsymbol{\varphi}$, careful initialization of the iterative optimization algorithm is important: we refer to \cite[section 3.1.2]{taddei2020registration} for further details.

\begin{remark}
\label{remark:filtering}
\textbf{Choice of the registration sensor 
$\mathfrak{s}$.}
As for   shock capturing methods (e.g., \cite{persson2006sub}), 
the choice of the sensor 
$\mathfrak{s}$ is important to correctly capture relevant features associated with the solution field. In this work,  we consider $\mathfrak{s}(U) = U$ for the Burgers equation  and 
$\mathfrak{s}(U) = h$ for the shallow water equations:
{
a general strategy for the choice of the registration sensor is beyond the scope of the present paper.}
 For nearly discontinuous fields, we empirically found that filtering might improve the robustness of the registration procedure. In this work, we 
resort to a spatial moving average filter based on the Matlab routine \texttt{smooth}.
\end{remark}

\subsubsection{Parametric registration}
\label{sec:generalization}

In Algorithm \ref{alg:registration}, we propose a Greedy procedure to iteratively build a low-dimensional approximation space 
${\mathcal{W}}_M \subset {\mathcal{W}}_{\rm hf}$
for the displacement field and the template space ${\mathcal{T}}_N$. Here, the routine 
$$
[  \boldsymbol{\varphi}^{\star}, \psi^{\star}, \mathfrak{f}_{N,M}^{\star}   ]
=
\texttt{registration} \left(
U , {\mathcal{T}}_N, 
{\mathcal{W}}_M, C_{\rm exp}, \delta, \epsilon, \xi 
\right)
$$ 
takes as input 
{the target snapshot $U$, 
the template space ${\mathcal{T}}_N$,
the displacement space ${\mathcal{W}}_M$,   the constants 
$C_{\rm exp}, \delta, \epsilon>0$ (cf. \eqref{eq:condition_equivalent_weak}), and the weighting parameter  
$\xi>0$ (cf.  \eqref{eq:optimization_statement_a})},
and returns a local minimum $(\boldsymbol{\varphi}^{\star}, \psi^{\star})$ to \eqref{eq:optimization_statement_a} and the corresponding value of the proximity measure
$\mathfrak{f}_{N,M}^{\star} :=\|  \psi^{\star} \circ (\texttt{id} + \boldsymbol{\varphi}^{\star})  - U    \|_{L^2(\Omega)}^2$. On the other hand, the routine 
$$
[  \mathcal{W}_M,  \{  \mathbf{a}^k \}_{k=1}^{n_{\rm train}}
  ]  
=
\texttt{POD} 
\left( 
\{ \boldsymbol{\varphi}^{\star,k} \}_{k=1}^{n_{\rm train}}, 
tol_{\rm pod} ,
\; \;
(\cdot, \cdot)_{\star}
\right)
$$
takes as input the optimal displacement fields obtained by repeatedly solving \eqref{eq:optimization_statement_a} for different target snapshots,
the tolerance $tol_{\rm pod}  > 0$, and the inner product 
$(\cdot, \cdot)_{\star}$,
 and returns 
 the   POD space associated with the  
 first $M$   modes
$\mathcal{W}_M = {\rm span} \{ \boldsymbol{\varphi}_m \}_{m=1}^M$, 
where $M$ is chosen according to \eqref{eq:POD_cardinality_selection}, and
the vectors of coefficients
$ \{  \mathbf{a}^k \}_k$  such that
$\left( \mathbf{a}^k  \right)_m = 
( \boldsymbol{\varphi}_m,  \boldsymbol{\varphi}^{\star,k} )_{\star}$.
 As in \cite{taddei2020registration}, 
we consider the  inner product
\begin{equation}
\label{eq:mapping_inner_product}
(  \boldsymbol{\phi}',  \boldsymbol{\phi})_{\star}
=
\mathbf{a}' \cdot \mathbf{a},
\quad
{\rm for \; any} \;
\boldsymbol{\phi} , \boldsymbol{\phi}'
\;\;
{\rm s.t.} 
\;\;
\boldsymbol{\phi} = \sum_{m=1}^{M_{\rm hf}}
\left(  \mathbf{a} \right)_m 
\boldsymbol{\varphi}_m^{\rm hf},
 \;\;
 \boldsymbol{\phi}' = \sum_{m=1}^{M_{\rm hf}}
\left(  \mathbf{a}' \right)_m 
\boldsymbol{\varphi}_m^{\rm hf}.
\end{equation}

If $\mathcal{W}_{\rm hf} = \emptyset$, Algorithm \ref{alg:registration} reduces to the well-known strong-Greedy algorithm (see, e.g., \cite{binev2011convergence})
for the manifold
$\mathcal{M}_{\rm s} = \{
\mathfrak{s}(U_{\mu}): \mu \in \mathcal{P} \}$. In our experience, for moderate values of $N_{\rm max}$, the offline cost is dominated by the cost of performing the first iteration: POD indeed effectively leads to an approximation space $\mathcal{W}_M$ of size $M \ll M_{\rm hf}$ and ultimately simplifies the solution to the optimization problem for the subsequent iterations.

\begin{algorithm}[H]                      
\caption{Registration algorithm}     
\label{alg:registration}     

 \small
\begin{flushleft}
\emph{Inputs:}  $\{ (\mu^k,  U_{\mu^k}) \}_{k=1}^{n_{\rm train}} \subset  \mathcal{P} \times \mathcal{M}$ snapshot set, 
$\mathcal{T}_{N_0} = {\rm span} \{ \psi_n \}_{n=1}^{N_0}$ template space, 
$\mathfrak{s}: \mathcal{X} \to L^2(\Omega)$ registration sensor;
\smallskip

\emph{Hyper-parameters:}
$tol_{\rm pod}$ (cf. \eqref{eq:POD_cardinality_selection}),
$C_{\rm exp}, \delta, \epsilon$ (cf. \eqref{eq:condition_equivalent_weak}), 
$\xi$ (cf.  \eqref{eq:optimization_statement_a}), 
$N_{\rm max}$ maximum number of iterations, $\texttt{tol}$ tolerance for termination condition,
$(\cdot,\cdot)_{\star}$
mapping inner  product (cf. \eqref{eq:mapping_inner_product}).
\smallskip

\emph{Outputs:} 
${\mathcal{T}}_N = {\rm span} \{ \psi_n  \}_{n=1}^N$ template space, 
$\mathcal{W}_M = {\rm span} \{ \boldsymbol{\varphi}_m  \}_{m=1}^M$ displacement space,
$\{  \mathbf{a}^k  \}_k$ mapping coefficients.
\end{flushleft}                      

 \normalsize 

\begin{algorithmic}[1]
\State
Set
$\mathcal{T}_{N=N_0} = \mathcal{T}_{N_0}$,
$\mathcal{W}_M =\mathcal{W}_{\rm hf}$.
\vspace{3pt}

\For {$N=N_0, \ldots, N_{\rm max}-1$ }

\State
$
[  \boldsymbol{\varphi}^{\star,k}, \psi^{\star,k}, \mathfrak{f}_{N,M}^{\star,k}   ]
=
\texttt{registration} \left(
U^k ,  {\mathcal{T}}_N, 
{\mathcal{W}}_M, C_{\rm exp}, \delta, \epsilon , \xi 
\right)
$ for $k=1,\ldots,n_{\rm train}$.
\vspace{3pt}

\State
$[  \mathcal{W}_M , 
\; 
\{  \mathbf{a}^k  \}_k ]  =
\texttt{POD} \left( \{ \boldsymbol{\varphi}^{\star,k} \}_{k=1}^{n_{\rm train}}, 
tol_{\rm pod}  ,   (\cdot, \cdot)_{\star}\right)$
\vspace{3pt}

\If{   $\max_k   \mathfrak{f}_{N,M}^{\star,k}   < \texttt{tol}$}, \texttt{break}

\Else

\State
$\mathcal{T}_{N+1} = \mathcal{T}_{N} \cup {\rm span} \{ 
\mathfrak{s}\left( U_{\mu^{k^{\star}}}  \right) \circ  \boldsymbol{\Phi}^{\star,k^{\star}}  \}$ with 
$k^{\star} = {\rm arg} \max_k \mathfrak{f}_{N,M}^{\star,k}$.
\EndIf

\EndFor

\end{algorithmic}

\end{algorithm}

\begin{remark}
\label{remark:template_space}
\textbf{Choice of the initial templace space
$\mathcal{T}_{N=N_0}$.}
Algorithm \ref{alg:registration} 
requires the definition of the initial template space $\mathcal{T}_{N=N_0}$:
in this work, we use $
\mathcal{T}_{N_0=1} = {\rm span} \{ 
U_{\mu = \bar{\mu}}  \} $ for the Burgers equation, and we consider
$\mathcal{T}_{N_0=2} = {\rm span} \{ 
h_{\mu = \bar{\mu}},  h_{\rm bf}    \} $ for the shallow-water model problem, 
where 
$\bar{\mu}$ denotes the centroid of $\mathcal{P}$, and
$h_{\rm bf} = ( U_{\rm bf} )_1$ denotes the initial height.
For the Burgers equation, we empirically found that our Greedy procedure weakly depends  on the choice of  the initial template $\mathcal{T}_{N=1}$.
On the other hand, for the shallow-water problem, the use of a two-dimensional template is important  for accuracy:
{we can significantly reduce the impact  of the choice of the initial template space by not performing POD (cf. Step 4 Algorithm \ref{alg:registration}) during the first few iterations, at the price of (significantly) higher offline costs.
}
Nevertheless, if compared to  \cite{taddei2020registration}, our empirical findings suggest that the Greedy procedure significantly reduces the sensitivity of the algorithm  with respect to the initial choice of the template compared to the the approach, which represented an issue of the original proposal (cf. \cite[Fig. 7]{taddei2020registration}).
\end{remark}

\subsection{POD compression}
\label{sec:POD_compression}

 Given the mappings $\{ \boldsymbol{\Phi}^k = \texttt{id} +  \boldsymbol{\varphi}^k  \}_{k=1}^{n_{\rm train}}$, we define the mapped snapshots
 $\widetilde{U}^k := U^k \circ    \boldsymbol{\Phi}^k $ for $k=1,\ldots, n_{\rm train}$.  Then, we apply POD based on the $L^2$ inner product to generate the space $\mathcal{Z}_N = {\rm span} \{ \zeta_n \}_{n=1}^N$ and the coefficients $\{ \boldsymbol{\alpha}^k \}_{k=1}^{n_{\rm train}}$ such that
$(\boldsymbol{\alpha}^k)_n = (\zeta_n, \widetilde{U}^k)$ for $n=1,\ldots,N$ and $k=1,\ldots,n_{\rm train}$.
The dimension $N$  is chosen according to \eqref{eq:POD_cardinality_selection}.
Note that application of POD requires the definition of all mapped snapshots  on a parameter-independent spatio-temporal mesh.

  {Algorithm \ref{alg:repod} summarizes the computational procedure. 
 }  We   observe that, although our   data compression  procedure  is applied to a specific class of problems, hyperbolic conservation laws with parameter-dependent discontinuities,  our approach is general, that is, independent of the underlying mathematical model. 
 
 \begin{algorithm}[H]                      
\caption{ {Data compression (\texttt{RePOD}) } }     
\label{alg:repod}     

 \small
\begin{flushleft}
\emph{Inputs:}  
see  Algorithm \ref{alg:registration}
\smallskip

\emph{Hyper-parameters:}
$(\cdot, \cdot)$ solution inner product, 
and parameters of Algorithm \ref{alg:registration}
\smallskip

\emph{Outputs:} 
$\mathcal{Z}_N = {\rm span} \{ \zeta_n  \}_{n=1}^N$ reduced space, 
$\mathcal{W}_M = {\rm span} \{ \boldsymbol{\varphi}_m  \}_{m=1}^M$ displacement space,
$\{  \mathbf{a}^k  \}_k$ mapping coefficients,
$\{  \boldsymbol{\alpha}^k  \}_k$
solution coefficients.
\end{flushleft}                      

 \normalsize 

\begin{algorithmic}[1]
\State
{ Apply Algorithm \ref{alg:registration} to obtain
 $\mathcal{W}_M$ and $\{ \mathbf{a}^k  \}_k$}
\vspace{3pt}
 
\State
{Define the mapped snapshots 
$ \widetilde{U}^k := U^k \circ    \boldsymbol{\Phi}^k$ with  
$\boldsymbol{\Phi}^k = \texttt{id} +W_M\mathbf{a}^k$,
$k=1, \ldots, n_{\rm train}$; }
\vspace{3pt}

 \State
 {$[  \mathcal{Z}_N , 
\; 
\{  \boldsymbol{\alpha}^k  \}_k ]  =
\texttt{POD} \left( \{ \widetilde{U}^k  \}_k, \; 
tol_{\rm pod} , \; (\cdot, \cdot) \right)$
.}
\end{algorithmic}

\end{algorithm}

 \section{Projection-based  reduced-order model}
\label{sec:projection_based_ROM}
In section \ref{sec:registration}, we discussed how to generate the operators $Z_N, W_M$ associated with 
\eqref{eq:mapping_representation} based on the snapshots $\{  U^k = U_{\mu^k} \}_{k}$, and how to compute (quasi-)optimal values of  the solution/mapping coefficients  for all training points,  $ \{ \boldsymbol{\alpha}^k \}_k$,  $ \{ \mathbf{a}^k \}_k$. In this section, we address the problem of predicting solution and mapping coefficients for a new value of the parameter $\mu \in \mathcal{P}$: as anticipated in the introduction, we resort to a kernel-regression algorithm  to predict the mapping coefficients, while we resort to minimum residual projection to predict the   coefficients associated with the estimate $\widehat{U}$ of the mapped solution.
To clarify the presentation, we here focus on the main features of the formulation and we defer to the  appendix for a thorough discussion of several technical aspects.

\subsection{Non-intrusive construction of the mapping $\boldsymbol{\Phi}$}
\label{sec:non_intrusive_ROM}

Algorithm \ref{alg:registration} returns the space $\mathcal{W}_M = {\rm span} \{ \boldsymbol{\varphi}_m \}_{m=1}^M$ and the mapping coefficients  $\{ \mathbf{a}^k \}_{k=1}^{n_{\rm train}}$. 
As in \cite{taddei2020registration}, we apply a multi-target regression procedure based on radial basis function (RBF, \cite{wendland2004scattered}) approximation to compute predictors of the mapping coefficients for all $\mu \in \mathcal{P}$:
\begin{equation}
\label{eq:mapping_phi_RBF}
\widehat{\boldsymbol{\Phi}}_{\mu}
=\texttt{id} + \sum_{m=1}^M \; \left( \widehat{\mathbf{a}}_{\mu}   \right)_m \; \boldsymbol{\varphi}_m,
\quad
\widehat{\mathbf{a}} : \mathcal{P} \to \mathbb{R}^M.
\end{equation}
Each entry of $\widehat{\mathbf{a}} $ is built separately based on the datasets $\mathcal{D}_m = \{ (\mu^k,  
a_m^k) \}_{k=1}^{n_{\rm train}}$, $m=1,\ldots,M$,
$a_m^k := \left(   \mathbf{a}^k  \right)_m$. To assess the goodness of fit of the regression model, we compute an estimate of the out-of-sample R-squared (e.g., \cite[Chapter 14]{rice2006mathematical}) using cross-validation. We recall  that given training and test sets
$\{ (\mu^k,  a_m^{k}  )  \}_{k=1}^{n_{\rm train}}$ and
$\{ (\mu^j,  a_m^{j}  )  \}_{j=1}^{n_{\rm test}}$, the R-squared is defined as
\begin{equation}
\label{eq:Rsquared}
\texttt{R}_m^2 =  
1 -
\frac{
\sum_{j=1}^{n_{\rm test}}
\left(  a_m^{j} -   \left( \widehat{\mathbf{a}}_{\mu^j}   \right)_m   \right)^2}{
\sum_{j=1}^{n_{\rm test}}
\left(  a_m^{j} -   \bar{a}_m^{\rm train}
\right)^2
}
,
\quad
\bar{a}_m^{\rm train} = \frac{1}{n_{\rm train}} \sum_{k=1}^{n_{\rm train}} \,  a_m^k.
\end{equation}
To reduce the risk of over-fitting, we only keep coefficients for which 
 $\texttt{R}_m^2 \geq \texttt{R}_{\rm min}^2 = 0.75$.  We remark that the mapping in \eqref{eq:mapping_phi_RBF} is not guaranteed to be bijective  for all $\mu \in \mathcal{P}$, particularly for small-to-moderate values of $n_{\rm train}$: as discussed in \cite{taddei2020registration}, this represents a major issue of the proposed approach and is the motivation to consider intrusive methods to simultaneously learn mapping and solution coefficients.
 {The development of a fully intrusive ROM is beyond the scope of the present paper and is the subject of ongoing research.}

\subsection{Projection-based ROM for the solution coefficients }
\label{sec:pMOR_minres}
\subsubsection{Reduced-order statement: Galerkin projection; (approximate) minimum residual}

We introduce the matrix representation 
$\mathbf{Z}_N = [\boldsymbol{\zeta}_1,\ldots,\boldsymbol{\zeta}_N ]$ 
of the operator $Z_N: \mathbb{R}^N \to \mathcal{Z}_N$
associated with the DG FE basis $\{  \varphi_j \}_{j=1}^{N_{\rm hf}}$,
 such that
$\mathbf{Z}_N^T \mathbf{X}_{\rm hf} \mathbf{Z}_N = \mathbbm{1}_N$; we further introduce the dual residual operator
$\mathbf{R}_N^{\rm hf}: \mathbb{R}^N \times \mathcal{P} \to \mathbb{R}^{N_{\rm hf}}$ such that
\begin{subequations}
\begin{equation}
\label{eq:RN_algebraic}
 \left(   \mathbf{R}_N^{\rm hf}(\boldsymbol{\alpha},  \mu)   \right)_j 
 =
 R_{\Phi_{\mu}} (Z_N \boldsymbol{\alpha}, \varphi_j),
 \quad
 j=1,\ldots,N_{\rm hf},
 \end{equation} 
and the Jacobian
$\boldsymbol{\mathcal{J}}_{N}^{\rm hf}: \mathbb{R}^N \times \mathcal{P} \to \mathbb{R}^{N_{\rm hf}, N}$ such that
\begin{equation}
\label{eq:JacN_algebraic}
\boldsymbol{\mathcal{J}}_{N}^{\rm hf}(\boldsymbol{\alpha},  \mu) 
:=
\boldsymbol{\mathcal{J}}^{\rm hf}(Z_N\boldsymbol{\alpha},  \mu) \mathbf{Z}_N,
\end{equation} 
where $\boldsymbol{\mathcal{J}}^{\rm hf}(U,  \mu) \in \mathbb{R}^{N_{\rm hf}, N_{\rm hf}}$ is the high-fidelity Jacobian associated with a given field $U \in \mathcal{X}_{\rm hf}$ and the parameter $\mu$,
\begin{equation}
\left( \mathcal{J}^{\rm hf}(U,  \mu) \right)_{i,j}
:=
D  R_{\Phi_{\mu}}  [U] (\varphi_j, \varphi_i)
=
\lim_{\epsilon \to 0} \; \frac{   
 R_{\Phi_{\mu}} (U + \epsilon \varphi_j, \varphi_i) - 
 R_{\Phi_{\mu}} (U , \varphi_i)
}{\epsilon},
\quad
i,j=1,\ldots,N_{\rm hf},
\end{equation}
and $D  R_{\Phi_{\mu}}[ U] : \mathcal{X}_{\rm hf} \times   \mathcal{X}_{\rm hf}  \to \mathbb{R}$ is the Fréchet derivative of $R_{\Phi_{\mu}}$ at $U$.  We present below several projection-based statements: to clarify the approaches we report both the variational and the algebraic formulations.
 \end{subequations}

We have now the elements to introduce the Galerkin ROM: 
find $\widehat{U}_{\mu} = Z_N \widehat{\boldsymbol{\alpha}}_{\mu} \in \mathcal{Z}_N$ such that
\begin{equation}
\label{eq:galerkinROM}
R_{\Phi_{\mu}}( \widehat{U}_{\mu}, \zeta  ) = 0 \quad
\forall \; \zeta \in \mathcal{Z}_N  \quad \Leftrightarrow  \quad 
\mathbf{Z}_N^T  \; \mathbf{R}_N^{\rm hf} \left( \widehat{\boldsymbol{\alpha}}_{\mu},  \mu \right)
=
\mathbf{0}.
\end{equation}  
 Similarly, we introduce the minimum residual ROM:
 \begin{equation}
\label{eq:minresROM}
\widehat{U}_{\mu} \in {\rm arg} \min_{U \in \mathcal{Z}_N }
\sup_{\eta \in \mathcal{Y}_{\rm hf}}
\frac{R_{\Phi_{\mu}}^{\rm hf}(U, \eta) }{ \vertiii{\eta} }
  \quad \Leftrightarrow  \quad 
  \widehat{\boldsymbol{\alpha}}_{\mu} \in
   {\rm arg} \min_{  \boldsymbol{\alpha}  \in \mathbb{R}^N } 
   \;
   \|  \mathbf{R}_N^{\rm hf} \left(  \boldsymbol{\alpha},  \mu \right)  \|_{ 
   \mathbf{Y}_{\rm hf}^{-1} },
\end{equation}  
 where   $\|  \mathbf{w}  \|_{\mathbf{Y}_{\rm hf}^{-1} }^2 = \mathbf{w}^T \mathbf{Y}_{\rm hf}^{-1} \mathbf{w}$. In the numerical results, we demonstrate the superiority of the minimum residual ROM \eqref{eq:minresROM}  compared to the Galerkin ROM \eqref{eq:galerkinROM}.

 In order to devise an online-efficient ROM, we first introduce the approximate    minimum residual statement:
\begin{subequations}
\label{eq:approx_minresROM_temp}
\begin{equation}
\widehat{U}_{\mu} = 
Z_N \widehat{\boldsymbol{\alpha}}_{\mu}
\in {\rm arg} \min_{U \in \mathcal{Z}_N }
\sup_{\eta \in \mathcal{Y}_{J}}
\frac{R_{\Phi_{\mu}} (U, \eta) }{ \vertiii{\eta} },
\end{equation}
where 
  $\mathcal{Y}_{J} = {\rm span} \{  \eta_j \}_{j=1}^J \subset \mathcal{Y}_{\rm hf}$ is referred to as \emph{empirical test space}. If $(\eta_j, \eta_i)_{ \mathcal{Y}_{\rm hf}   } = \delta_{i,j}$, it is possible to verify that the solution coefficients
  $\widehat{\boldsymbol{\alpha}}_{\mu}$ satisfy
\begin{equation}
 \widehat{\boldsymbol{\alpha}}_{\mu} \in
   {\rm arg} \min_{  \boldsymbol{\alpha}  \in \mathbb{R}^N } 
   \;
   \|  \mathbf{R}_{N,J}^{\rm hf} \left(  \boldsymbol{\alpha},  \mu \right)  \|_2,
   \quad
   {\rm with}
   \;
   \left(
   \mathbf{R}_{N,J}^{\rm hf} \left(  \boldsymbol{\alpha},  \mu \right)
      \right)_j
  =
  R_{\Phi_{\mu}} (Z_N\boldsymbol{\alpha} , \eta_j ),
  \;\;
  j=1,\ldots,J.
\end{equation}
\end{subequations}
  Then, following \cite{yano2019discontinuous}, we replace the truth residual 
in \eqref{eq:approx_minresROM_temp}  
  with the EQ residual 
\begin{subequations}
\label{eq:approx_minresROM}
\begin{equation}
R_{\Phi_{\mu}}^{\rm eq}(\zeta, \eta)
=
\sum_{k \in \mathcal{I}_{\rm eq}} \rho_k^{\rm eq} \;
\left(
r_k^{\rm c, \mu}(\zeta, \eta) + r_k^{\rm d}(\zeta, \eta)
\right)
\end{equation}
where 
$ \mathcal{I}_{\rm eq} \subset \{ 1,\ldots,N_{\rm e} \}$ is a subset of the mesh elements and
$\rho_1, \ldots, \rho_{N_{\rm e}} \geq 0$ are a set of non-negative weights. In conclusion, we obtain the hyper-reduced approximate minimum residual ROM:
\begin{equation}
\widehat{U}_{\mu} \in {\rm arg} \min_{U \in \mathcal{Z}_N }
\sup_{\eta \in \mathcal{Y}_{J}}
\frac{R_{\Phi_{\mu}}^{\rm eq} (U, \eta) }{ \vertiii{\eta} } 
\quad \Leftrightarrow  \quad 
 \widehat{\boldsymbol{\alpha}}_{\mu} \in
   {\rm arg} \min_{  \boldsymbol{\alpha}  \in \mathbb{R}^N } 
   \;
   \|  \mathbf{R}_{N,J}^{\rm eq} \left(  \boldsymbol{\alpha},  \mu \right)  \|_2,
  \;\;
  j=1,\ldots,J.
\end{equation}
with 
$\mathbf{R}_{N,J}^{\rm eq}: \mathbb{R}^N \times \mathcal{P} \to \mathbb{R}^J$, 
$\left(
   \mathbf{R}_{N,J}^{\rm eq} \left(  \boldsymbol{\alpha},  \mu \right)
      \right)_j
  =
  R_{\Phi_{\mu}}^{\rm eq} (Z_N\boldsymbol{\alpha} , \eta_j )$.
 In the next two sections, we  discuss how to construct the test space $\mathcal{Y}_{J}$ and how to compute the empirical quadrature rule.
\end{subequations}
  
  \begin{remark}
  \label{remark:online_efficiency}
  \textbf{Online efficiency.}
  Computation of $\mathbf{R}_{N,J}^{\rm eq}$
  and its Jacobian $\mathbf{J}_{N,J}^{\rm eq}$ can 
  be performed efficiently, provided that $| \mathcal{I}_{\rm eq}| \ll N_{\rm e}$;
furthermore, since \eqref{eq:approx_minresROM} is a nonlinear least-squares problem, we can resort to the Gauss-Newton method  to efficiently compute the solution.
More in detail, computation of 
$\{ r_k^{\rm c, \mu}(Z_N \boldsymbol{\alpha}, \eta_j)
+r_k^{\rm d, \mu}(Z_N \boldsymbol{\alpha}, \eta_j)
  \}_j$ for a given $k \in \mathcal{I}_{\rm eq}$ requires the storage of 
  $\zeta_1,\ldots,\zeta_N, \eta_1,\ldots,\eta_J$ in the $k$-th element and in its neighbors. Note    
  that if $\mathcal{Z}_N, \mathcal{Y}_J \subset C(\Omega)$, computation of 
  $r_k^{\rm c, \mu}, r_k^{\rm d, \mu}$ only depends on the value of trial and test functions in the 
  $k$-th element (see Appendix \ref{sec:implementation}): continuous approximations of trial and test spaces thus allow quite significant reductions in online memory and computational costs.
 In the numerical results, we investigate the accuracy of  continuous approximations for the two model problems. 
The continuous  trial space
$\mathcal{Z}_N$ is obtained by applying  POD to  continuous approximations of the mapped snapshots $\{ \widetilde{U}^k \}_k$:
in our implementation, the continuous approximation is computed by simply averaging over facets. 
 \end{remark}

\subsubsection{Construction of the empirical test space}
It is possible to verify that the solution $\widehat{U}_{\mu}$ to \eqref{eq:minresROM} satisfies
\begin{equation}
\label{eq:optimal_test_space}
R_{\Phi_{\mu}} (\widehat{U}_{\mu}, \eta) = 0
\;\;
\forall \, \eta \in \mathcal{Y}_{N,\mu}^{\rm opt}
=
{\rm span} \{   
 \texttt{R}_{\mathcal{Y}_{\rm hf}} \left(
 D  R_{\Phi_{\mu}}  [\widehat{U}_{\mu}] (\zeta_n, \cdot)
 \right)
 \}_{n=1}^N.
\end{equation}
Note that the algebraic representation $\mathbf{Y}_{N,\mu}^{\rm opt}$ of the  space $\mathcal{Y}_{N,\mu}^{\rm opt}$ satisfies
$
\mathbf{Y}_{N,\mu}^{\rm opt} = 
\mathbf{Y}_{\rm hf}^{-1}
\boldsymbol{\mathcal{J}}^{\rm hf}(\widehat{U}_{\mu},  \mu) \mathbf{Z}_N$.
For this reason, we propose to choose the test space $\mathcal{Y}_J$ to approximate the manifold
$\mathcal{M}_{\rm test,N} = \bigcup_{\mu \in \mathcal{P}} \;
\mathcal{Y}_{N,\mu}^{\rm opt}$.

In Appendix \ref{sec:AMR_linear_theory}, we rigorously justify our choice by  presenting a detailed analysis for the linear case; in Algorithm \ref{alg:test_space}, we summarize the computational procedure 
for the construction of the test space employed in our code. We remark that problem-adapted test spaces have been  first considered in 
\cite{dahmen2014double} for linear problems, and more recently in  \cite{collins2020petrov}: a thorough comparison with 
\cite{collins2020petrov,dahmen2014double} is beyond the scope of the present paper.
  
\begin{algorithm}[H]                      
\caption{Construction of the empirical test space}     
\label{alg:test_space}     

 \small
\begin{flushleft}
\emph{Inputs:}  $\{ (\mu^k,  U_{\mu^k}) \}_{k=1}^{n_{\rm train}} \subset  \mathcal{P} \times \mathcal{M}$ snapshot set, 
$\mathcal{Z}_{N} = {\rm span} \{ \zeta_n \}_{n=1}^{N}$ trial  space;
\smallskip

\emph{Hyper-parameters:}
$tol_{\rm pod}$ (cf. \eqref{eq:POD_cardinality_selection}).
\smallskip

\emph{Outputs:} 
${\mathcal{Y}}_J = {\rm span} \{ \eta_j  \}_{j=1}^J$ test space.
\end{flushleft}                      

 \normalsize 

\begin{algorithmic}[1]

\For {$k=1, \ldots, n_{\rm train}$, $n=1,\ldots,N$ }

\State 
Compute
$\eta_{k,n} =  \texttt{R}_{\mathcal{Y}_{\rm hf}} \left(
D  R_{\Phi_{\mu^k}}  [ {U}_{\mu^k}] (\zeta_n, \cdot)
\right)$
 
 \EndFor
\vspace{3pt}

\State 
$  \mathcal{Y}_J  =
\texttt{POD} \left( \{ \eta_{k,n} \}_{k,n}, tol_{\rm pod}, (\cdot, \cdot)_{\mathcal{Y}_{\rm hf}}  \right)$

\end{algorithmic}

\end{algorithm}

\begin{remark}
\label{remark:continuous_test_space}
\textbf{Continuous approximation.}
As discussed in Remark \ref{remark:online_efficiency}, for computational reasons, it might be convenient to consider a continuous test space $\mathcal{Y}_{J}$.
This can be achieved by performing POD over continuous approximations of the test functions   $\{ \eta_{k,n} \}_{k,n}$.
As for the trial space, the continuous approximation is computed by simply averaging over facets.
\end{remark}

\subsubsection{Construction of the empirical quadrature rule}
 
 We denote by $\boldsymbol{\rho}^{\rm eq} \in \mathbb{R}_+^{N_{\rm e}}$ a vector of positive weights associated with \eqref{eq:approx_minresROM} and we denote by $\mathcal{I}_{\rm eq}$ the set of indices $k \in \{ 1,\ldots, N_{\rm e} \}$ such that $\rho_k^{\rm eq} > 0$.
 We seek $\boldsymbol{\rho}^{\rm eq} \in \mathbb{R}_+^{N_{\rm e}}$ such that
 \begin{enumerate}
 \item
 the number of nonzero entries in 
 $\boldsymbol{\rho}^{\rm eq}$ is as small as possible;
 \item
 the constant function is approximated correctly, 
 \begin{equation}
 \label{eq:constant_function_constraint}
\Big|
\sum_{k=1}^{N_{\rm e}} \rho_k^{\rm eq} | \texttt{D}^k  |
\,-\,
| \Omega | 
 \Big|
 \ll 1; 
 \end{equation}
\item
for all $\mu \in \mathcal{P}_{\rm train} = \{  \mu^k \}_{k=1}^{n_{\rm train}}$, the  empirical residual satisfies
  \begin{equation}
 \label{eq:accuracy_constraint}
\Big \|
\mathbf{J}_{N,J}^{\rm hf} 
\left(  \boldsymbol{\alpha}_{\mu}^{\rm train}  ,
\mu \right)^T
\left(  
\mathbf{R}_{N,J}^{\rm eq} \left(
\boldsymbol{\alpha}_{\mu}^{\rm train}, \mu
\right)
\,-\,
\mathbf{R}_{N,J}^{\rm hf} \left(
\boldsymbol{\alpha}_{\mu}^{\rm train}, \mu
\right)
\right)
 \Big \|_2
 \ll 1.
 \end{equation}
  Here, $ \boldsymbol{\alpha}_{\mu}^{\rm train} $ is chosen equal to the projection, that is 
 $Z_N \boldsymbol{\alpha}_{\mu}^{\rm train}= 
\Pi_{\mathcal{Z}_N} 
\left(
U_{\mu} \circ \boldsymbol{\Phi}_{\mu}
\right)
$.
  \end{enumerate}
The constant function constraint   --- which was   considered in \cite{yano2019discontinuous} --- is empirically found to improve the accuracy of the EQ procedure when the integral is close to zero due to the cancellation of the integrand in different parts of the domain. The accuracy  constraint in \eqref{eq:accuracy_constraint}  is an adaptation of the
\emph{manifold accuracy} constraints in \cite{yano2019discontinuous} to minimum residual ROMs and is motivated by the error analysis in 
 Appendix \ref{sec:BRR_EQ}. 
As shown in \cite{yano2019discontinuous}, the 
hyper-reduced system inherits the stability of the
DG discretization: 
(i) energy stability for linear hyperbolic systems, 
(ii) symmetry and non-negativity for
steady linear diffusion systems, and hence (iii) energy stability for linear convection-diffusion systems.

It is easy to verify that \eqref{eq:constant_function_constraint} and \eqref{eq:accuracy_constraint} could be rewritten in matrix form as
$$
\begin{array}{l}
\displaystyle{
\Big|
\sum_{k=1}^{N_{\rm e}} \rho_k^{\rm eq} | \texttt{D}^k  |
\,-\,
| \Omega | 
\Big|
=
|  \mathbf{G}_{\rm const}  \boldsymbol{\rho}^{\rm eq}  - | \Omega |  
\Big|
 \ll 1,
}
\\[3mm]
\displaystyle{
\Big \|
\mathbf{J}_{N,J}^{\rm hf} 
\left(  \boldsymbol{\alpha}_{\mu}^{\rm train}  ,
\mu \right)^T
\left(  
\mathbf{R}_{N,J}^{\rm eq} \left(
\boldsymbol{\alpha}_{\mu}^{\rm train}, \mu
\right)
\,-\,
\mathbf{R}_{N,J}^{\rm hf} \left(
\boldsymbol{\alpha}_{\mu}^{\rm train}, \mu
\right)
\right)
 \Big \|_2
 =
 \Big \|
 \mathbf{G}_{\mu}  \boldsymbol{\rho}^{\rm eq}  -  \mathbf{b}_{\mu}
 \Big\|_2
  \ll 1,
}
\\
\end{array}
$$
where $\mathbf{G}_{\rm const} = [ | \texttt{D}_1   |, \ldots,  | \texttt{D}_{N_{\rm e}}   |   ]$,
and
$ \mathbf{G}_{\mu}  \in \mathbb{R}^{N, N_{\rm e}},  \mathbf{b}_{\mu}  \in \mathbb{R}^N$ are  a suitable matrix and vector, whose explicit expressions can be derived exploiting the same argument as in \cite{farhat2015structure,taddei2018offline,yano2019discontinuous}.
The problem of finding   $\boldsymbol{\rho}^{\rm eq}$  can thus be reformulated as a sparse-representation (or best-subset selection) problem:
\begin{equation}
\label{eq:sparse_representation}
\min_{  \boldsymbol{\rho} \in \mathbb{R}^{N_{\rm e}} }
\;
\| \boldsymbol{\rho}   \|_0,
\quad
{\rm s.t} \quad
\left\{
\begin{array}{l}
\|\mathbf{G} \boldsymbol{\rho} - \mathbf{b}  \|_{\star} \leq \delta; \\[3mm]
\boldsymbol{\rho} \geq \mathbf{0}; \\
\end{array}
\right.
{\rm where} \; \;
\mathbf{G}  = \left[
\begin{array}{l}
\mathbf{G}_{\rm const}   \\[3mm]
\mathbf{G}_{\mu^1}  \\
\vdots \\
\mathbf{G}_{\mu^{n_{\rm train}}}  \\
\end{array}
\right]
\;\;
\mathbf{b}  = \left[
\begin{array}{l}
|\Omega|  \\[3mm]
\mathbf{b}_{\mu^1}  \\
\vdots \\
\mathbf{b}_{\mu^{n_{\rm train}}}  \\
\end{array}
\right]
\end{equation}
for   suitable choices of the   vector norm 
$\|\cdot   \|_{\star} $, and the tolerance $\delta>0$.
Here, $\| \cdot   \|_0$ is the $\ell^0$-norm, which counts the number of nonzero entries.  
In this work, we resort  to the nonnegative linear least-squares method considered in 
\cite{farhat2015structure} to obtain approximate solutions to \eqref{eq:sparse_representation}: more in detail, we rely on 
the Matlab function \texttt{lsqnonneg} which implements a variant of the Lawson and Hanson active set iterative algorithm \cite{lawson1974solving}.
Note that Yano in \cite{yano2019discontinuous} relies on $\ell^1$ relaxation 
to obtain approximate solutions to \eqref{eq:sparse_representation}: a thorough comparison between the two methods is beyond the scope of this work.

\subsection{Summary of the offline/online computational procedure}
\label{sec:summary}

We conclude this section by summarizing the offline/online computational decomposition.
During the offline stage, given the snapshots $\{ U^k =U_{\mu^k}  \}_k$, we proceed as follows:
\begin{enumerate}
\item
we apply 
\texttt{RePOD} to obtain the low-dimensional operators $Z_N, W_M$ and the training coefficients 
$\{ \boldsymbol{\alpha}^k \}_k \subset \mathbb{R}^N$ and 
$\{ \mathbf{a}^k \}_k \subset \mathbb{R}^M$ (cf. section \ref{sec:registration});
\item
we apply kernel regression to obtain the predictor
$\widehat{\mathbf{a}}: \mathcal{P} \to \mathbb{R}^M$ 
for the mapping coefficients
(cf. section \ref{sec:non_intrusive_ROM}) {and 
the predictor 
$\widehat{\boldsymbol{\alpha}}^{(0)}:\mathcal{P} \to \mathbb{R}^N$
for the solution coefficients};
\item
we perform hyper-reduction and we build the ROM for the solution coefficients
(cf. section \ref{sec:pMOR_minres}).
\end{enumerate}
{The RBF estimate of the solution coefficients is used during the online stage as initial guess 
$\widehat{\boldsymbol{\alpha}}_{\mu}^{(0)}$ for the Gauss-Newton method to solve \eqref{eq:approx_minresROM}.
}
{The output of the offline stage  includes
(i) the RBF data structure needed to estimate the mapping and solution  coefficients during the online stage;
(ii) the data structures for the evaluation of the dual residual
in \eqref{eq:approx_minresROM}; 
(iii) the space-time mesh $\mathcal{T}_{\rm hf}$, 
and
(iv) the DG representation of the  reduced operator $Z_N$, 
$\mathbf{Z}_{N} \in \mathbb{R}^{D \cdot N_{\rm hf}, N}$, and of the displacement operator
$W_M$,
$\mathbf{W}_{M} \in \mathbb{R}^{2 \cdot N_{\rm hf}, M}$, for visualization.
}

{Algorithm \ref{alg:online} summarizes the online stage.
The algorithm returns a FE vector 
$\widehat{\mathbf{U}}_{\mu} \in \mathbb{R}^{D \cdot N_{\rm hf}}$ and the deformed mesh
$\boldsymbol{\Phi}_{\mu}(\mathcal{T}_{\rm hf})$: if we interpret $\widehat{\mathbf{U}}_{\mu}$ as an element of the FE space $\mathcal{X}_{\rm hf}$,
$\widehat{U}_{\mu}$, we obtain an estimate of the mapped field $\widetilde{U}_{\mu}$; on the other hand, 
if we interpret $\widehat{\mathbf{U}}_{\mu}$ as an element of the mapped FE space $\mathcal{X}_{\rm hf, \Phi_{\mu}}$ (cf. \eqref{eq:mappedFE_space}),
$\widehat{U}_{\mu}^{\star}$, we obtain an estimate of 
${U}_{\mu}$.
Note that the $\widehat{\mathbf{U}}_{\mu}$ and the deformed mesh can be computed through a simple matrix vector multiplication:
\begin{equation}
\label{eq:definition_solution_reduced}
\widehat{\mathbf{U}}_{\mu} =
\mathbf{Z}_N \widehat{\boldsymbol{\alpha}}_{\mu},
\quad
\left\{ 
\boldsymbol{\Phi} \left(  \mathbf{x}_{i,k}^{\rm hf} \right)
=
 \mathbf{x}_{i,k}^{\rm hf}  +
[\left(  \mathbf{W}_M   \widehat{\mathbf{a}}_{\mu} \right)_{j_{i,k,1}}, \left(  \mathbf{W}_M   \widehat{\mathbf{a}}_{\mu} \right)_{j_{i,k,2}}]: \,  
  i=1,\ldots,n_{\rm lp}, k=1,\ldots, N_{\rm e} \right\},
\end{equation}
with $j_{i,k,d} = i + (k-1) n_{\rm lp} + (d-1) N_{\rm hf}$.
The first two steps of the procedure are independent of the dimension of the hf space; the third step is required for visualization and is in practice very fast.
}

\begin{algorithm}[H]                      
\caption{ {Online evaluation of the ROM} }     
\label{alg:online}     

 \small
\begin{flushleft}
\emph{Inputs:}  
$\mu \in \mathcal{P}$, parameter value, ROM structure (output of Offline stage),
 $\mathcal{T}_{\rm hf}$ space-time mesh.
\smallskip

\emph{Outputs:} 
 $(
\widehat{\boldsymbol{\alpha}}_{\mu}, 
\widehat{\mathbf{a}}_{\mu})$ prediction of solution and mapping coefficients;
$\widehat{\mathbf{U}}_{\mu}$  DG estimate of  the solution field,
$\boldsymbol{\Phi}_{\mu}(\mathcal{T}_{\rm hf})$ deformed mesh.
\end{flushleft}                      

 \normalsize 

\begin{algorithmic}[1]
\State
{Apply Algorithm \ref{alg:registration} to obtain
 $\mathcal{W}_M$ and $\{ \mathbf{a}^k  \}_k$}
\vspace{3pt}
 
\State
{Evaluate  RBF predictors of solution and mapping coefficients, 
$\widehat{\boldsymbol{\alpha}}_{\mu}^{(0)}$,
$\widehat{\mathbf{a}}_{\mu}$
};
\vspace{3pt}

 \State
 {Solve \eqref{eq:approx_minresROM} using Gauss-Newton method with initial guess $\widehat{\boldsymbol{\alpha}}_{\mu}^{(0)}$.
}
\vspace{3pt}

 \State
 {Compute
$\widehat{\mathbf{U}}_{\mu}$   and 
$\boldsymbol{\Phi}_{\mu}(\mathcal{T}_{\rm hf})$ using \eqref{eq:definition_solution_reduced}
.} 
 
\end{algorithmic}

\end{algorithm}

\section{Numerical results}
\label{sec:numerics}
We illustrate here the numerical performance of the proposed approach for the two model problems introduced at the end of section \ref{sec:formulation}. 
In Appendix \ref{sec:data_compression_vis}, we present further investigations of the data compression approach. {Numerical results are performed in \texttt{Matlab 2019b} on a commodity laptop.}

\subsection{Burgers equation}
\label{sec:burgers}

In Figure \ref{fig:spacetimereg_burgers1}, we illustrate performance of space-time registration. Here, the mapping is generated based on $n_{\rm train}=200$ snapshots through Algorithm \ref{alg:registration} with  $N_{\rm max}=3$, $\xi=10^{-4}$, $M_{\rm hf}=128$, $tol_{\rm pod}=10^{-4}$.  The resulting map consists of a four-term expansion ($M=4$).
In Figure  \ref{fig:spacetimereg_burgers1}(a), we show 
the behavior of normalized POD eigenvalues associated with the unregistered and registered  space-time snapshots, while 
Figure  \ref{fig:spacetimereg_burgers1}(b) shows the 
out-of-sample maximum relative projection error with and without registration 
$E^{\rm bf,\infty} =
\max_{j=1,\ldots,n_{\rm test}}
\; E^{\rm bf}(\mu^j),$ where $\mu^1,\ldots,\mu^{n_{\rm test}} \overset{\rm iid}{\sim} {\rm Uniform}(\mathcal{P})$ 
(cf. \eqref{eq:best_fit_error}) {with $n_{\rm test}=20$}.
Note that for $N \geq 4$, projection error is less than $10^{-2}$ and is comparable with the discretization error. 
 We conclude that the space-time registration procedure dramatically improves the linear reducibility of the space-time solution manifold.  

\begin{figure}[h!]
\centering
 \subfloat[ ] 
{  \includegraphics[width=0.4\textwidth]
 {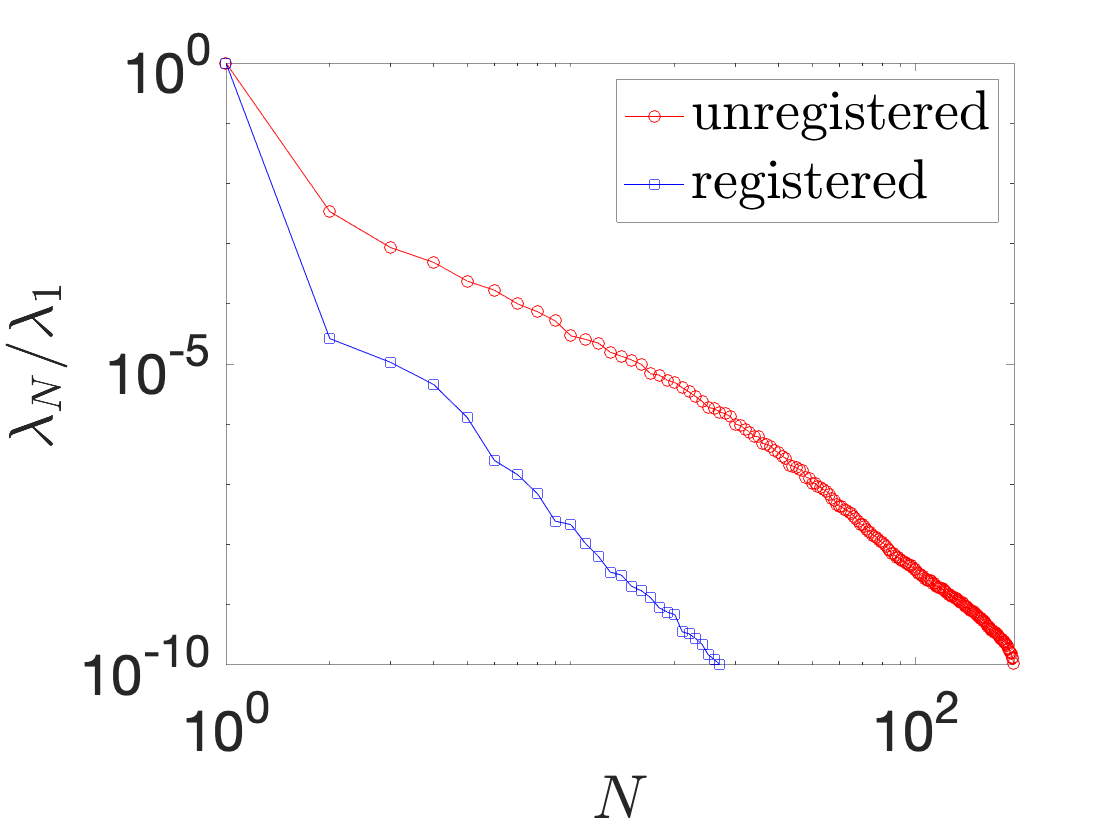}}
  ~~
 \subfloat[ ] 
{  \includegraphics[width=0.4\textwidth]
 {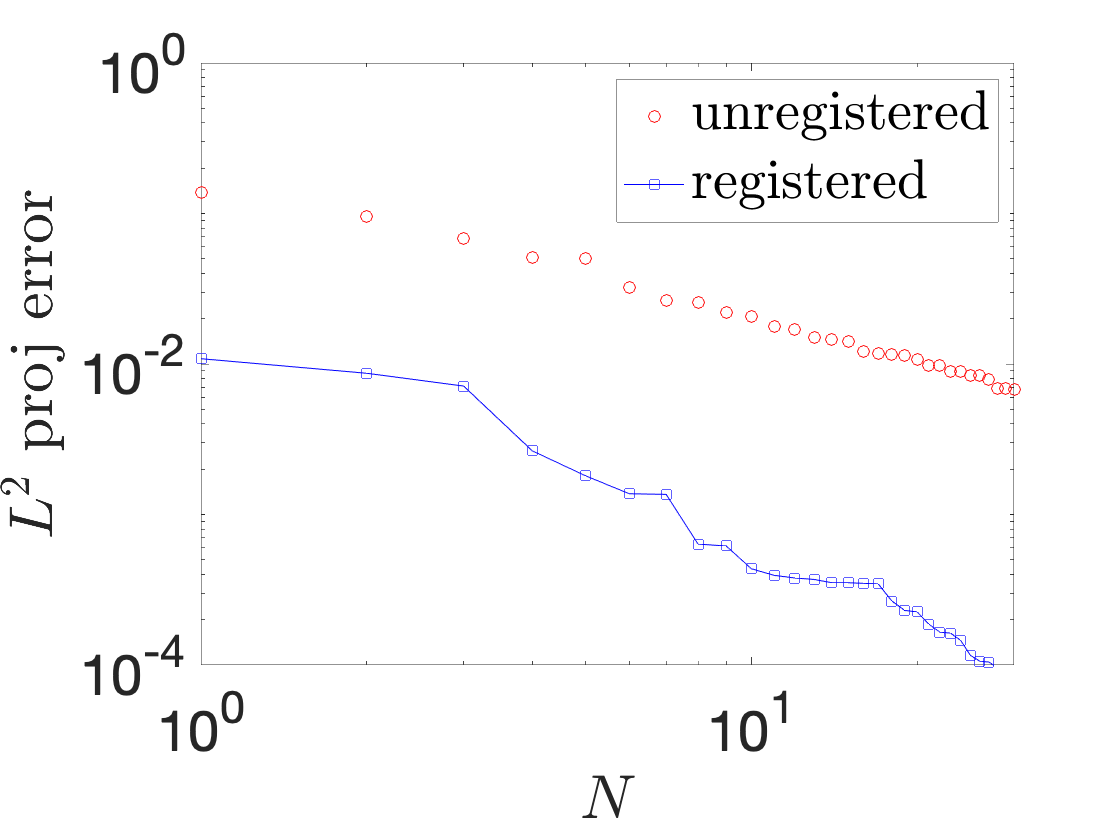}}
  
 \caption{Burgers equation;  space-time registration.
 (a) behavior of normalized POD eigenvalues associated with the unregistered and registered  space-time snapshots.
 (b)  behavior of the out-of-sample maximum relative projection error
(cf. \eqref{eq:best_fit_error}) 
  with and without registration.
}
 \label{fig:spacetimereg_burgers1}
  \end{figure}  

In Figure \ref{fig:burgers_projection}, we assess performance of
the Galerkin ROM \eqref{eq:galerkinROM}, 
the minimum residual ROM \eqref{eq:minresROM}, and
the approximate minimum residual ROM \eqref{eq:approx_minresROM_temp} for three different choices of the  size $J$ of the  test space $\mathcal{Y}_J$ for each value of $N$; in all cases, we do not perform hyper-reduction.
In Figure \ref{fig:burgers_projection}(a), we consider continuous trial and test spaces
(cf. Remarks \ref{remark:online_efficiency} and \ref{remark:continuous_test_space}), while in Figure \ref{fig:burgers_projection}(b), we consider discontinuous trial and test spaces.
On the y-axis, we here report the average relative out-of-sample  $L^2$ error in the reference configuration:
\begin{equation}
\label{eq:reference_config}
E_{\rm avg}^{\rm hf} =
\frac{1}{n_{\rm test}}
\;\;
\sum_{j=1}^{n_{\rm test}}
\;\;
\frac{
\| \widetilde{U}_{\mu^j} - \widehat{U}_{\mu^j}^{\rm hf}  \|
}{\| \widetilde{U}_{\mu^j}  \|},
\end{equation}
where the superscript $^{\rm hf}$ emphasizes the fact  that  the ROM relies on the hf quadrature rule (hf ROM). 
We observe that   minimum residual projection is superior to Galerkin projection in terms of performance; 
we also observe that the continuous approximation introduces an additional error that is negligible for $N \leq 5$.
Furthermore, we observe that approximate minimum residual approaches minimum residual for discontinuous test spaces: this is expected for sufficiently large values of $n_{\rm train}$ and $J$. For continuous approximations, we observe that minimum residual leads to slightly worse results than the approximate formulation: the difference is, however, modest for all values of $N$ considered.

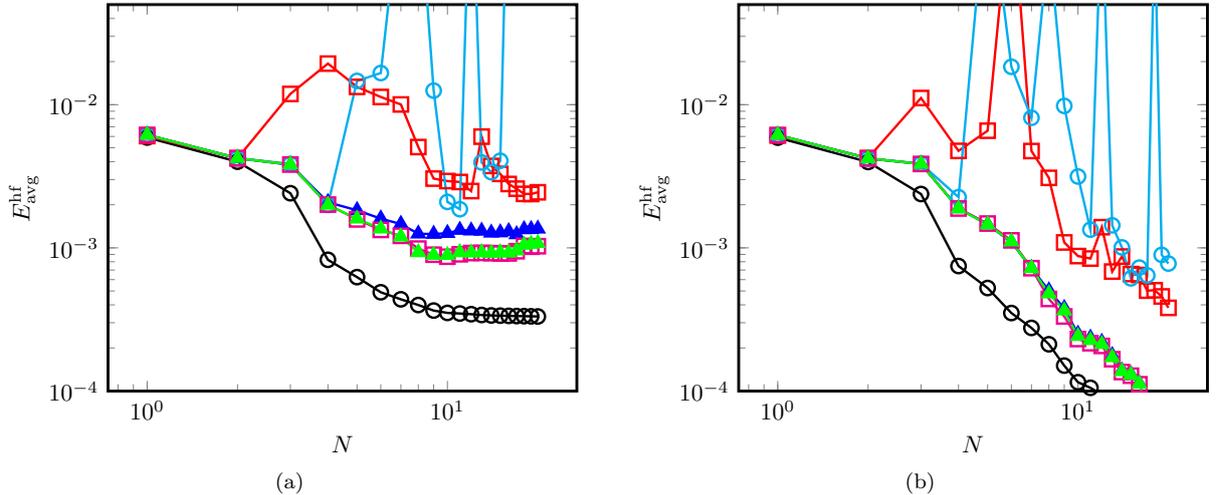
\begin{figure}[H]
\centering
\subfloat[ ] {
\begin{tikzpicture}[scale=0.9]
\begin{loglogaxis}[
xlabel = {$N$},
  ylabel = {$E_{\rm avg}^{\rm hf}$ },
  line width=1.2pt,
  mark size=3.0pt,
  ymin=0.0001,   ymax=0.05,
  ]
 
   \addplot[line width=1.pt,color=black,mark=o] table {data_burgers_CG_proj.dat};\label{p1}
  \addplot[line width=1.pt,color=red,mark=square]  table {data_burgers_CG_gal.dat};\label{p2}
  \addplot[line width=1.pt,color=blue,mark=triangle*] table {data_burgers_CG_minres.dat};\label{p3}
  
    \addplot[line width=1.pt,color=cyan,mark=o] table {data_burgers_CG_ES1.dat};\label{p4}
  \addplot[line width=1.pt,color=magenta,mark=square]  table {data_burgers_CG_ES2.dat};\label{p5}
  \addplot[line width=1.pt,color=green,mark=triangle*] table {data_burgers_CG_ES3.dat};\label{p6}  
\end{loglogaxis}

\end{tikzpicture}
}
~~~
\subfloat[ ] {
\begin{tikzpicture}[scale=0.9]
\begin{loglogaxis}[
xlabel = {$N$},
  ylabel = {$E_{\rm avg}^{\rm hf}$ },
  line width=1.2pt,
  mark size=3.0pt,
  ymin=0.0001,   ymax=0.05,
  ]
 
   \addplot[line width=1.pt,color=black,mark=o] table {data_burgers_DG_proj.dat}; 
  \addplot[line width=1.pt,color=red,mark=square]  table {data_burgers_DG_gal.dat};
  \addplot[line width=1.pt,color=blue,mark=triangle*] table {data_burgers_DG_minres.dat};
  
    \addplot[line width=1.pt,color=cyan,mark=o] table {data_burgers_DG_ES1.dat};
  \addplot[line width=1.pt,color=magenta,mark=square]  table {data_burgers_DG_ES2.dat};
  \addplot[line width=1.pt,color=green,mark=triangle*] table {data_burgers_DG_ES3.dat};
\end{loglogaxis}

\end{tikzpicture}
}

\bgroup
\sbox0{\ref{data}}%
\pgfmathparse{\ht0/1ex}%
\xdef\refsize{\pgfmathresult ex}%
\egroup

\caption[Caption in ToC]{
Burgers equation: performance of various ROMs (without hyper-reduction).
(a) continuous approximation.
(b) discontinuous approximation.
Projection error \tikzref{p1}, 
 Galerkin \tikzref{p2},
minimum residual \tikzref{p3} ,
approximate minimum residual   ($J=N$ \tikzref{p4} , 
$J=2N$ \tikzref{p5} , $J=3N$ \tikzref{p6}).
}
\label{fig:burgers_projection}
\end{figure}

In Figure \ref{fig:burgers_projection_EQ}, we illustrate performance of the hyper-reduction procedure;
here, we consider empirical test spaces of size $J=2N$. 
In Figure \ref{fig:burgers_projection_EQ}(a), 
we show the number of sampled elements $Q$ for several choices of $N$ for two different tolerances
(the total number of elements is equal to $N_{\rm e}=2616$): we observe that the number of sampled elements grows linearly with $N$,  and    ranges from $1 \%$ to 
$4 \%$ 
of the total number  for $tol = 10^{-8}$, and from 
$1 \%$ to  $10 \%$
for $tol = 2.5 \cdot 10^{-11}$.
Here, the tolerance $tol$ is a lower bound on the size of a step: the active-set   iterative procedure terminates at the $k$-th iteration if 
$\| \boldsymbol{\rho}^{{\rm eq,} k+1} -
\boldsymbol{\rho}^{{\rm eq,} k} 
 \|_2 \leq tol$.
In Figure \ref{fig:burgers_projection_EQ}(b), 
in the case of continuous approximations, 
we show the relative $L^2$ error 
\begin{equation}
\label{eq:reference_config_true}
E_{\rm avg}=
\frac{1}{n_{\rm test}}
\;\;
\sum_{j=1}^{n_{\rm test}}
\;\;
\frac{
\| \widetilde{U}_{\mu^j} - \widehat{U}_{\mu^j}   \|
}{\| \widetilde{U}_{\mu^j}  \|},
\end{equation}
for the hyper-reduced ROM, and we compare it with the error $E_{\rm avg}^{\rm hf}$   (hf quad)
 obtained with $J=2 N$ (same as magenta curve in Figure  \ref{fig:burgers_projection}(a)). 
Figure \ref{fig:burgers_projection_EQ}(c) shows the average online computational cost of Algorithm \ref{alg:online} with and without hyper-reduction; in 
Figure \ref{fig:burgers_projection_EQ}(d),
we report the average speedup of the hyper-reduced ROM with respect  to the ROM with hf quadrature and the average speedup of  
the hyper-reduced ROM compared to an explicit Runge-Kutta DG (RKDG) time-marching scheme with $N_{\rm hf}=450$  spatial degrees of freedom\footnote{The space-time solver used for snapshot generation relies on a RKDG time-marching scheme to generate an initial guess and then  to a Newton method with approximate line search to compute the space-time solution. Since we have not optimized performance of the hf space-time solver, we report absolute timings and we provide a comparison with respect to a state-of-the-art explicit solver with   comparable accuracy.}.
Note that hyper-reduction reduces online costs by a factor ten for all $N$; on the other hand, the speedup with respect to the explicit RKDG solver ranges from $2.1 \cdot 10^3$ for $N=1$ to 
$1.5 \cdot 10^2$ for $N=10$.
 
\begin{figure}[h!]
\centering
 \subfloat[ ] 
{  \includegraphics[width=0.4\textwidth]
 {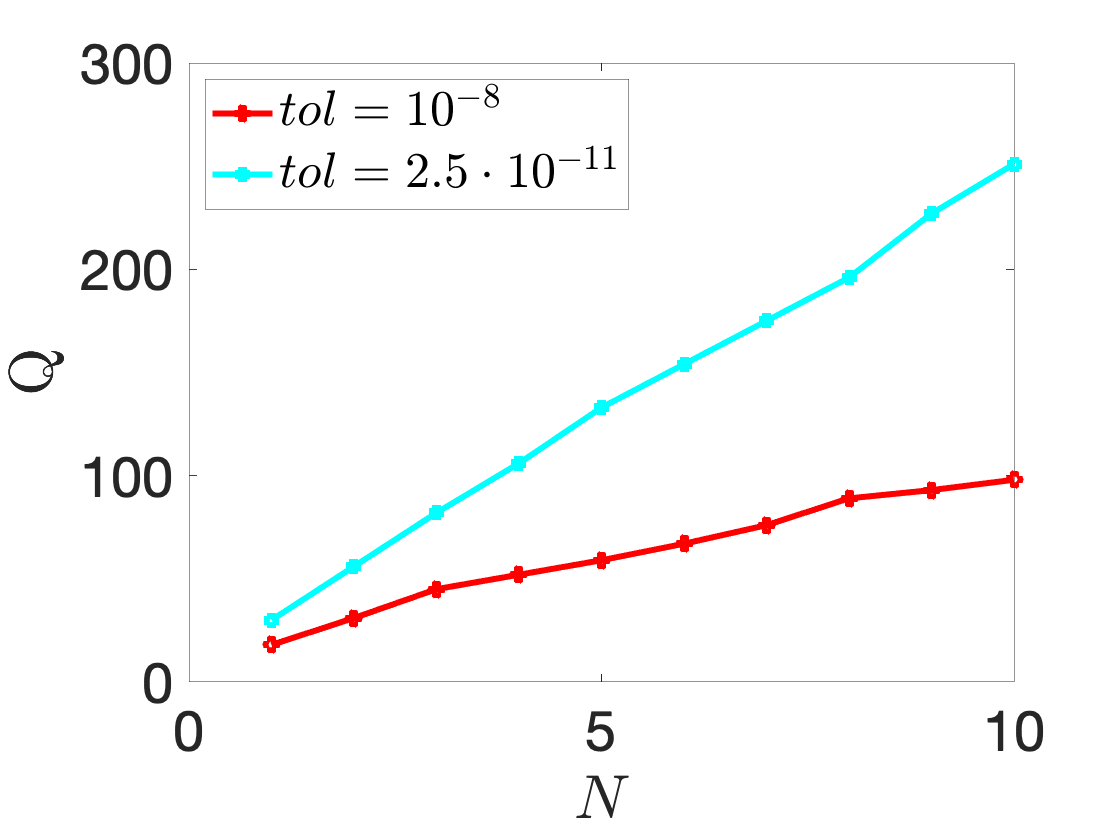}}
  ~~
 \subfloat[ ] 
{  \includegraphics[width=0.4\textwidth]
 {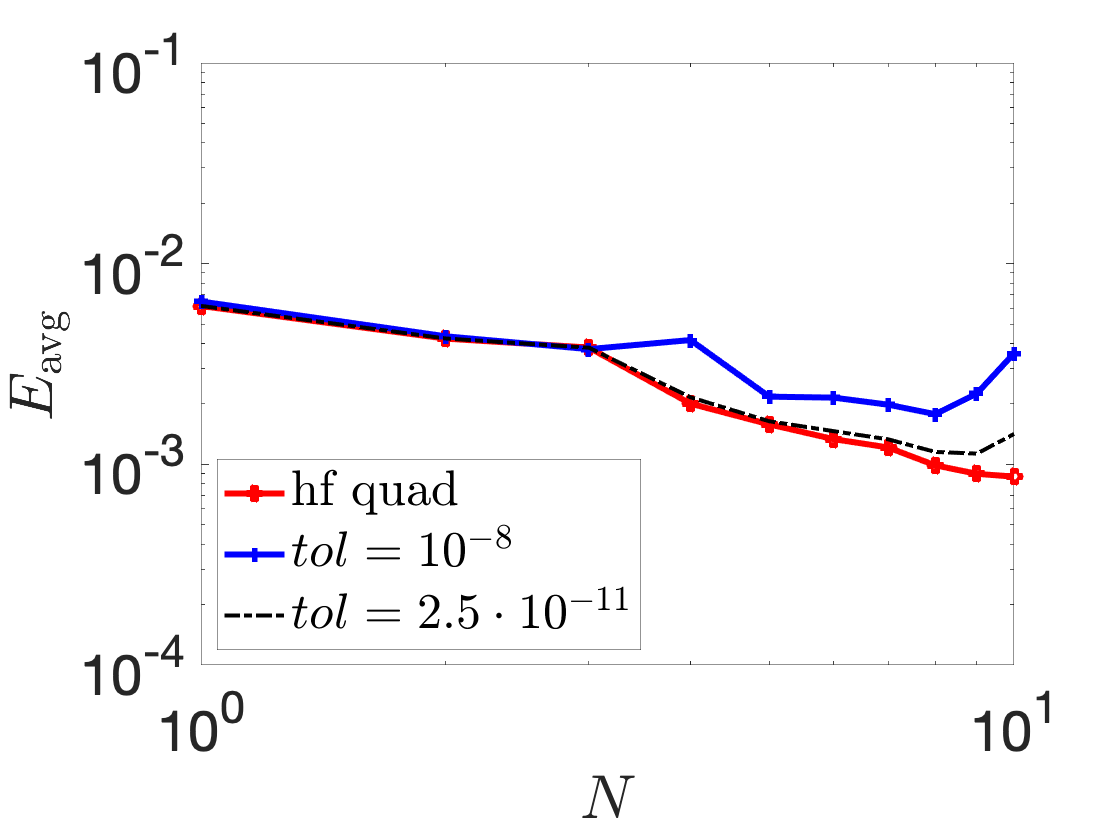}}
  
  \subfloat[ ] 
{  \includegraphics[width=0.4\textwidth]
 {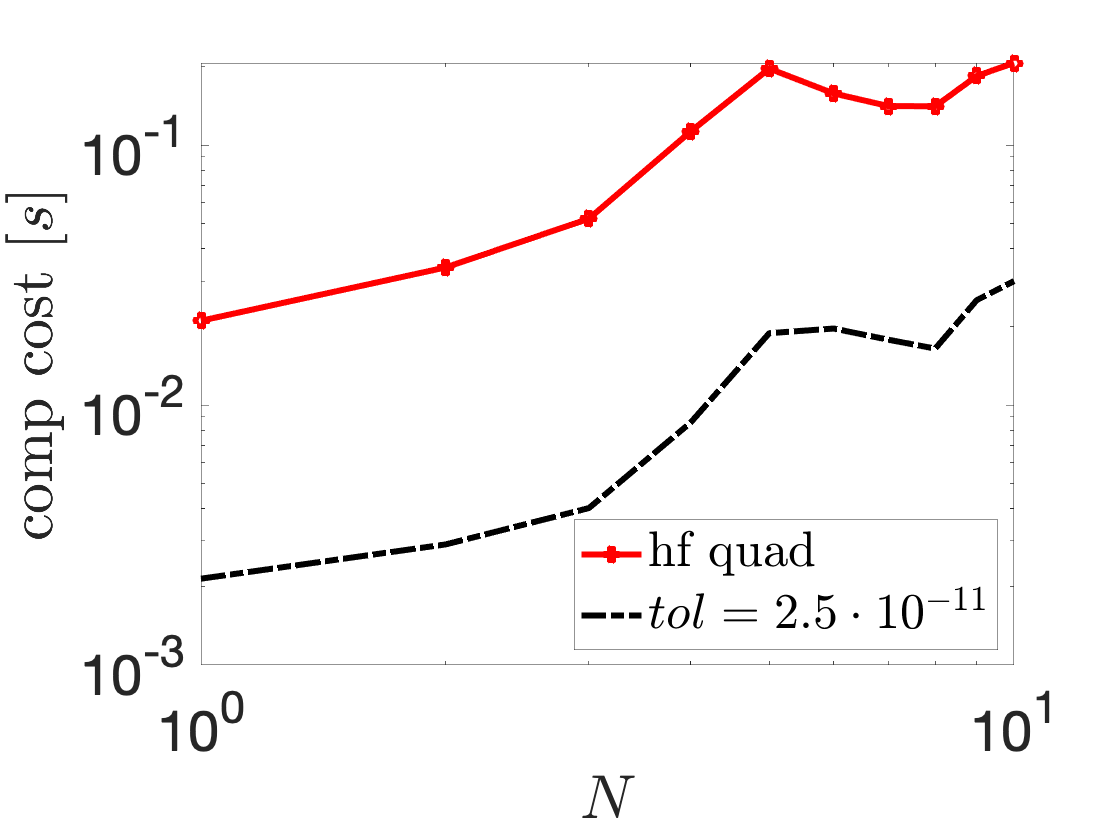}}
  ~~
 \subfloat[ ] 
{\includegraphics[width=0.4\textwidth]
 {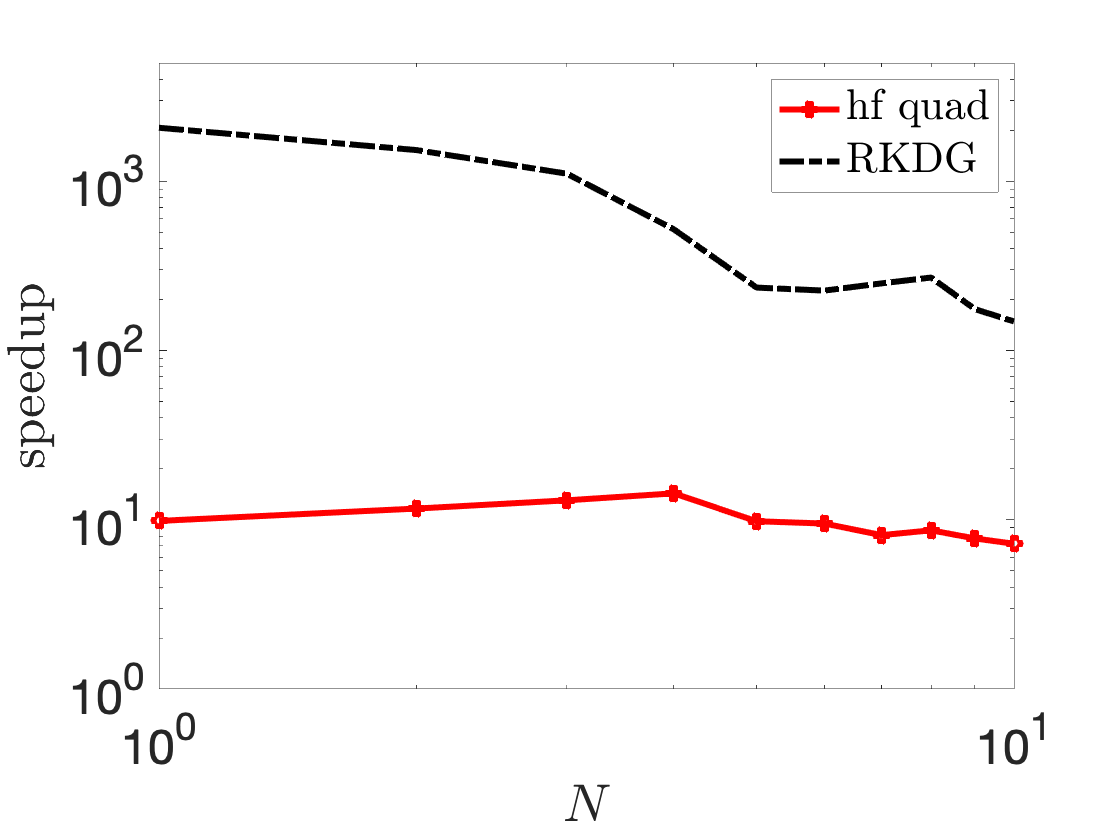}} 
  
\caption{Burgers equation; hyper-reduction for continuous approximation. 
(a) number of sampled elements $Q$ for several values of $N$, $J=2N$ and two tolerances.
(b) relative $L^2$ error  $E_{\rm avg}$ with respect to $N$ for two tolerances. 
(c) average  computational cost of Algorithm \ref{alg:online} with and without hyper-reduction; 
(d) average speedup 
with respect  to the ROM with hf quadrature (hf quad) and
with respect to a RKDG explicit solver (RKDG).
 }
 \label{fig:burgers_projection_EQ}
  \end{figure} 

Finally, in Figure \ref{fig:burgers_projection_EQ2}, we show the mesh and the reduced meshes for two choices of  trial and test spaces: we observe that most sampled elements are located in the proximity of the shock.
Thanks to the registration process, the position of the shocks is nearly parameter-independent: this explains why the hyper-reduction procedure is able to achieve accurate performance with a limited number of elements.

\begin{figure}[h!]
\centering
\subfloat[] 
{  \includegraphics[width=0.3\textwidth]
 {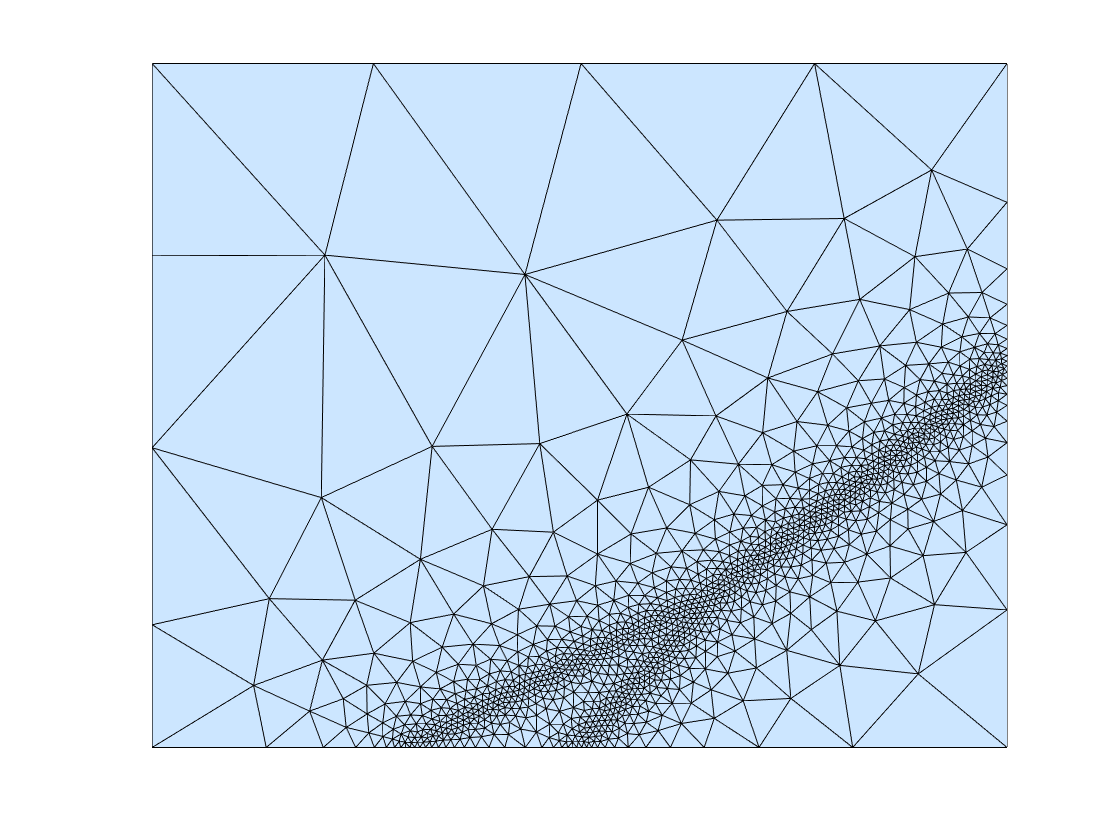}}
~~
\subfloat[$N=2,J=4$] 
{  \includegraphics[width=0.3\textwidth]
 {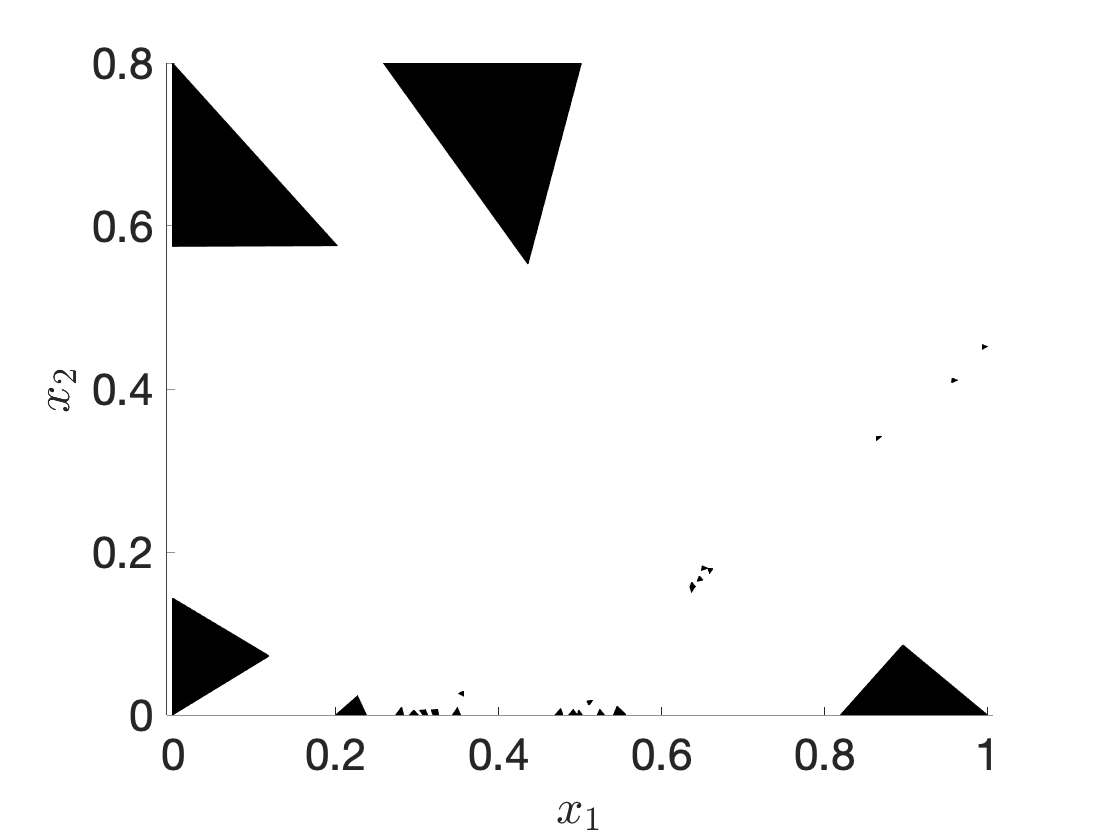}}
    ~~
 \subfloat[$N=6,J=12$] 
{  \includegraphics[width=0.3\textwidth]
 {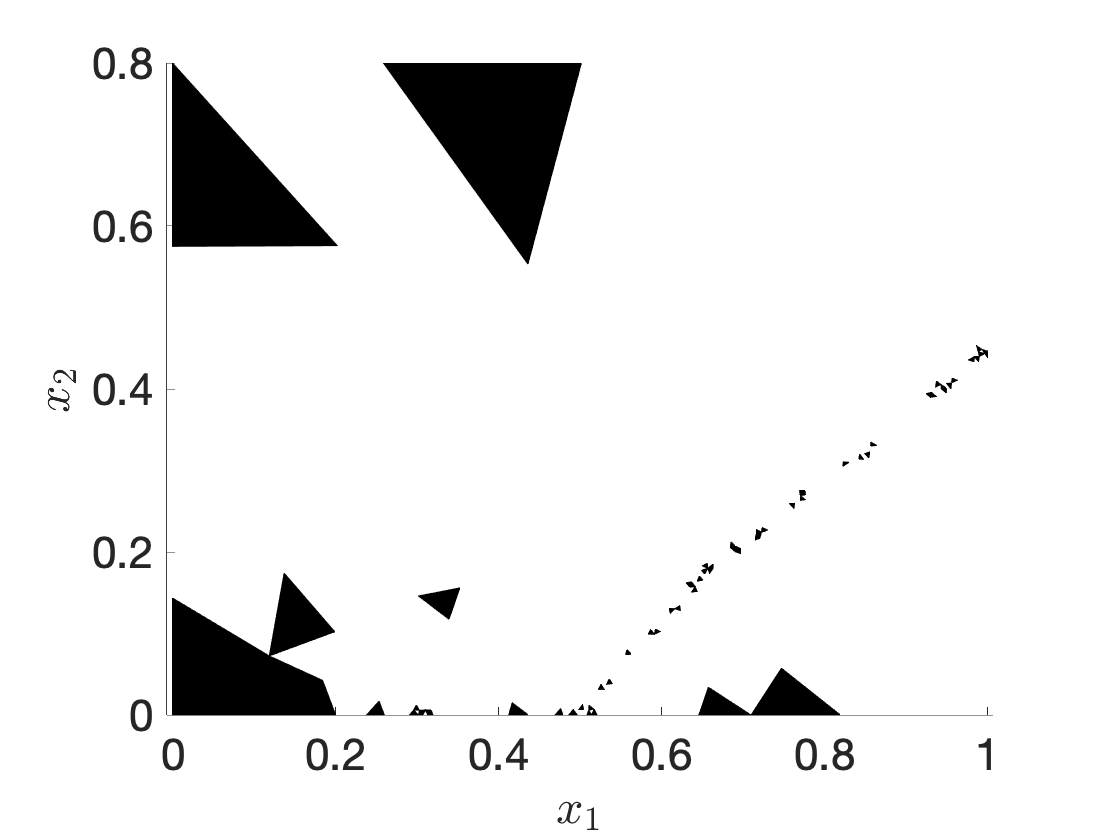}}
 
 \caption{Burgers equation; hyper-reduction.
Sampled elements for $tol=10^{-8}$, for two values of $N$.
}
 \label{fig:burgers_projection_EQ2}
  \end{figure} 
 
\subsection{Shallow water equations}
\label{sec:shallow_water}

In Figure \ref{fig:spacetimereg_sv1}, 
 we illustrate performance of space-time registration. The  mapping is generated based on $n_{\rm train}=100$ snapshots through Algorithm \ref{alg:registration} with  $N_{\rm max}=5$, $\xi=10^{-4}$, $M_{\rm hf}=128$, $tol_{\rm pod}=10^{-4}$. 
We recall that the algorithm is applied to the filtered height $h_{\mu}^{\rm f}$ and that the initial template space is set equal to
$\mathcal{T}_{N_0=2} = {\rm span} \{ h_{\bar{\mu}}^{\rm f},  h_{\rm bf}^{\rm f} \}$.
 The resulting map consists of a five-term expansion ($M=5$).
Figure \ref{fig:spacetimereg_sv1}(a) shows the behavior of the $L^2$ POD eigenvalues associated with the snapshots $\{ U_{\mu^k}  \}_{k=1}^{n_{\rm train}}$ and with the mapped snapshots
$\{ \widetilde{U}_{\mu^k}  \}_{k=1}^{n_{\rm train}}$;
Figure \ref{fig:spacetimereg_sv1}(b) shows the behavior of the out-of-sample relative projection error based on $n_{\rm test}=20$ snapshots
$\{ U_{\mu^j}  \}_{j=1}^{n_{\rm test}}$ with 
$\mu^1,\ldots, \mu^{n_{\rm test}} \overset{\rm iid}{\sim} {\rm Uniform}(\mathcal{P})$
(cf. \eqref{eq:best_fit_error}). We observe that the approach is extremely effective to reduce the linear complexity of the solution manifold for moderate values of $N$. 
 
\begin{figure}[h!]
\centering
\subfloat[ ] 
{  \includegraphics[width=0.4\textwidth]
 {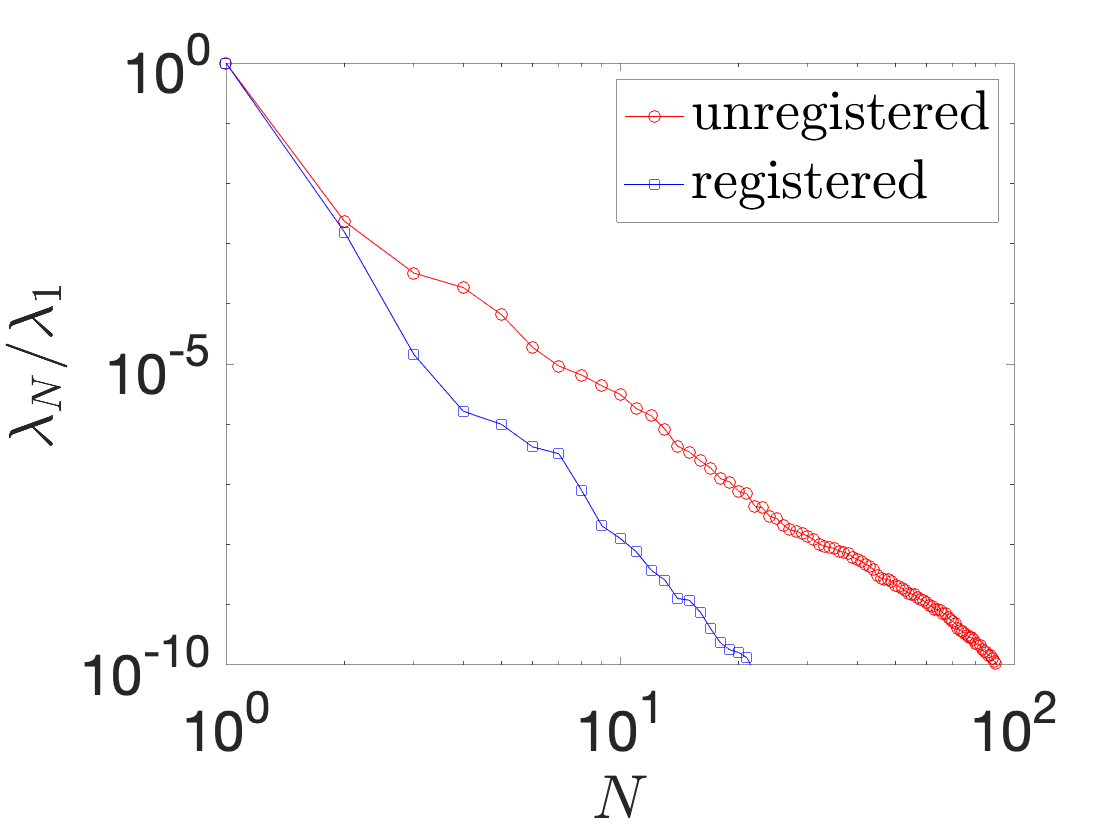}}
  ~~
 \subfloat[ ] 
{  \includegraphics[width=0.4\textwidth]
 {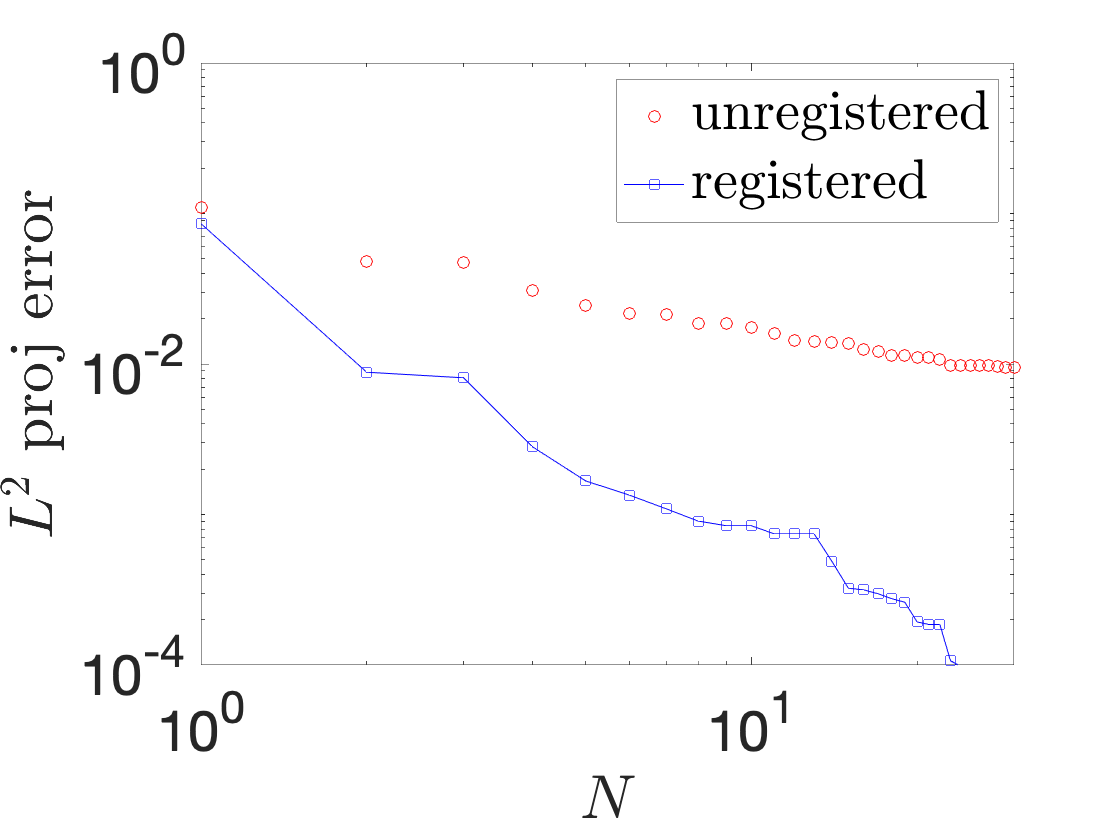}}
  
 \caption{Shallow water equations;  space-time registration.
(a)  behavior of the POD eigenvalues associated with the unregistered 
and registered configurations. 
(b)  behavior of the out-of-sample maximum relative projection error 
(cf. \eqref{eq:best_fit_error})
with and without registration.
}
 \label{fig:spacetimereg_sv1}
  \end{figure}  

 In Figure  \ref{fig:sv_projection},   we illustrate performance of projection-based ROMs without hyper-reduction. 
Similarly to the Burgers equation (cf. Figure \ref{fig:burgers_projection}), we empirically find that Galerkin projection might lead to instabilities; on the other hand, our approximate minimum residual approach is effective, provided that the size of the test space satisfies $J \gtrsim 2 N$. Furthermore, the continuous approximation introduces an additional error that is negligible for $N \leq 7$. 

 In Figure \ref{fig:sv_projection_EQ},
 we illustrate performance of  the   hyper-reduction procedure.
Figure \ref{fig:sv_projection_EQ}(a) shows the number of sampled elements $Q$ for several choices of $N$,  $J=2N$, and  two different tolerances (the total number of elements is equal to $N_{\rm e}=2364$).
Note that as for the Burgers model problem
the number of sampled elements grows linearly with $N$.
Figure \ref{fig:sv_projection_EQ}(b) shows the behavior of the relative $L^2$ error  \eqref{eq:reference_config_true} for the hyper-reduced ROM, and we compare it with the error of the ROM based on the truth quadrature. Note that for $tol = 2.5 \cdot 10^{-11}$ the hyper-reduced ROM guarantees the same accuracy as the non-hyper-reduced ROM, for all values of $N$ considered.
Figure \ref{fig:sv_projection_EQ}(c) shows the average online computational cost of Algorithm \ref{alg:online} with and without hyper-reduction; in 
Figure \ref{fig:sv_projection_EQ}(d),
we report the average speedup of the hyper-reduced ROM with respect  to the ROM with hf quadrature and the average speedup of the 
the hyper-reduced ROM compared to a explicit Runge-Kutta DG time-marching scheme with $N_{\rm hf}=900$  spatial degrees of freedom.
Note that hyper-reduction reduces online costs by roughly a factor ten  for all $N$; on the other hand, the speedup with respect to the explicit RKDG solver ranges from $4.4 \cdot 10^2$ for $N=1$ to 
$7.2\cdot 10^1$ for $N=10$.
 
\begin{figure}[H]
\centering
\subfloat[ ] {
\begin{tikzpicture}[scale=0.9]
\begin{loglogaxis}[
xlabel = {$N$},
  ylabel = {$E_{\rm avg}^{\rm hf}$ },
  line width=1.2pt,
  mark size=3.0pt,
  ymin=0.0001,   ymax=0.05,
  ] 
   \addplot[line width=1.pt,color=black,mark=o] table {data_sv_CG_proj.dat};
  \addplot[line width=1.pt,color=red,mark=square]  table {data_sv_CG_gal.dat};
  \addplot[line width=1.pt,color=blue,mark=triangle*] table {data_sv_CG_minres.dat};
  
    \addplot[line width=1.pt,color=cyan,mark=o] table {data_sv_CG_ES1.dat};
  \addplot[line width=1.pt,color=magenta,mark=square]  table {data_sv_CG_ES2.dat};
  \addplot[line width=1.pt,color=green,mark=triangle*] table {data_sv_CG_ES3.dat};
\end{loglogaxis}

\end{tikzpicture}
}
~~~
\subfloat[ ] {
\begin{tikzpicture}[scale=0.9]
\begin{loglogaxis}[
xlabel = {$N$},
  ylabel = {$E_{\rm avg}^{\rm hf}$ },
  line width=1.2pt,
  mark size=3.0pt,
  ymin=0.0001,   ymax=0.05,
  ]
 
   \addplot[line width=1.pt,color=black,mark=o] table {data_sv_DG_proj.dat};
  \addplot[line width=1.pt,color=red,mark=square]  table {data_sv_DG_gal.dat};
  \addplot[line width=1.pt,color=blue,mark=triangle*] table {data_sv_DG_minres.dat};
  
    \addplot[line width=1.pt,color=cyan,mark=o] table {data_sv_DG_ES1.dat};
  \addplot[line width=1.pt,color=magenta,mark=square]  table {data_sv_DG_ES2.dat};
  \addplot[line width=1.pt,color=green,mark=triangle*] table {data_sv_DG_ES3.dat};
\end{loglogaxis}

\end{tikzpicture}
}

\bgroup
\sbox0{\ref{data}}%
\pgfmathparse{\ht0/1ex}%
\xdef\refsize{\pgfmathresult ex}%
\egroup

\caption[Caption in ToC]{
Shallow water equations: performance of various ROMs (without hyper-reduction).
(a) continuous approximation.
(b) discontinuous approximation.
Projection error \tikzref{p1}, 
 Galerkin \tikzref{p2},
minimum residual \tikzref{p3} ,
approximate minimum residual ($J=N$ \tikzref{p4} , 
$J=2N$ \tikzref{p5} , $J=3N$ \tikzref{p6}).
}
\label{fig:sv_projection}
\end{figure}
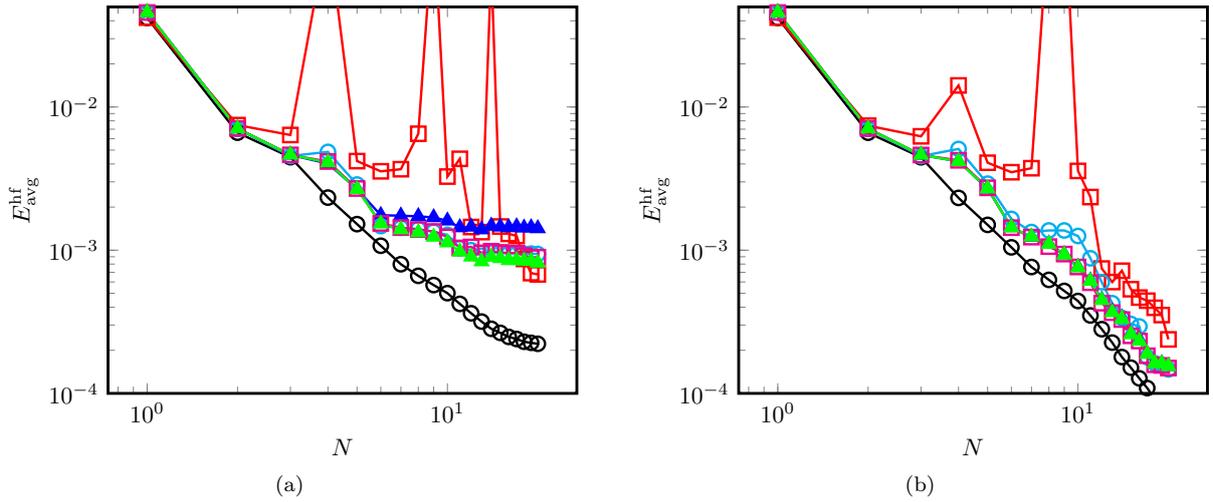
 
\begin{figure}[h!]
\centering
 \subfloat[ ] 
{  \includegraphics[width=0.4\textwidth]
 {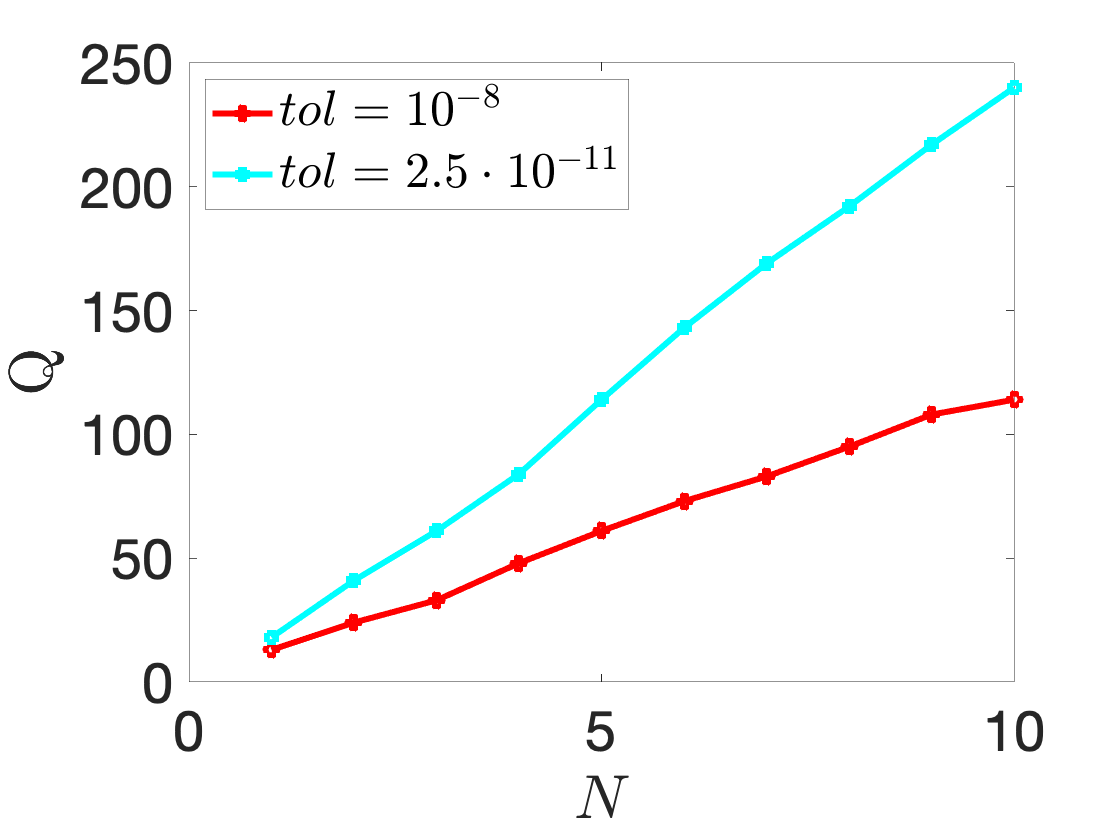}}
  ~~
 \subfloat[ ] 
{  \includegraphics[width=0.4\textwidth]
 {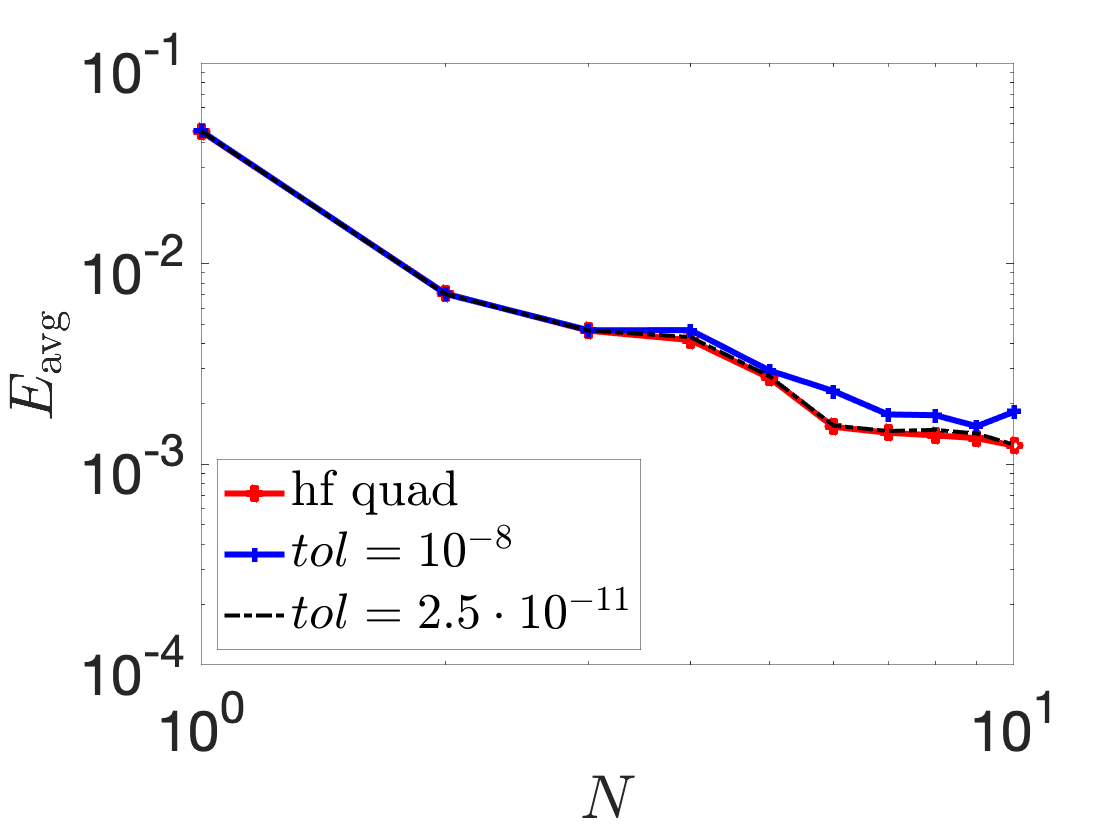}}
  
 \subfloat[ ] 
{  \includegraphics[width=0.4\textwidth]
 {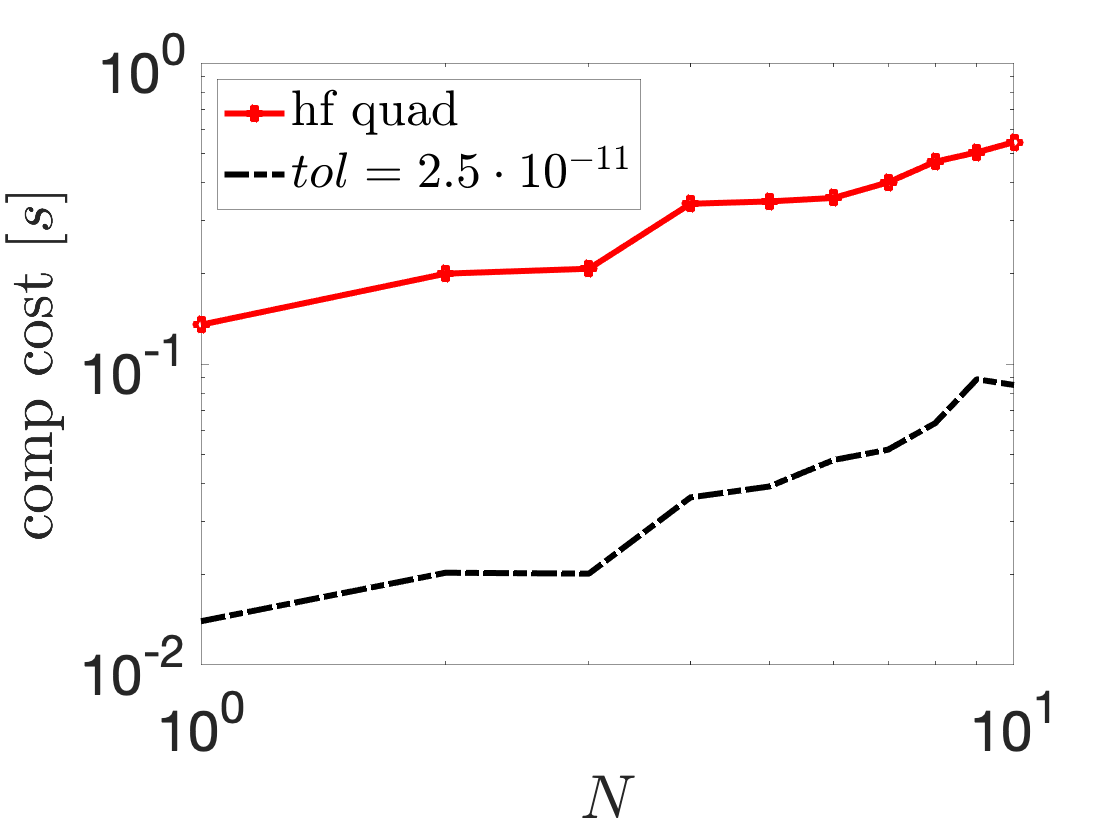}}
  ~~
 \subfloat[ ] 
{  \includegraphics[width=0.4\textwidth]
 {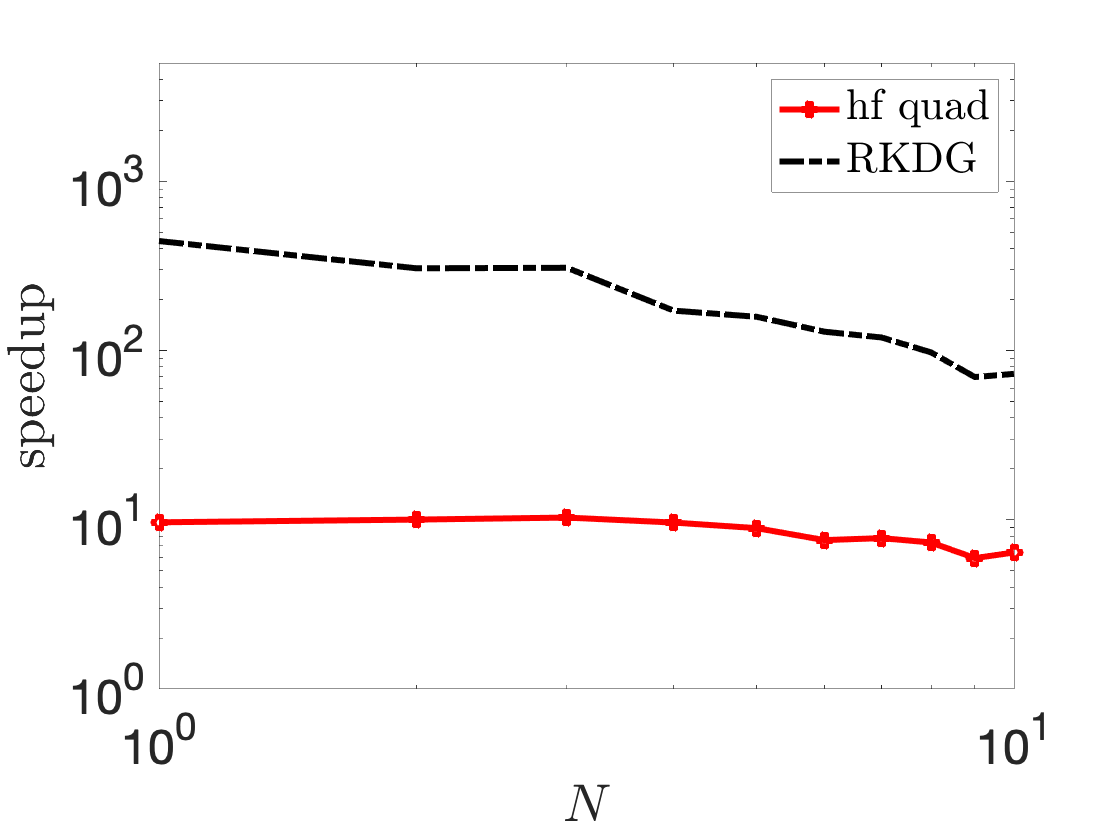}}  
  
\caption{Shallow water equations; hyper-reduction for continuous approximation. 
(a) number of sampled elements $Q$ for several values of $N$, $J=2N$ and two tolerances.
(b) relative $L^2$ error  $E_{\rm avg}$ with respect to $N$ for two tolerances. 
c) average  computational cost of Algorithm \ref{alg:online} with and without hyper-reduction; 
(d) average speedup 
with respect  to the ROM with hf quadrature (hf quad) and
with respect to a RKDG explicit solver (RKDG).
 }
 \label{fig:sv_projection_EQ}
  \end{figure}

\section{Conclusions}
\label{sec:conclusions}
In this work, we developed and numerically validated a model reduction procedure for hyperbolic PDEs in presence of shocks. The approach relies on a general (i.e., independent of the underlying PDE model) data compression procedure: given the snapshot set, we first perform    space-time registration to ``freeze" the position of the shock; then, we resort to POD to approximate the registered (mapped) field. To estimate the registered field, we resort to an hyper-reduced approximate minimum residual formulation: our statement is based on the introduction of a low-dimensional empirical test space \cite{taddei2018offline} and of an empirical quadrature rule 
 \cite{farhat2015structure,yano2019discontinuous} to reduce online assembling costs.

Numerical results demonstrate the effectiveness of the space-time registration-based compression strategy to construct a low-dimensional representation of the solution manifold for the two model problems considered.  Furthermore, 
by improving the linear compressibility of the solution manifold, registration has also the effect of reducing the required size $J$ of the empirical test space, and also the number $Q$ of sampled elements.

We aim to extend the approach in several directions.
In this work, we resorted to a non-intrusive method for the computation of the mapping coefficients, and to an intrusive (projection-based) method for the computation of the registered solution: in the future, we aim to combine our data compression procedure with fully-intrusive ROMs and non-intrusive ROMs.
Fully-intrusive approaches based on projection (see \cite{mojgani2017arbitrary})
might help us devise robust ROMs based on moderate-dimensional snapshot sets.
Furthermore, if complemented by reliable \emph{a posteriori} error indicators, they might also lead to the development of adaptive sampling algorithms to dramatically reduce offline training costs.
On the other hand, non-intrusive techniques
(see \cite{chakir2009methode,gallinari2018reduced,guo2018reduced})
might be considerably easier to implement and also
to integrate with existing codes, and --- at the price of larger training costs --- might contribute to reduce online costs.

We also wish to investigate the performance of the proposed model reduction technique
to a broader  class of PDE models in computational mechanics, in one and more dimensions. In this respect, we wish to consider two-dimensional steady and unsteady advection-dominated problems that arise in incompressible and compressible fluid mechanics applications. Furthermore, we  envision that our approach might be of interest for 
solid mechanics applications such as
contact problems.

\medskip

\textbf{Acknowledgments:}
The authors thank 
Dr. Andrea Ferrero (Politecnico di Torino) and
Professor  Angelo Iollo (Inria Bordeaux)   for fruitful discussions.
The present work was  partially 
supported by EDF
(EDF-INRIA research project number  8610-4220129899).
\appendix

\section{Change-of-variable formulas}
\label{sec:change_variable}
For completeness, we report here standard change-of-variable formulas used to derive the mapped formulation of section \ref{sec:space_time}. Given $\widetilde{\texttt{D}} \subset \Omega$ and the bijection
$\boldsymbol{\Phi}:  \widetilde{\texttt{D}} \to \texttt{D}$, we have
\begin{equation}
\label{eq:change_variable}
\begin{array}{l}
\displaystyle{
\int_{\texttt{D}}  \,  u  \, d \mathbf{x} \, = \, 
\int_{  \widetilde{\texttt{D}}   } \, \left( u \circ  \boldsymbol{\Phi}   \right)  \;
g   \, d  \mathbf{X};
}
\\[3mm]
\displaystyle{ 
\int_{\texttt{D}}  \, 
\mathbf{b}  \cdot \nabla \, u  \, d\mathbf{x} \, = \,
\int_{  \widetilde{\texttt{D}}   }  \,
 \mathbf{b} \circ   \boldsymbol{\Phi}   \cdot 
\left(   \mathbf{G} ^{-T} \, 
\widetilde{\nabla} \,  u \circ  \boldsymbol{\Phi}    \right)   \, 
g \, \,  d  \mathbf{X} 
}
\\[3mm]
\displaystyle{ 
\int_{ \partial   \texttt{D}}  \,  u   \, d \mathbf{x} \, = \, 
\int_{ \partial  \widetilde{\texttt{D}}   } \, u \circ \boldsymbol{\Phi}    \, 
\|   g   \;     \mathbf{G}^{-T} \;  \mathbf{N}    \|_2
 \, d  \mathbf{X};
}
\\
\end{array}
\end{equation}
for all $u, \mathbf{b} \in C^1(\Omega)$. Note that \eqref{eq:change_variable}$_3$ is a straightforward consequence of Nanson's formula (cf. \cite[Chapter 1]{marsden1994mathematical}).

Fluxes $F_{\Phi}, S_{\Phi}$ should satisfy 
$$
\int_{ \widetilde{\texttt{D}} }  \;
\left(
\widetilde{\nabla} \cdot F_{\Phi}(\widetilde{U}) \, -\, S_{\Phi}(\widetilde{U})
\right)
\; d\mathbf{X} 
=
\int_{  {\texttt{D}} }  \;
\left(
 {\nabla} \cdot F({U}) \, -\, S({U})
\right) \; d\mathbf{x},
$$
for all $U \in C^1(\Omega; \mathbb{R}^P)$. Exploiting \eqref{eq:change_variable}, 
and the divergence theorem, 
we obtain that
$F_{\Phi}, S_{\Phi}$ satisfy
$F_{\Phi}(\cdot) = g  F(\cdot) \mathbf{G}^{-T}$  and
$S_{\Phi}(\cdot) = g  S(\cdot)$, which is \eqref{eq:mapped_strong_form_b}.

\section{Online residual calculations}
\label{sec:implementation}
We provide some details concerning the online calculation of the residual; we further explain why the CG approximation reduces the memory cost of the ROM.
As in the main body of the paper, we denote by $\mathcal{I}_{\rm eq} \subset \{1,\ldots, N_{\rm e} \}$ the sampled elements over which we perform online integration, and we define  $\mathcal{X}_{\rm hf}^{\rm cg} = \mathcal{X}_{\rm hf}  \cap [C(\Omega)]^D$.  We also denote by  $\mathcal{U}_{\rm hf}$ the scalar DG FE space of order $p$ such that $\mathcal{X}_{\rm hf} = [ \mathcal{U}_{\rm hf} ]^D$.  We  define the average operator $\{ \cdot \}$, the normal vector average $\{ \cdot \}_{\mathbf{n}}$, and the jump operator $\mathcal{J} (\cdot)$
 such that
\begin{equation}
\label{eq:edge_ops}
\{ w \} = \left\{
\begin{array}{ll}
\frac{1}{2}(w^+ + w^-) &  {\rm on} \; \partial \mathcal{T}_{\rm hf} \setminus \partial \Omega, \\[3mm]
w &  {\rm on} \; \partial \mathcal{T}_{\rm hf} \cap \partial \Omega; \\
\end{array}
\right.
\quad
\{ w \}_{\mathbf{n}} =
\{ w \}  \cdot \mathbf{n}^+;
\quad
\mathcal{J}  w 
=
 \left\{
\begin{array}{ll}
 w^+ - w^-  &  {\rm on} \; \partial \mathcal{T}_{\rm hf} \setminus \partial \Omega, \\[3mm]
w   &  {\rm on} \; \partial \mathcal{T}_{\rm hf} \cap \partial \Omega. \\
\end{array}
\right.
\end{equation}
We further denote by $\Gamma_{\rm D}^{i}  \subset \partial \Omega $ the Dirichlet boundary for the $i$-th component of the solution field, $i=1,\ldots,D$; and we introduce  the Dirichlet operators
\begin{equation}
\label{eq:dirichlet_operator}
\left( \mathcal{D}(w)  \right)_i
=
 \left\{
\begin{array}{ll}
( U_{\rm D}  )_i &  {\rm on} \; \partial \mathcal{T}_{\rm hf} \cap \Gamma_{\rm D}^{i} , \\[3mm]
w_i   &  {\rm on} \; \partial \mathcal{T}_{\rm hf} \setminus \Gamma_{\rm D}^{i}. \\
\end{array}
\right.
\quad
  \mathcal{J}_{\rm D}^i v  
=
 \left\{
\begin{array}{ll}
v^+ - v^- &  
{\rm on} \; \partial \mathcal{T}_{\rm hf} \cap \Gamma_{\rm D}^{i} , \\[3mm]
v - ( U_{\rm D}  )_i    &  {\rm on} \; \partial \mathcal{T}_{\rm hf} \setminus \Gamma_{\rm D}^{i}. \\
\end{array}
\right.
i=1,\ldots,D;
\end{equation}
where $w \in \mathcal{X}_{\rm hf}$ and $v\in \mathcal{U}_{\rm hf}$. Finally,
given the facet $\partial \texttt{D}_{\ell}^k$, $\ell=1,2,3$, 
$k=1,\ldots,N_{\rm e}$, 
 we introduce the lifting operator
 $\mathbf{r}_{\ell,k} : \left[L^2(\partial \texttt{D}_{\ell}^k) \right]^d \to  [\mathcal{U}_{\rm hf}]^2$
\begin{equation}
\label{eq:lifting_operator}
\sum_{k'} \int_{\texttt{D}^{k'}} \,  \mathbf{r}_{\ell,k}(  w) \,  \cdot \, v \, dx \,  = \, - \int_{
\partial \texttt{D}_{\ell}^k
} \, w \,\cdot \,  \{  v \} \, dx \quad
\forall \, \mathbf{v} \in [  \mathcal{U}_{\rm hf}]^2.
\end{equation}

Recalling the expression of $R_{\Phi}^{\rm c}$ and
$R^{\rm d}$,  we find that
$R_{\Phi}^{\rm c,eq}(w,v)  =\sum_{k \in  \mathcal{I}_{\rm eq}  }  \rho_k^{\rm eq} r_k^{\rm c}(w,v)$
and
$R^{\rm d,eq}(w,v)  =\sum_{k \in  \mathcal{I}_{\rm eq}  }  \rho_k^{\rm eq} r_k^{\rm d}(w,v)$
 with 
\begin{equation}
\label{eq:fluxes_explained}
\begin{array}{l}
\displaystyle{
r_k^{\rm c}(w,v) = 
\int_{\partial \texttt{D}^k} 
v \cdot \mathcal{H}_{\Phi}(w^+, w^-, \mathbf{N})  d \mathbf{X} \, - \,
\int_{  \texttt{D}^k} 
\widetilde{\nabla}  v\cdot  
F_{\Phi}(w )
 \; d \mathbf{X} \, 
-
\int_{ \texttt{D}^k} 
v\cdot S_{\Phi}(w) d \mathbf{X}
}\\[3mm]
\displaystyle{
r_k^{\rm d}(w,v) = 
\left(
\sum_{i=1}^D \; 
\sum_{\ell=1}^3 \; 
\int_{\partial \texttt{D}_{\ell}^k} 
\delta_i^{\rm D}
\left(
\{ \varepsilon \widetilde{\nabla} w_i  \}_{\mathbf{n}} 
\cdot \mathcal{J} v_i
\,+ \, 
\eta 
\{  \varepsilon  \mathbf{r}_{\ell,k} ( \mathbf{n}^+ \mathcal{J}_{\rm D}^i(w_i)   ) \}_{\mathbf{n}}
\cdot \mathcal{J} v_i
\,+ \, 
\{ \varepsilon \widetilde{\nabla} v_i  \}_{\mathbf{n}} 
\cdot \mathcal{J}_{\rm D}^i w_i  
\right)
d \mathbf{X}
\right)
}\\[2mm]
\displaystyle{
\hspace{4.7in}
\, - \,
\int_{  \texttt{D}^k} 
\; \varepsilon \; 
\widetilde{\nabla}  w  \cdot  
\widetilde{\nabla}  v
 \; d \mathbf{X},
 }
 \\
\end{array}
\end{equation}
for all $w,v \in \mathcal{X}_{\rm hf}$,
where $\eta=3$, $\delta_i^{\rm D}=\frac{1}{2}$ on interior facets, $\delta_i^{\rm D}=1$ on $\Gamma_{\rm D}^i$ and
$\delta_i^{\rm D}=0$ on 
$\partial \Omega \setminus    \Gamma_{\rm D}^i$.
Then, 
exploiting the consistency of the numerical flux,
and the fact that if $w$ is continuous, $w^+=w^-$
on interior facets, we obtain
\begin{subequations}
\label{eq:fluxes_explained_CG}
\begin{equation}
\label{eq:convection_flux_explained}
r_k^{\rm c}(w,v) = 
\int_{\partial \texttt{D}^k}  \;
\| g \mathbf{G}^{-T} \mathbf{N}  \|_2 \;
v \cdot
\left(  F(\mathcal{D}(w)) \cdot \mathbf{n}  \right) d \mathbf{X} 
\, - \,
\int_{  \texttt{D}^k} 
\widetilde{\nabla}  v\cdot  
F_{\Phi}(  w )
 \; d \mathbf{X} \, 
-
\int_{ \texttt{D}^k} 
v\cdot S_{\Phi}(w) d \mathbf{X},
\end{equation}
and
\begin{equation}
\label{eq:diffusion_flux_explained}
\begin{array}{rl}
\displaystyle{r_k^{\rm d}(w,v) = }
&
\displaystyle{  
\left(
\sum_{i=1}^D \; 
\sum_{\ell=1}^3 \; 
\int_{\partial \texttt{D}_{\ell}^k \cap  \Gamma_{\rm D}^i}  
\delta_i^{\rm D}
\left(
\left(
  \varepsilon \widetilde{\nabla} w_i   
\,+ \, 
\eta 
  \varepsilon  \mathbf{r}_{\ell,k} ( \mathbf{n}  \mathcal{J}_{\rm D}^i w_i   )  
\right)  \cdot \mathbf{n}   v_i
\,+ \, 
\{ \varepsilon \widetilde{\nabla} v_i  \}_{\mathbf{n}} 
\cdot \mathcal{J}_{\rm D}^i w_i    
\right)
d \mathbf{X}
\right)
}
\\
&
\displaystyle{
\, - \,
\int_{  \texttt{D}^k} 
\; \varepsilon \; 
\widetilde{\nabla}  w  \cdot  
\widetilde{\nabla}  v
 \; d \mathbf{X}.
}
\\
\end{array}
\end{equation}
Note that the computation of
$r_k^{\rm c}(w,v) $ and 
$r_k^{\rm d}(w,v) $ in 
 \eqref{eq:fluxes_explained_CG}
 requires    the knowledge of $w,v$ only in $\texttt{D}^k$ and can be performed using element-wise residual evaluation routines implemented in many DG codes.
\end{subequations}

\section{Approximate minimum residual: analysis of the linear case}
\label{sec:AMR_linear_theory}
We study the performance of the approximate minimum residual (AMR) formulation for linear inf-sup stable problems:
\begin{equation}
\label{eq:variational_problem_linear}
{\rm find} \; u^{\star} \in \mathcal{X} :\; \;
A(u^{\star}, v) = F(v)
\qquad
\forall \, v \in \mathcal{Y},
\end{equation}
where
$(\mathcal{X}, \|  \cdot \| = \sqrt{(\cdot,\cdot)}  )$ and
$(\mathcal{Y}, \vertiii{\cdot}   = \sqrt{((\cdot,\cdot))}  )$  are suitable Hilbert spaces, and $A$ and $F$ are a bilinear and a linear form,
 $A \in \mathfrak{L}(\mathcal{X} ,  \mathcal{Y}')$, $F \in \mathcal{Y}'$. 
We denote by $\gamma$ and $\beta$ the continuity and inf-sup constants associated with the form $A$:
$$
\beta = 
\inf_{w\in \mathcal{X}\setminus \{ 0 \}   } \,\sup_{v\in \mathcal{Y}\setminus \{ 0 \}   }
\frac{A(w, v)}{\| w  \|  \vertiii{v}  },
\quad
\gamma = 
\sup_{w\in \mathcal{X}\setminus \{ 0 \}   } \,\sup_{v\in \mathcal{Y}\setminus \{ 0 \}   }
\frac{A(w, v)}{\| w  \|  \vertiii{v}  }.
$$ 
 Given the $N$-dimensional space $\mathcal{Z}_N \subset \mathcal{X}$ and 
 the $J$-dimensional space   $\mathcal{Y}_J \subset \mathcal{Y}$, $J \geq N$, we define the AMR statement:
 \begin{equation}
\label{eq:least_squares_ROM}
\hat{u} = {\rm arg} \min_{u \in \mathcal{Z}_N} \| A(u, \cdot) - F  \|_{\mathcal{Y}_J'}
:=
\sup_{v \in \mathcal{Y}_J \setminus \{ 0 \}} \frac{A(u, v) - F(v) }{
\vertiii{v}}.
\end{equation}
Note that  AMR reduces to   Galerkin   for $\mathcal{Y}_J = \mathcal{Z}_N$, while AMR reduces to minimum residual for $\mathcal{Y}_J = \mathcal{Y}$.

 In view of the analysis, 
we introduce the reduced inf-sup constant
 \begin{equation}
\label{eq:inf_sup_ROM}
\beta_{N,J} =  \inf_{w\in \mathcal{Z}_N\setminus \{ 0 \}   } \,\sup_{v\in \mathcal{Y}_J \setminus \{ 0 \}   }
\frac{A(w, v)}{\| w  \| \vertiii{v}   };
\end{equation} 
  furthermore,  we introduce the supremizing operator
 $S: \mathcal{Z}_N \to \mathcal{Y}$ such that
$S(\zeta) = \texttt{R}_{\mathcal{Y}} A(\zeta, \cdot) $, that is
$$
(( S(\zeta) , \, v ))
\,=\,
A(\zeta, v) \;\;
\forall \; 
v\in \mathcal{Y}, 
$$ 
 and the constant
 \begin{equation}
\label{eq:infsup_test}
\delta_{N,J}^{\rm test} = 
\inf_{ s \in  \mathcal{Y}_N^{\rm opt}   } \, 
\sup_{ v \in \mathcal{Y}_J    } \,
\frac{  ((s,v))   }{ \vertiii{s}  \vertiii{v}    }
\end{equation}
 where $\mathcal{Y}_N^{\rm opt}  = \{  S(\zeta) : \zeta \in  \mathcal{Z}_N \}$. Note that $\mathcal{Y}_N^{\rm opt}$ is the linear counterpart of the space in \eqref{eq:optimal_test_space}:  the constant $\delta_{n,m}^{\rm test} $ measures the proximity between the test space $\mathcal{Y}_J$ and the optimal test space 
 $\mathcal{Y}_N^{\rm opt} $.

 Next Proposition contains the key results of this section. In particular, we observe that the performance depends on the behavior of 
 $\delta_{N,J}^{\rm test} $ and thus on the  
 proximity between   $\mathcal{Y}_J$ and 
 $\mathcal{Y}_N^{\rm opt} $: this motivates the sampling strategy in Algorithm  \ref{alg:test_space}.

\begin{proposition}
\label{th:properties_AMR}
If $\beta, \beta_{N,J}>0$, the solution   $\hat{u}$ to \eqref{eq:least_squares_ROM} exists and is  unique. Furthermore, the following hold:
\begin{subequations}
\begin{equation}
\label{eq:stability_AMR}
\| \hat{{u}} \| \leq \frac{1}{\beta_{N,J}} \, \| F \|_{\mathcal{Y}'};
\end{equation}
\begin{equation}
\label{eq:error_AMR}
\| \hat{u}  -  u^{\star} \|  \leq
\frac{\gamma}{\delta_{N,J}^{\rm test}  \beta} \,  
 \inf_{u\in \mathcal{Z}_N} \, \| u -  u^{\star} \|.
\end{equation}
\end{subequations}
\end{proposition}

\begin{proof}
We first observe that any solution to \eqref{eq:least_squares_ROM}  satisfies
(the proof is straightforward):
\begin{equation}
\label{eq:AMR_petrov_gal}
 {\rm find} \; \hat{u}  \in \mathcal{Z}_N :\; \;
A(\hat{u}, v) = F(v)
\qquad
\forall \, v \in \mathcal{Y}_{N,J}:= {\rm span} \{  \phi_n^{J}     \}_{n=1}^N,
\end{equation}
where 
$\phi_n^{J}$ satisfies $((\phi_n^{J}, v  )) = A(\zeta_n, v)$ for all $v \in \mathcal{Y}_J$, 
$n=1,\ldots,N$. Then, we observe that $A(\zeta, \eta) = A(\zeta, \Pi_{ \mathcal{Y}_{N,J}   }^{\mathcal{Y}} \eta  )$
for all $\zeta \in \mathcal{Z}_N$ and
$\eta \in \mathcal{Y}_J$, where $\Pi_{ \mathcal{Y}_{N,J}   }^{\mathcal{Y}}: \mathcal{Y} \to \mathcal{Y}_{N,J}$ denotes the projection operator on $\mathcal{Y}_{N,J} $ with respect to the $\mathcal{Y}$ norm. As a result, we find that the inf-sup constant $\beta_{N,J}^{\star} $  associated with \eqref{eq:AMR_petrov_gal}  satisfies
$$
\beta_{N,J}^{\star} = 
\inf_{w\in \mathcal{Z}_N\setminus \{ 0 \}   } \,\sup_{v\in \mathcal{Y}_{N,J} \setminus \{ 0 \}   }
\frac{A(w, v)}{\| w  \|  \vertiii{v}  }
= 
\inf_{w\in \mathcal{Z}_N\setminus \{ 0 \}   } \,\sup_{v\in \mathcal{Y}_J \setminus \{ 0 \}   }
\frac{A(w, v)}{\| w  \|  \vertiii{v}  }
=
\beta_{N,J}.
$$
In conclusion, exploiting a standard argument for inf-sup stable problems, we find that the solution $\hat{u}$ to \eqref{eq:AMR_petrov_gal} --- and thus  the solution to \eqref{eq:least_squares_ROM} ---   is unique and 
$\| \hat{{u}} \| \leq \frac{1}{\beta_{N,J}} \, \| F \|_{\mathcal{Y}'}$, which is \eqref{eq:stability_AMR}.

In order to prove  \eqref{eq:error_AMR}, 
we first exploit the argument in \cite{xu2003some} to show that
$$
\| \hat{u}  -  u^{\star} \|  \leq
\frac{\gamma}{  \beta_{N,J}} \,  
 \inf_{u\in \mathcal{Z}_N} \, \| u -  u^{\star} \|.
$$
Towards this end, we define the operator 
$\mathfrak{S}: \mathcal{X} \to \mathcal{Z}_N$, 
$ \mathfrak{S} (w) : =  {\rm arg} \min_{\zeta \in \mathcal{Z}_N} \|A(\zeta - w, \cdot)  \|_{\mathcal{Y}_J'}$.
Note that 
$\hat{u} = \mathfrak{S}(u^{\star})$.
 Clearly, $\mathfrak{S}(\zeta) = \zeta$ for all $\zeta \in \mathcal{Z}_N$: this implies that $\mathfrak{S}$ is idempotent, that is
$\mathfrak{S}(\mathfrak{S}(w)) = \mathfrak{S}(w)$ for all $w \in \mathcal{X}$.
Therefore, exploiting a standard  result in Functional Analysis (see, e.g.,  \cite{szyld2006many}), 
we have $\| \mathfrak{S}  \|_{\mathfrak{L}(\mathcal{X}, \mathcal{X})}  =
\| \mathbbm{1} - \mathfrak{S} \|_{\mathfrak{L}(\mathcal{X}, \mathcal{X})} $.
 Furthermore, recalling  \eqref{eq:stability_AMR},   we find
$$
\| \mathfrak{S}(w) \| \leq
\frac{1}{\beta_{N,J}} \| A(w, \cdot)\|_{\mathcal{Y}'}
\leq
\frac{\gamma}{\beta_{N,J}} \| w \|
\; \Rightarrow \;
\| \mathfrak{S}  \|_{\mathfrak{L}(\mathcal{X}, \mathcal{X})}
\leq
\frac{\gamma}{\beta_{N,J}}.
$$
In conclusion, we obtain, for any $\zeta \in \mathcal{Z}_N$, 
$$
\| u^{\star}  - \hat{u}       \| = 
\| ( \mathbbm{1} - \mathfrak{S}  )  u^{\star}         \|
=
\| ( \mathbbm{1} - \mathfrak{S}  )  (u^{\star}  - \zeta)       \|
\leq
\| \mathbbm{1} - \mathfrak{S} \|_{\mathfrak{L}(\mathcal{X}, \mathcal{X})}
\| u^{\star}  - \zeta      \|
\leq
\frac{\gamma}{\beta_{N,J}} \; \| u^{\star}  - \zeta     \|,
$$
which is the desired result. Note that in the second identity we used the fact that
$ ( \mathbbm{1} - \mathfrak{S} )  \zeta = 0$ for all $\zeta \in \mathcal{Z}_N$.

It remains to prove that
$\beta_{N,J} \geq 
\delta_{N,J}^{\rm test}  \beta$. 
Recalling   the definition of the supremizing operator $S$ and the projection theorem, we find
$$
\left\{
\begin{array}{ll}
\displaystyle{
\sup_{v \in \mathcal{Y}_J}  
  \frac{A(\zeta,v)}{   \vertiii{v}    }
  =
  \sup_{v \in \mathcal{Y}_J}  
  \frac{((  S(\zeta), v ))
   }{  \vertiii{v}   } 
  =
\vertiii{  
\Pi_{\mathcal{Y}_J}^{\mathcal{Y}}   S (\zeta) }}
 &
  \forall \, \zeta \in \mathcal{Z}_N;
\\[3mm]
\displaystyle{
\vertiii{
\Pi_{\mathcal{Y}_J}^{\mathcal{Y}}  s
} 
  \geq \delta_{N,J}^{\rm test} \; 
  \vertiii{s}
   }
&
  \forall \, s \in \mathcal{Y}_{N}^{\rm opt};
\\[3mm]
\displaystyle{
  \vertiii{S (\zeta)}
  \geq \beta  \; 
   \|  \zeta  \| 
   }
&
  \forall \, \zeta \in \mathcal{Z}_N.
\\
\end{array}
\right.
$$
Then, exploiting the 
previous estimates,   we find
$$
 \sup_{v \in \mathcal{Y}_J}  
  \frac{A(\zeta,v)}{    \vertiii{v} }
  =
    \vertiii{ 
 \Pi_{\mathcal{Y}_J}^{\mathcal{Y}}   S (\zeta)  }
  \geq
\delta_{N,J}^{\rm test}   
   \vertiii{ 
S (\zeta)  }
  \geq
  \beta \, 
\delta_{N,J}^{\rm test}     \|   \zeta \| , \qquad
\forall \, \zeta  \in \mathcal{Z}_N,
$$
which is the desired result.
\end{proof}

\section{Derivation of the accuracy constraints \eqref{eq:accuracy_constraint}}
\label{sec:BRR_EQ}
We illustrate  how to apply the Brezzi-Rappaz-Raviart (BRR, \cite{brezzi1980finite,caloz1997numerical}) theory to estimate the error between the solution $\widehat{U}^{\rm hf}$ to \eqref{eq:approx_minresROM_temp} and the solution 
$\widehat{U}$ to \eqref{eq:approx_minresROM},
$E^{\rm eq} = \|  \widehat{U}^{\rm hf} - \widehat{U} \| = 
\| \widehat{\boldsymbol{\alpha}}^{\rm hf} -  \widehat{\boldsymbol{\alpha}} \|_2$. 
We omit the dependence on $\mu$ for notational brevity. First, we present the following Lemma (see \cite[Lemma 3.1]{yano2019discontinuous}).

\begin{lemma}
\label{th:brr_theory}
We introduce the $C^1$ function $\boldsymbol{\mathcal{N}}: \mathbb{R}^N \to \mathbb{R}^N$, 
${\boldsymbol{\alpha}} \in \mathbb{R}^N$ such that the Jacobian $D \boldsymbol{\mathcal{N}}({\boldsymbol{\alpha}}) \in \mathbb{R}^{N,N}$ is non-singular, and constants $\epsilon,\gamma$ and $L(r)$ such that
\begin{equation}
\label{eq:BRR_hypotheses}
\|  \boldsymbol{\mathcal{N}} ( {\boldsymbol{\alpha}} ) \|_2 \leq \epsilon, \quad
\|  D \boldsymbol{\mathcal{N}} ( {\boldsymbol{\alpha}} )^{-1} \|_2 \leq \gamma, \quad
\sup_{\mathbf{w} :  \|\mathbf{w} - \boldsymbol{\alpha} \|_2 \leq r } \; \;
\|  D \boldsymbol{\mathcal{N}} ( \mathbf{w} )   -
D \boldsymbol{\mathcal{N}} ( {\boldsymbol{\alpha}} )   
 \|_2
\leq L(r).
\end{equation} 
Suppose that $2 \gamma L(2 \gamma \epsilon) \leq 1$. Then, for all $\beta \geq 2 \gamma \epsilon$ such that $\gamma L( \beta) < 1$, there exists a unique solution ${\boldsymbol{\alpha}}^{\star} $ that satisfies $\boldsymbol{\mathcal{N}} ( {\boldsymbol{\alpha}}^{\star}  ) = \mathbf{0}$ in the ball of radius $\beta$ centered in $\boldsymbol{\alpha}$. Furthermore, we have
\begin{equation}
\label{eq:BRR_estimate}
\| \boldsymbol{\alpha}^{\star} - \boldsymbol{\alpha}  \|_2
\leq
2 \gamma  \| \boldsymbol{\mathcal{N}} ( {\boldsymbol{\alpha}}^{\star}  )  \|_2.
\end{equation}
\end{lemma}

We apply Lemma \ref{th:brr_theory} to analyze the quadrature error $E^{\rm eq}$. Towards this end, we define 
\begin{equation}
\label{eq:BRR_specialized}
\boldsymbol{\mathcal{N}} ( {\boldsymbol{\alpha}} )
=
\frac{1}{2} \nabla \; \|  \mathbf{R}_{N,J}^{\rm hf}( \boldsymbol{\alpha}  )    \|_2^2
=
 \left( 
\mathbf{J}_{N,J}^{\rm hf}( \boldsymbol{\alpha}  )
\right)^T 
\mathbf{R}_{N,J}^{\rm hf}( \boldsymbol{\alpha}  ).
\end{equation}
Clearly, any $\boldsymbol{\alpha}^{\star}$ satisfying 
$\boldsymbol{\mathcal{N}} ( {\boldsymbol{\alpha}}^{\star} )
= \mathbf{0}$ is a stationary point of the objective function in \eqref{eq:approx_minresROM_temp}: as a result, if $\widehat{\boldsymbol{\alpha}}$ satisfies the hypotheses of Lemma \ref{th:brr_theory} with 
$\boldsymbol{\mathcal{N}}$ as in \eqref{eq:BRR_specialized},
there exists a unique solution $\widehat{\boldsymbol{\alpha}}^{\star}$ such that 
$\boldsymbol{\mathcal{N}} ( {\boldsymbol{\alpha}}^{\star} )
= \mathbf{0}$ in a neighborhood of $\widehat{\boldsymbol{\alpha}}$ and
$
\|  \widehat{\boldsymbol{\alpha}} - \widehat{\boldsymbol{\alpha}}^{\star}\|_2
\leq 2 \gamma 
\| 
\boldsymbol{\mathcal{N}} (
\widehat{\boldsymbol{\alpha}} 
)  \|_2$.

Since
$\mathbf{J}_{N,J}^{\rm eq}( \widehat{\boldsymbol{\alpha}}  )^T 
\mathbf{R}_{N,J}^{\rm eq}( \widehat{\boldsymbol{\alpha}}  )
=
\mathbf{0}$, by straightforward manipulations, we find that
$$
\|  \boldsymbol{\mathcal{N}} (\widehat{\boldsymbol{\alpha}} )  \|_2
\leq
\underbrace{
\|
\mathbf{J}_{N,J}^{\rm hf}( \widehat{\boldsymbol{\alpha}}  )
-
\mathbf{J}_{N,J}^{\rm eq}( \widehat{\boldsymbol{\alpha}}  )
\|_2
}_{=: \rm (I)}
\|
\mathbf{R}_{N,J}^{\rm eq}( \widehat{\boldsymbol{\alpha}}  )
\|_2
\; + \;
\underbrace{
\|
\mathbf{J}_{N,J}^{\rm hf}( \widehat{\boldsymbol{\alpha}}  )^T
\left(
\mathbf{R}_{N,J}^{\rm hf}( \widehat{\boldsymbol{\alpha}}  )
-
\mathbf{R}_{N,J}^{\rm eq}( \widehat{\boldsymbol{\alpha}}  )
\right)
\|_2
}_{=: \rm (II)}.
$$
This estimate shows that the residual $\|  \boldsymbol{\mathcal{N}} (\widehat{\boldsymbol{\alpha}} )  \|_2$ is controlled by the quadrature errors
(I) and (II). 
Note that (II) corresponds to  the accuracy constraint \eqref{eq:accuracy_constraint}; on the other hand, we choose to exclude the constraints associated with the Jacobian. The reason is twofold: first, controlling (I) requires $N J n_{\rm train}$ additional constraints and is thus expensive for offline calculations;  second, (I) is multiplied by  the empirical residual
$\|
\mathbf{R}_{N,J}^{\rm eq}( \widehat{\boldsymbol{\alpha}}  )
\|_2$, which is expected to be small for $J = \mathcal{O}(N)$.

\section{Further investigations on data compression}
\label{sec:data_compression_vis}
\subsection{Burgers equation}

We investigate the compressibility of the manifold
$\mathcal{M}_{\rm space} = \{ U_{\mu}(t) : t \in (0,T) , \; \mu \in \mathcal{P} \} \subset L^2(0,L)$, which needs to be  approximated  in  time-marching ROMs. 
Towards this end, we assess performance of  POD in the unregistered and in the registered case; for simplicity, we here restrict ourselves to the case $\mathcal{P}  = \{ \bar{\mu} \}$, $\bar{\mu}  = [1,0.25]$.
Figure     \ref{fig:ALE_burgers_unreg}  shows the behavior of $U_{\mu}(t) $ and the projection
$\Pi_{\mathcal{Z}_N}  U_{\mu}(t)$ for two time instants. Here, 
$\mathcal{Z}_N$ is the $N=20$-dimensional POD space built based on $n_{\rm train}=200$ temporal snapshots associated with 
the equispaced sampling times $\{ t_{\rm s}^k \}_{k=1}^{n_{\rm train}}$. As expected,  linear methods are extremely inefficient to capture shock waves: the projection error is indeed significant despite the relatively-large  number of retained modes.

\begin{figure}[h!]
\centering
 \subfloat[$t=0.05$] 
{  \includegraphics[width=0.4\textwidth]
 {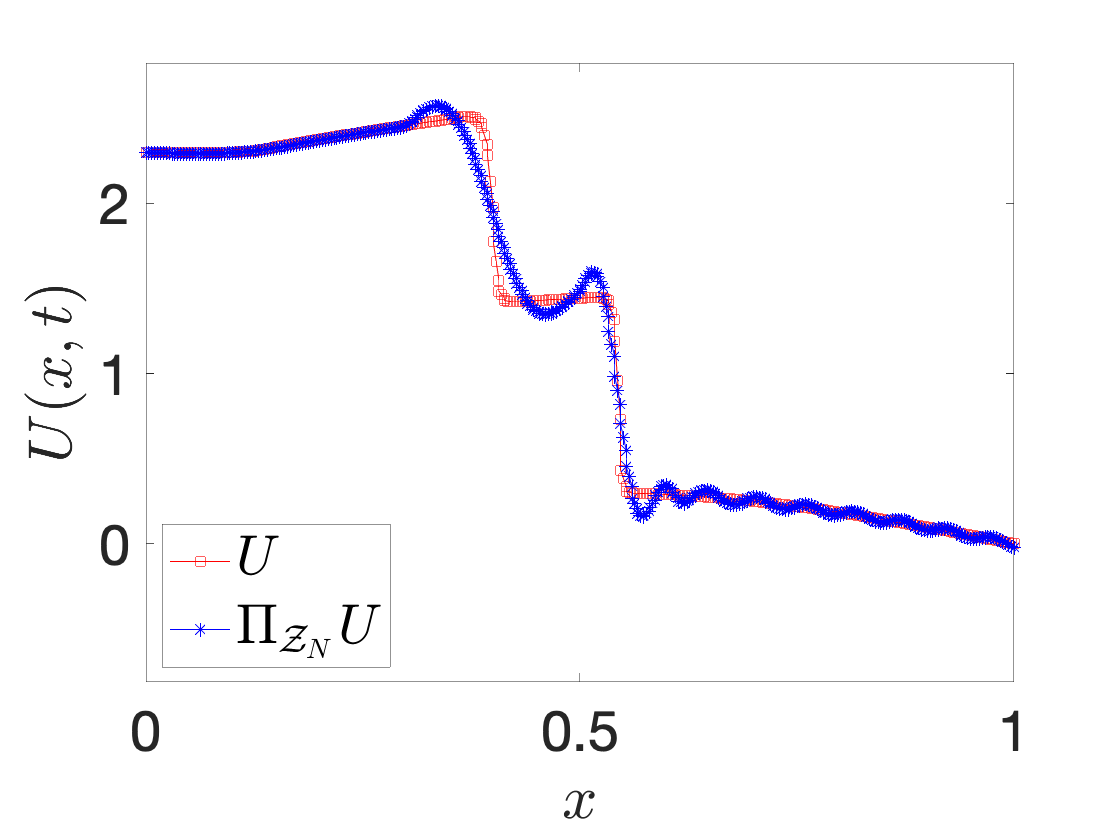}}
  ~~
 \subfloat[$t=0.3$] 
{  \includegraphics[width=0.4\textwidth]
 {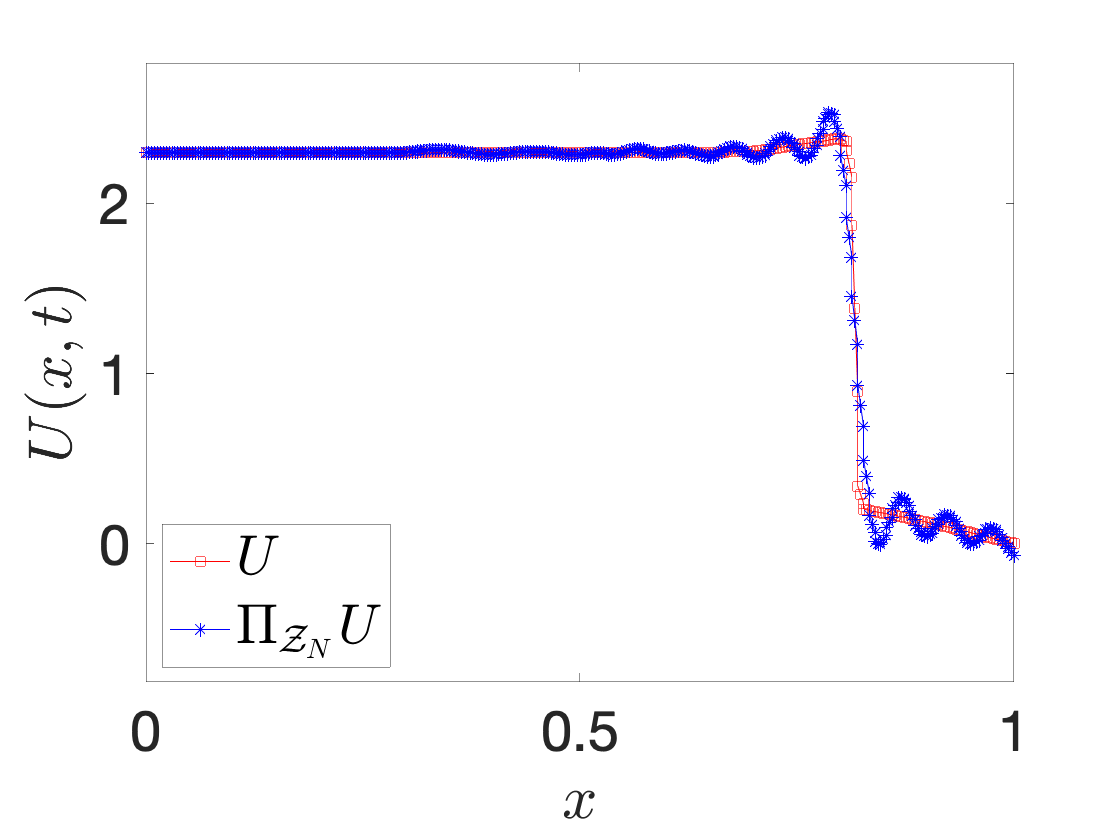}}
 
\caption{Burgers equation; performance of  spatial linear compression for $\mu=[1,0.25]$ (unregistered case, $N=20$). }
\label{fig:ALE_burgers_unreg}
\end{figure}

In Figures \ref{fig:ALE_burgers_reg},   \ref{fig:ALE_burgers_reg2} and \ref{fig:ALE_burgers_reg3}, 
we investigate performance of spatial registration.
 Here, the mapping is generated based on  
 $n_{\rm train}=200$ snapshots through Algorithm \ref{alg:registration} with $N_{\rm max}=5$, $\xi=10^{-2}$, $M_{\rm hf}=100$, $tol_{\rm pod}=10^{-4}$.
 We further consider $\mathcal{T}_{N_0=2} = {\rm span} \{ U_{\mu}(0),  U_{\mu}(T)\}$  as  initial template space; the resulting map consists of a four-term expansion ($M=4$). 
  In Figure \ref{fig:ALE_burgers_reg}, we show 
the behavior of $\widetilde{U}_{\mu}(t) $ and the projection $\Pi_{\mathcal{Z}_N} \widetilde{U}_{\mu}(t)$, 
 where $\mathcal{Z}_N$ is the $N=20$-dimensional POD space built based on  the  mapped snapshots.
 In Figure \ref{fig:ALE_burgers_reg2}(a), we compare the behavior of the normalized POD eigenvalues with and without registration; similarly, in Figure \ref{fig:ALE_burgers_reg2}(b),
 we show the in-sample projection error 
$E^{\rm bf,\infty} =
\max_{j=1,\ldots,n_{\rm train}}
\; E^{\rm bf}(t_{\rm s}^k),$
for registered and unregistered configurations 
(cf. \eqref{eq:best_fit_error}).
 We observe that registration improves performance of POD for this model problem. 
 
Figure \ref{fig:ALE_burgers_reg3}, which depicts the behavior of the physical and mapped solution for two time steps, shows that the mapping has the effect of ``squeezing" the transition from one shock to zero shock by artificially increasing the wave speed. In the framework of projection-based ROMs, this poses serious issues for the numerical temporal integration.

\begin{figure}[h!]
\centering
 \subfloat[$t=0.05$] 
{  \includegraphics[width=0.4\textwidth]
 {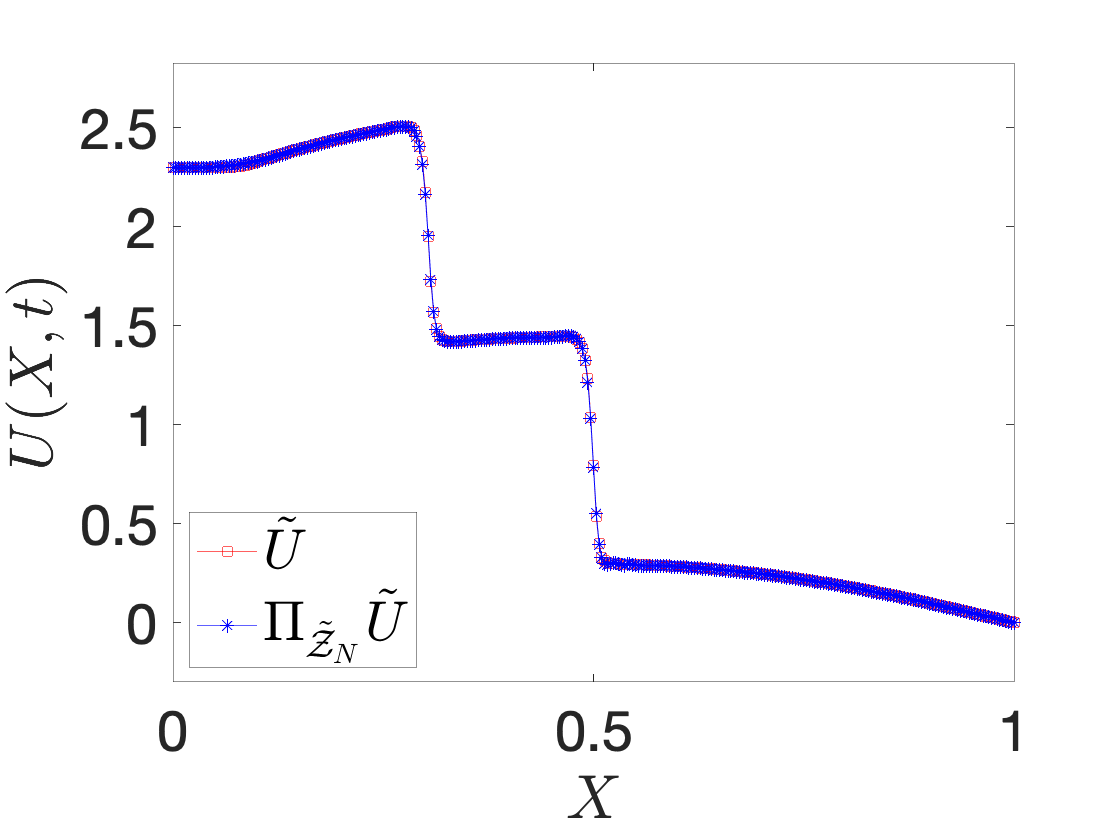}}
  ~~
 \subfloat[$t=0.3$] 
{  \includegraphics[width=0.4\textwidth]
 {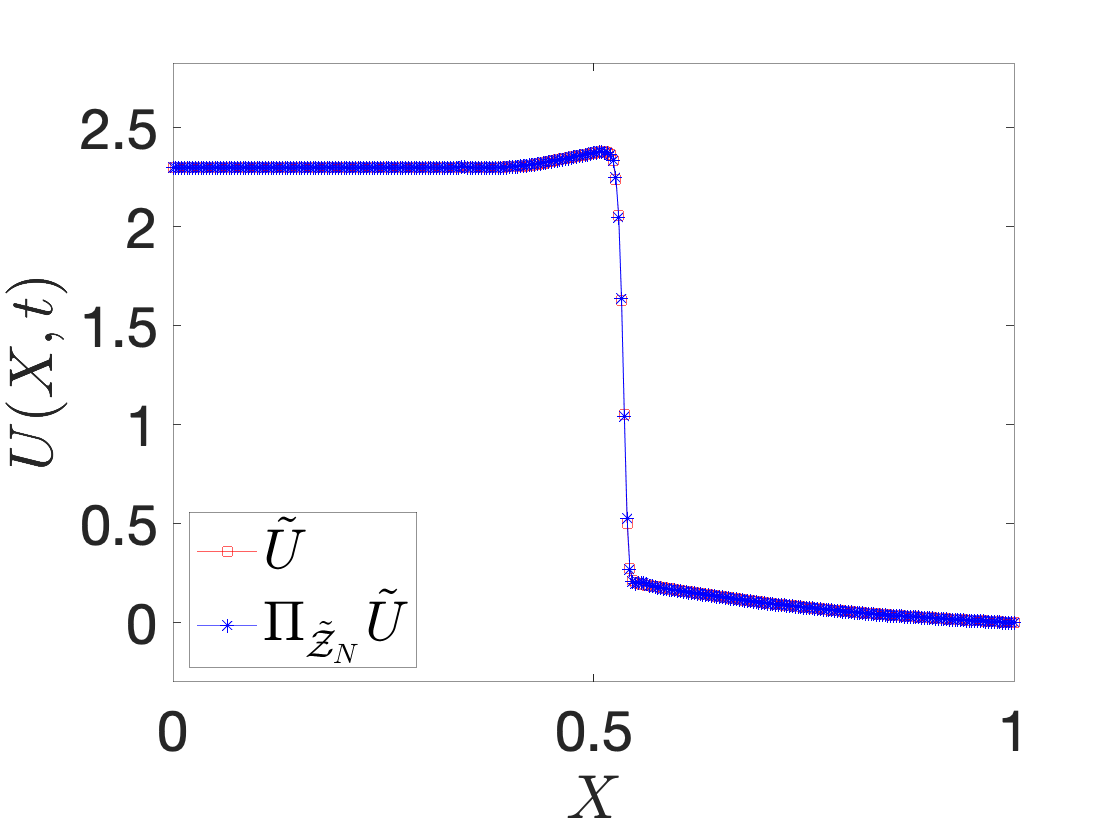}}
 
 \caption{Burgers equation; performance of  spatial  compression for $\mu=[1,0.25]$ with spatial registration ($N=20$). }
 \label{fig:ALE_burgers_reg}
  \end{figure}  
  
\begin{figure}[h!]
\centering
 \subfloat[] 
{  \includegraphics[width=0.4\textwidth]
 {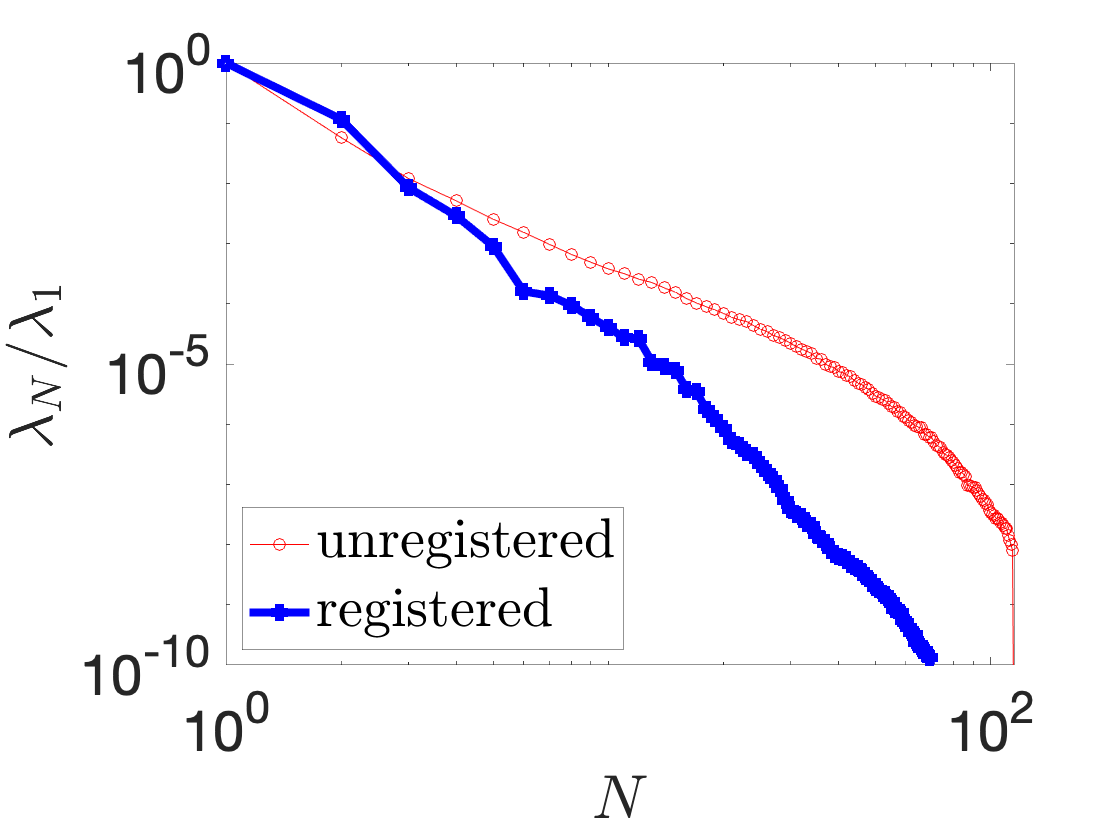}}
  ~~
 \subfloat[] 
{  \includegraphics[width=0.4\textwidth]
 {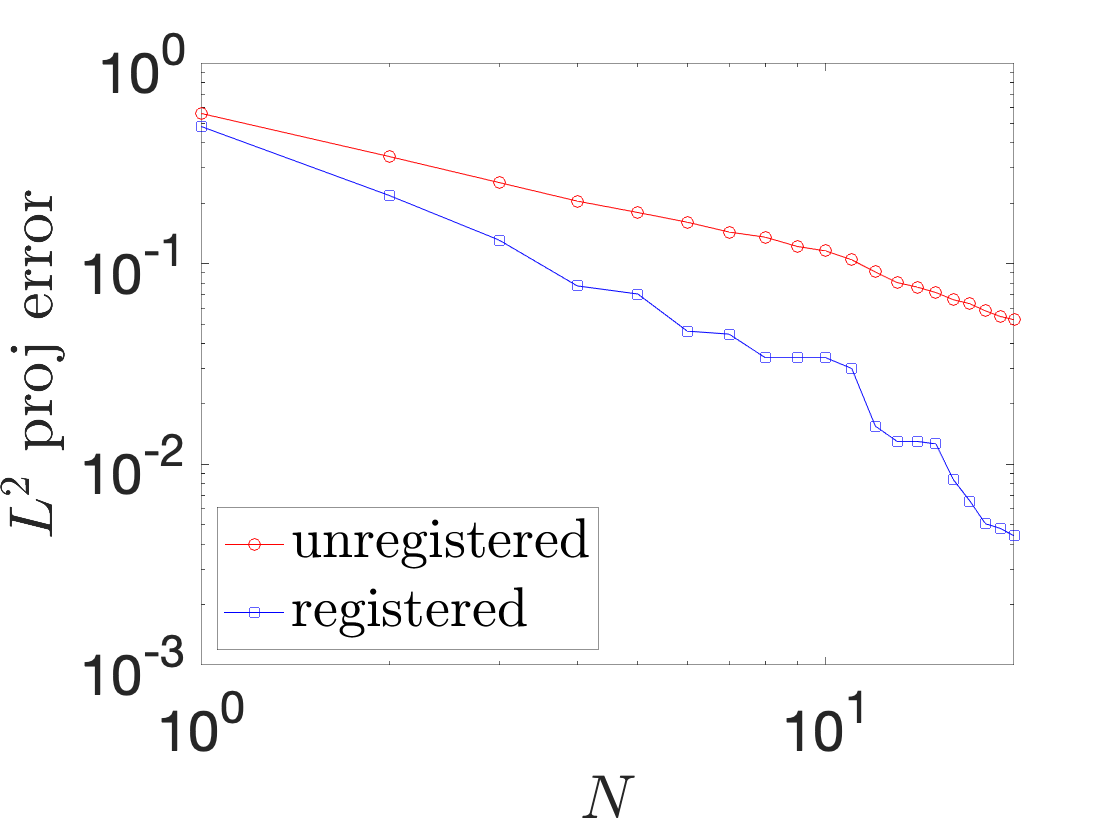}}
  
 \caption{Burgers equation;  compression for $\mu=[1,0.25]$ with spatial registration.
 (a)  behavior of normalized POD eigenvalues associated with the unregistered and registered temporal snapshots.
 (b)  behavior of the maximum in-sample  projection error $E^{\rm bf, \infty}$.
  }
 \label{fig:ALE_burgers_reg2}
  \end{figure}  
  
\begin{figure}[h!]
\centering
  \subfloat[$t=0.4$] 
{  \includegraphics[width=0.4\textwidth]
 {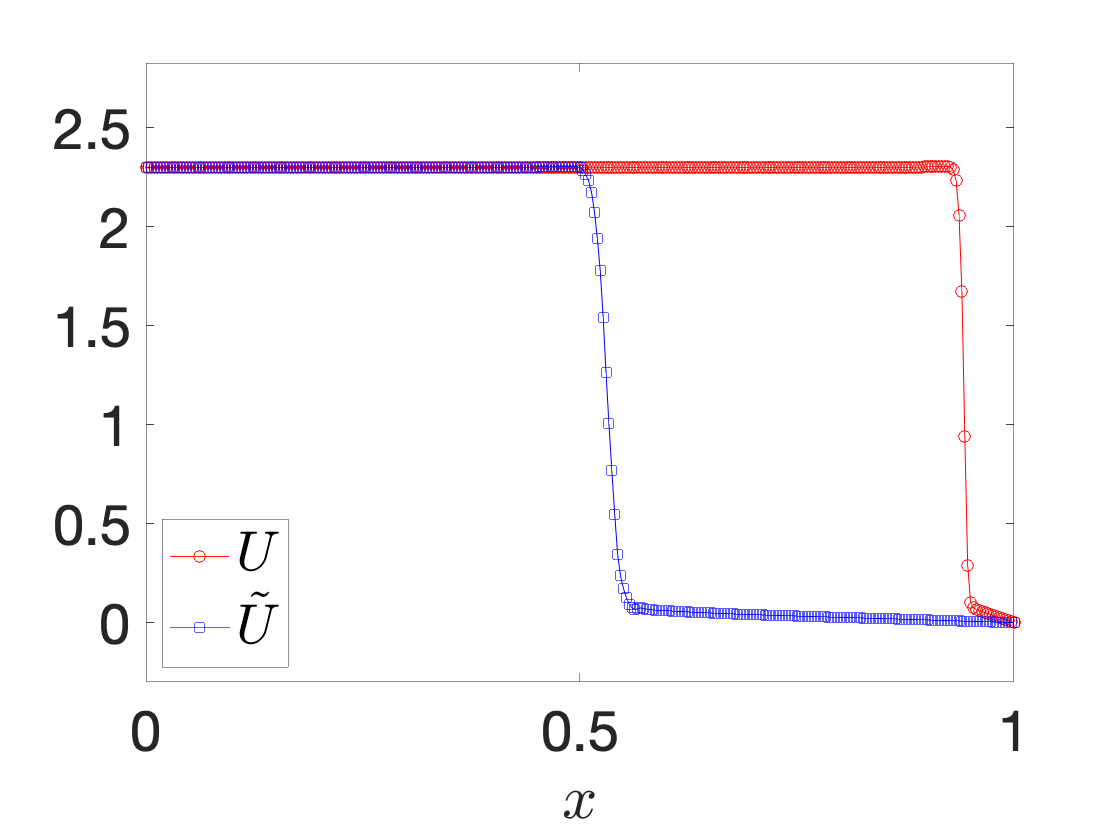}}
  ~~
 \subfloat[$t=0.45$] 
{  \includegraphics[width=0.4\textwidth]
 {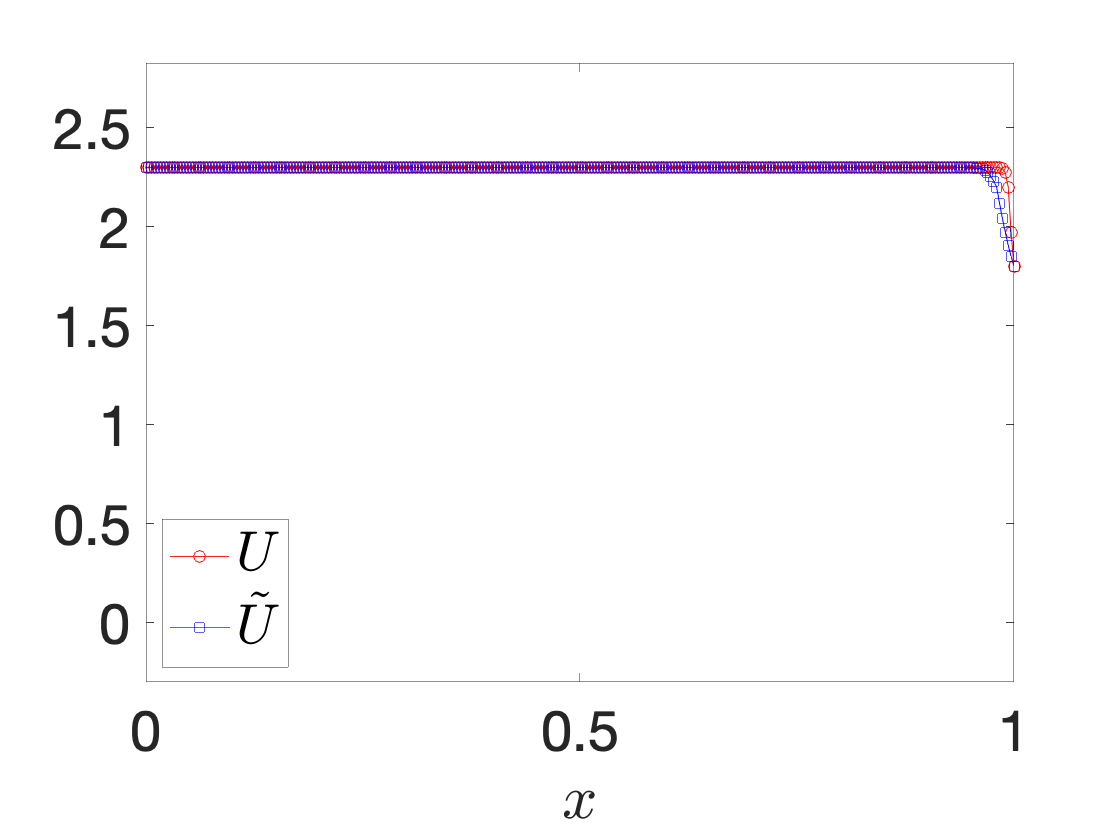}} 
 
 \caption{Burgers equation;
spatial  compression for $\mu=[1,0.25]$ with spatial registration. 
Behavior  of $U_{\mu}$ and  $\widetilde{U}_{\mu}$ for two time instants. }
 \label{fig:ALE_burgers_reg3}
  \end{figure}  
  
In Figure \ref{fig:spacetimereg_burgers2}, we show the unregistered and registered solution fields for two values of the parameter and for three 
horizontal slices of $\Omega$. 
In the unregistered case, these slices correspond to the solution for three time instants.
  We observe that  space-time registration is 
 able to nearly  ``freeze" the position of the jump discontinuities with respect to parameter. These results suggest that the self-similar structures of the present problem  can only be captured by considering the space-time behavior of the solution field.

\begin{figure}[h!]
\centering
\subfloat[$t=0.05$] 
{  \includegraphics[width=0.32\textwidth]
 {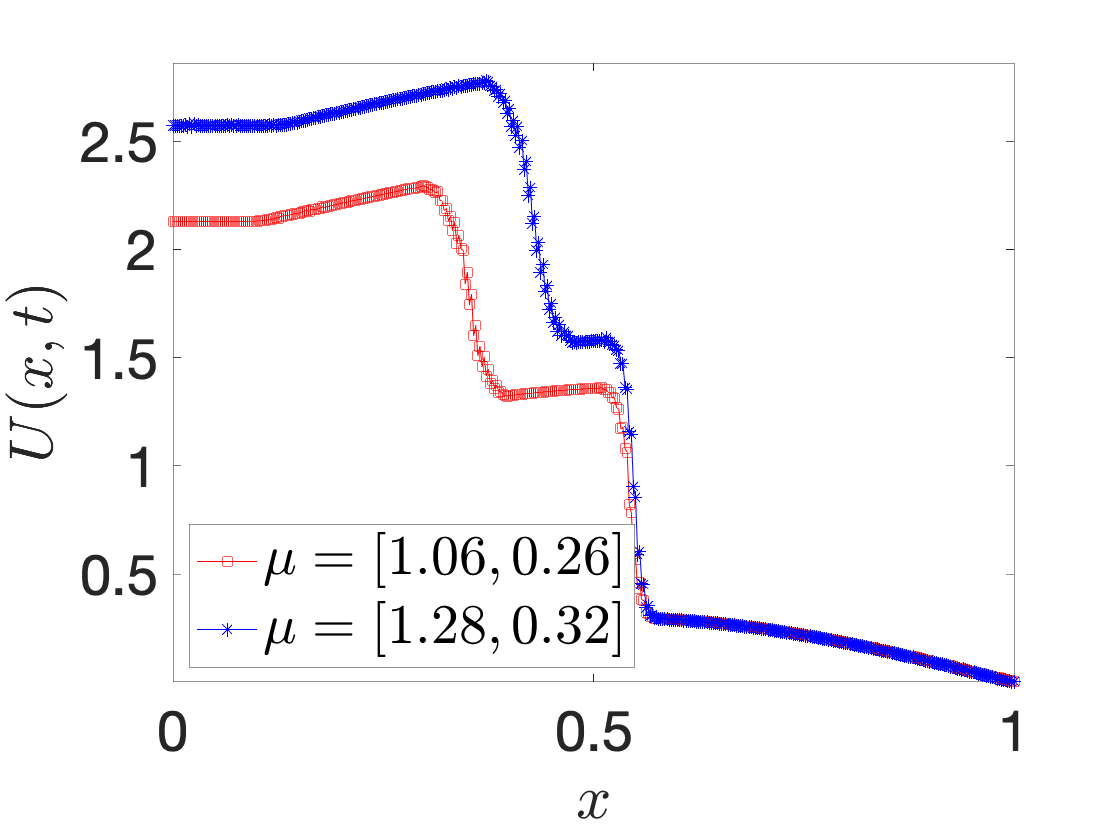}}
  ~~
 \subfloat[$t=0.3$] 
{  \includegraphics[width=0.32\textwidth]
 {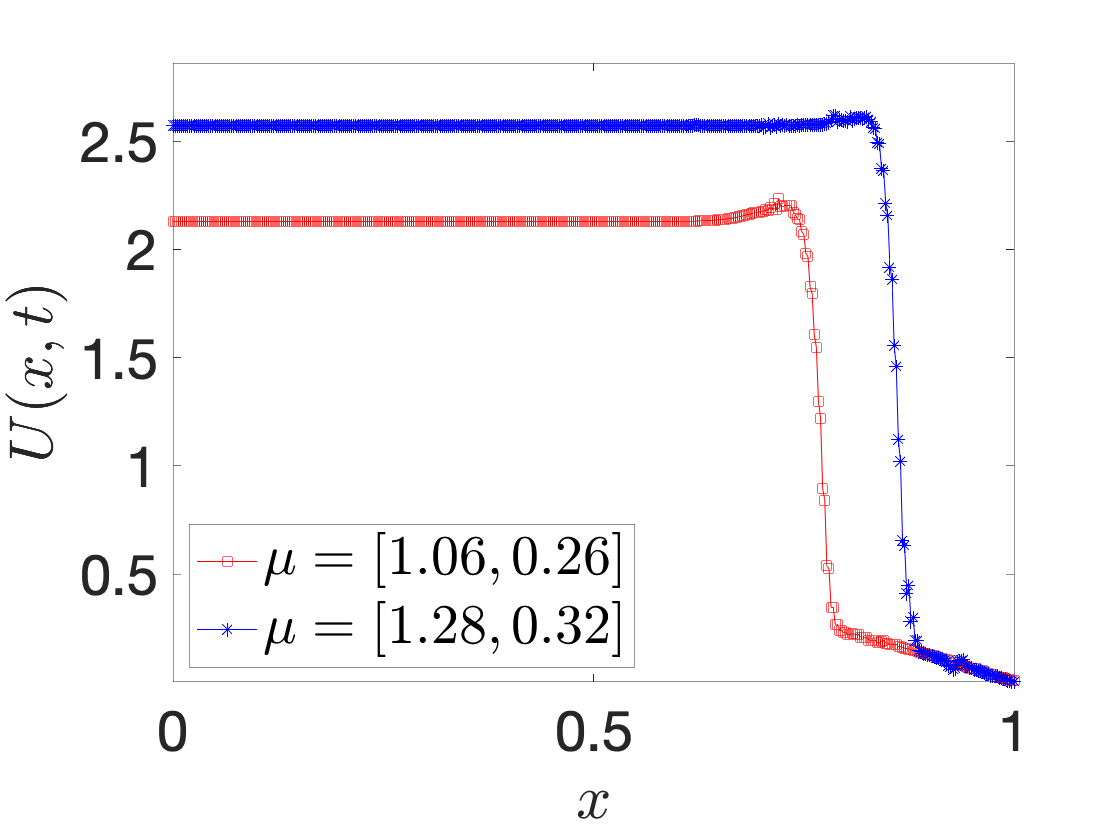}}
   ~~
 \subfloat[$t=0.8$] 
{  \includegraphics[width=0.32\textwidth]
 {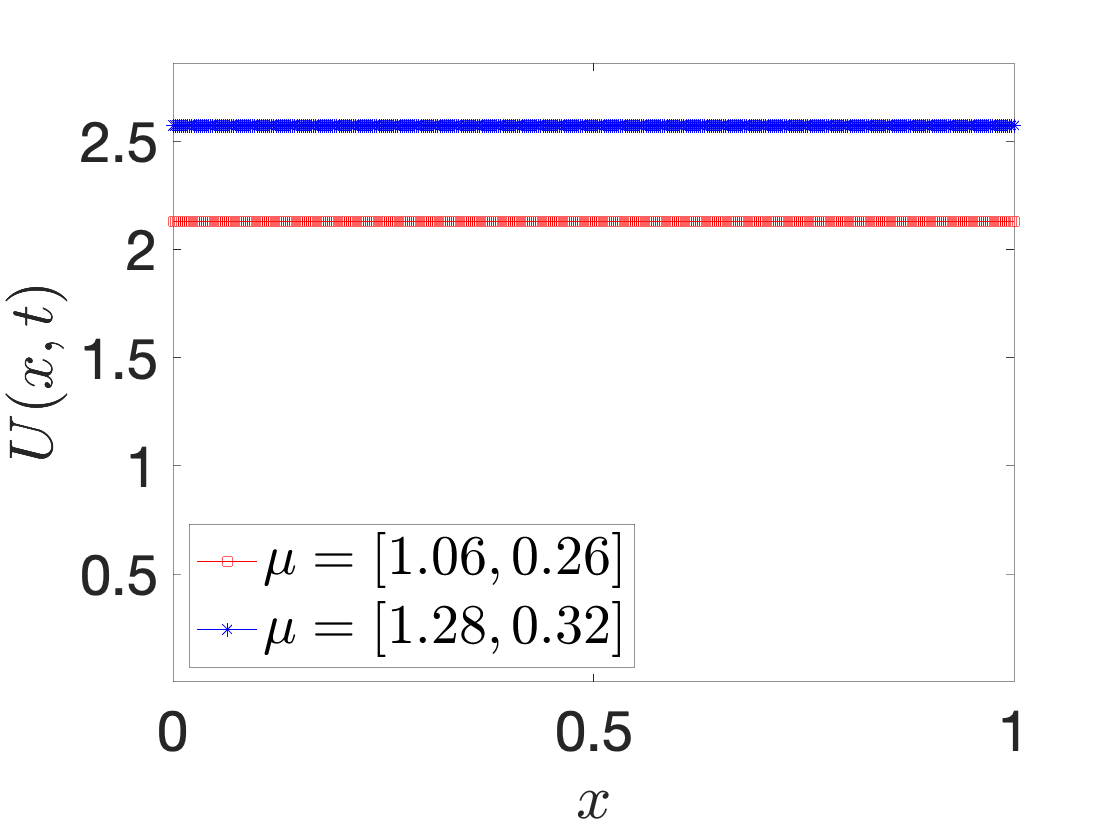}}
 
\subfloat[$X_2=0.05$] 
{  \includegraphics[width=0.32\textwidth]
 {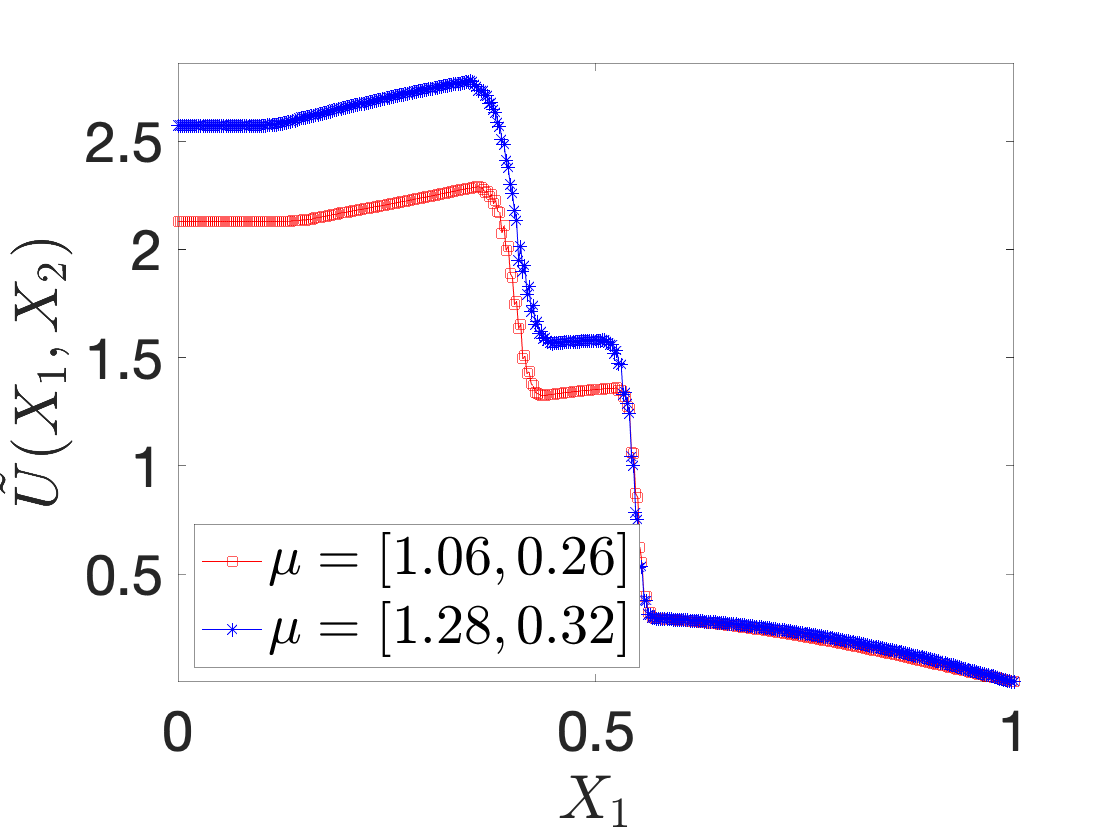}}
  ~~
 \subfloat[$X_2=0.3$] 
{  \includegraphics[width=0.32\textwidth]
 {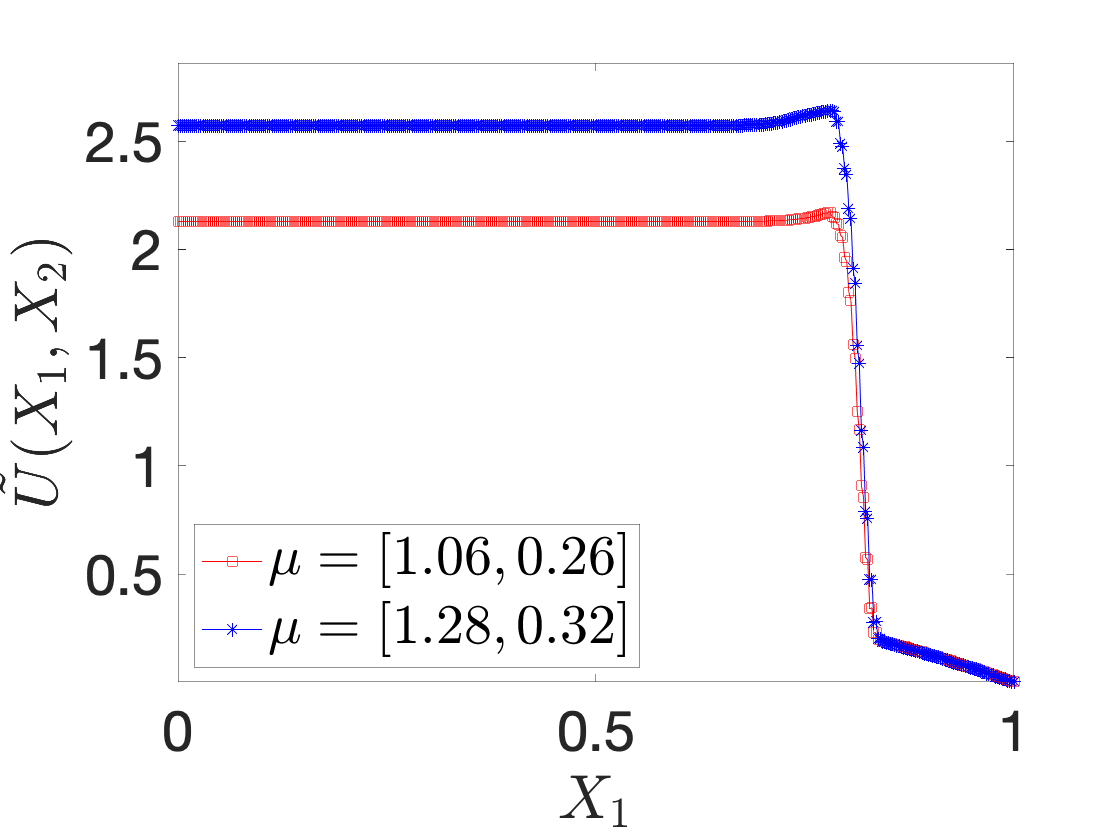}}
   ~~
 \subfloat[$X_2=0.8$] 
{  \includegraphics[width=0.32\textwidth]
 {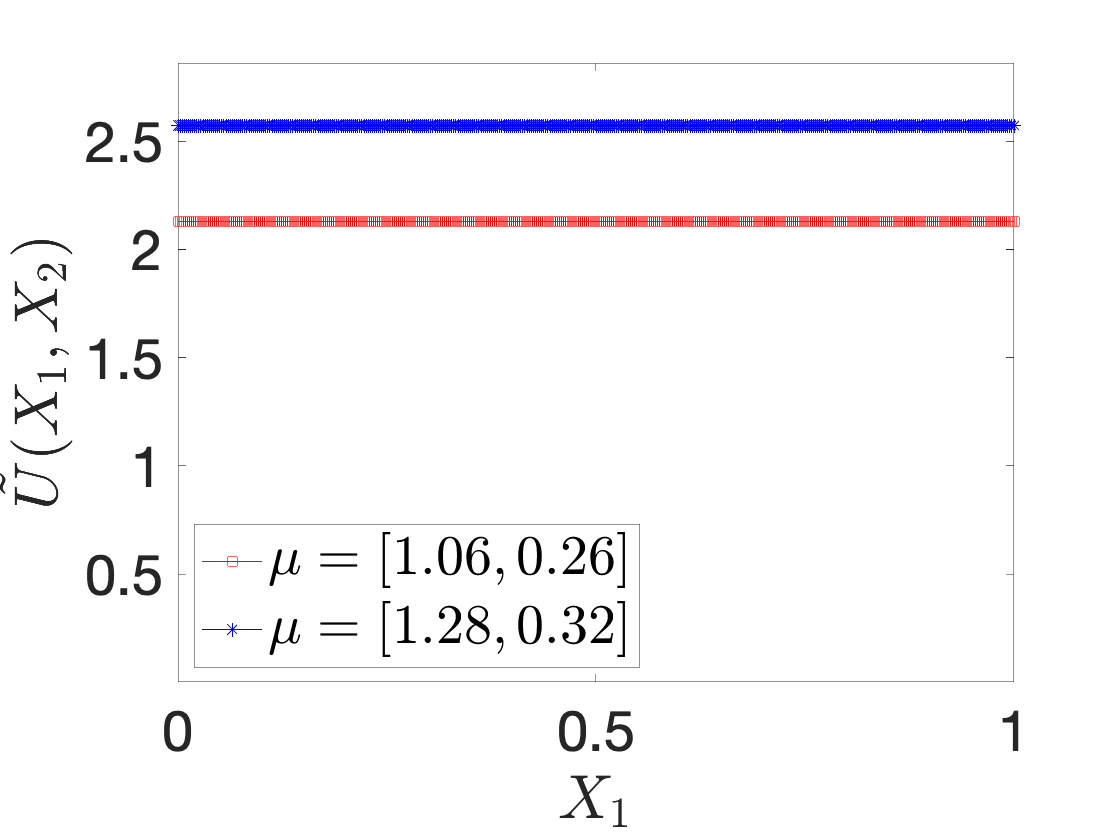}}
  
\caption{Burgers equation;  space-time registration.
Behavior of the solution field for two values of $\mu$ and three values of the second coordinate.
(a-b-c)  behavior in physical domain for $t=0.05,0.3,0.8$.
(d-e-f)  behavior  in  reference domain for
$X_2=0.05,0.3,0.8$.
}
 \label{fig:spacetimereg_burgers2}
  \end{figure}

\subsection{Shallow water equations}
 
We present further investigations of the   
  space-time registration for the shallow water equations.   
  Figure \ref{fig:spacetimereg_sv2} shows the unregistered and registered free surface $z$ for two values of the parameter and for three horizontal slices of $\Omega$. As for the previous model problem, the registration procedure is able to nearly fix the position of the travelling wave with respect to parameter.

 \begin{figure}[h!]
\centering
\subfloat[$t=0.4$] 
{  \includegraphics[width=0.32\textwidth]
 {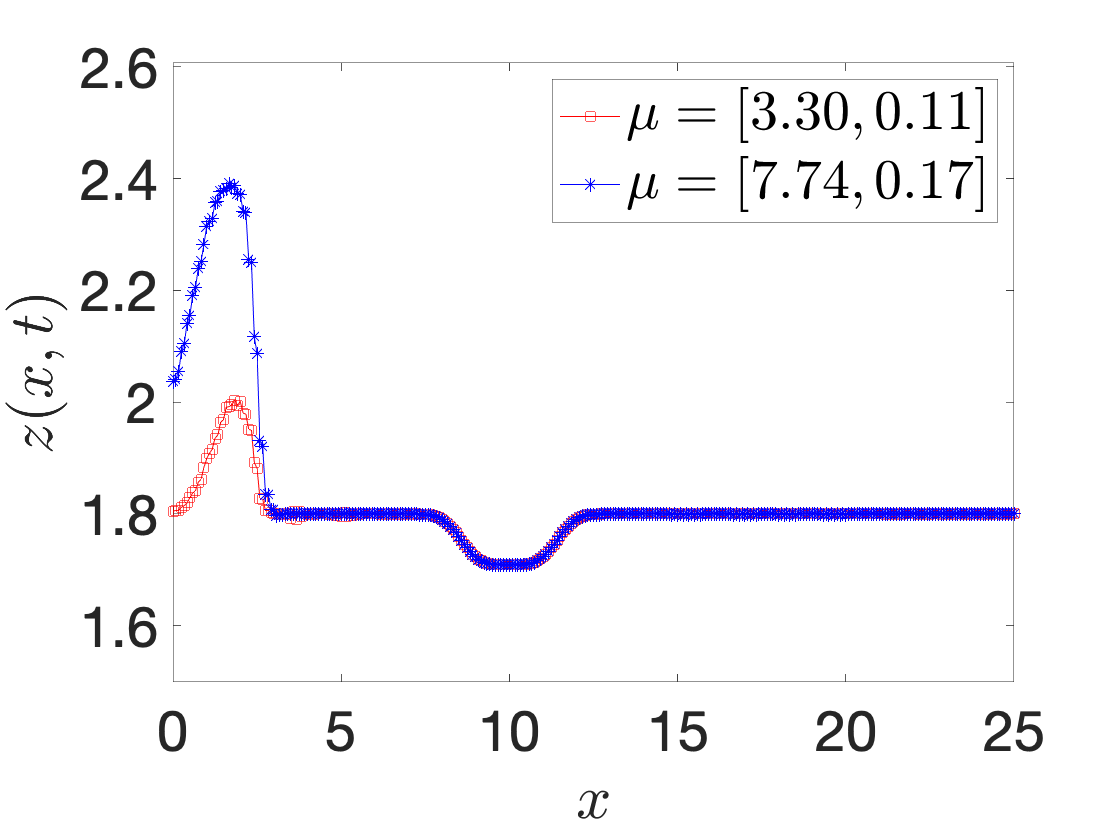}}
  ~~
 \subfloat[$t=1.5$] 
{  \includegraphics[width=0.32\textwidth]
 {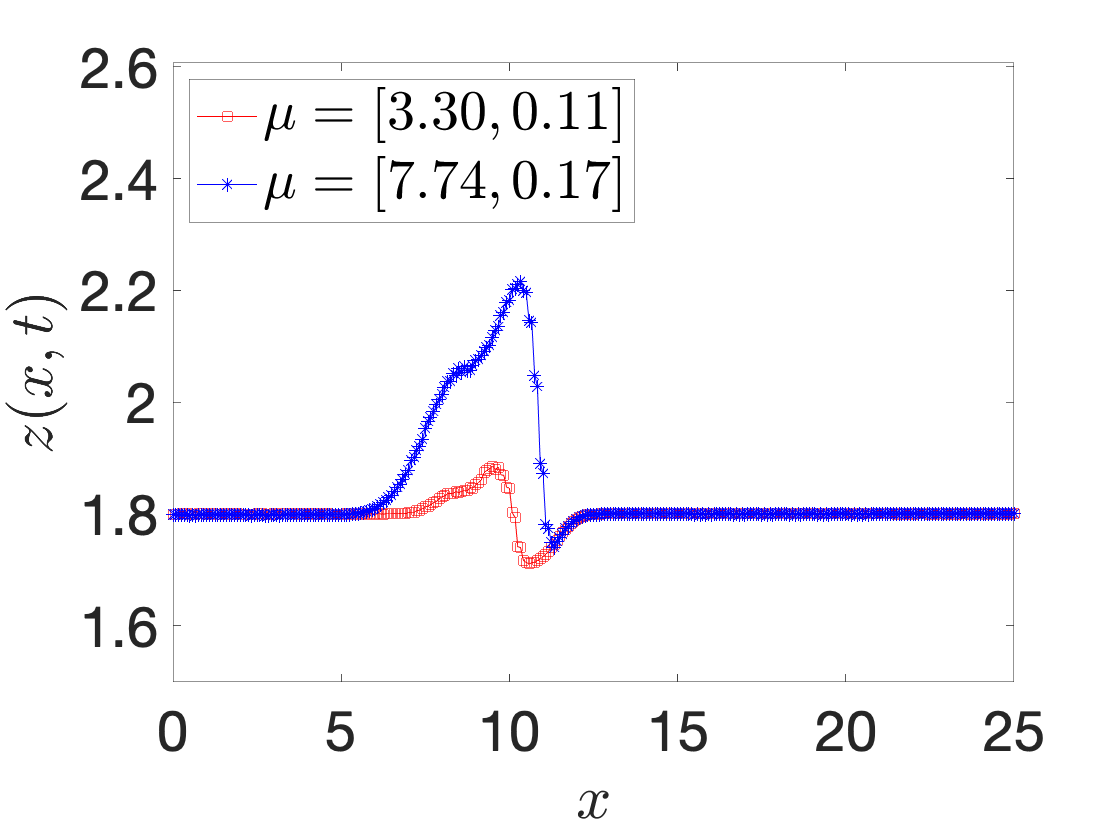}}
   ~~
 \subfloat[$t=3$] 
{  \includegraphics[width=0.32\textwidth]
 {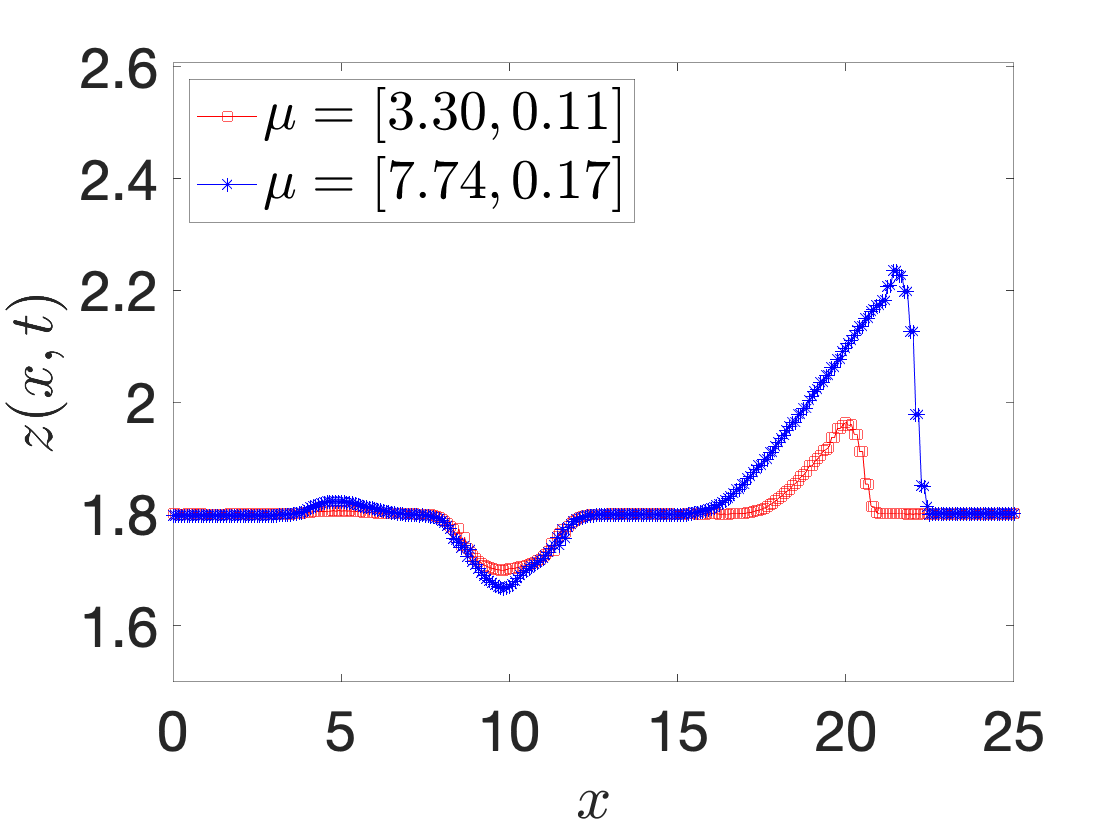}}
 
\subfloat[$X_2=0.4$] 
{  \includegraphics[width=0.32\textwidth]
 {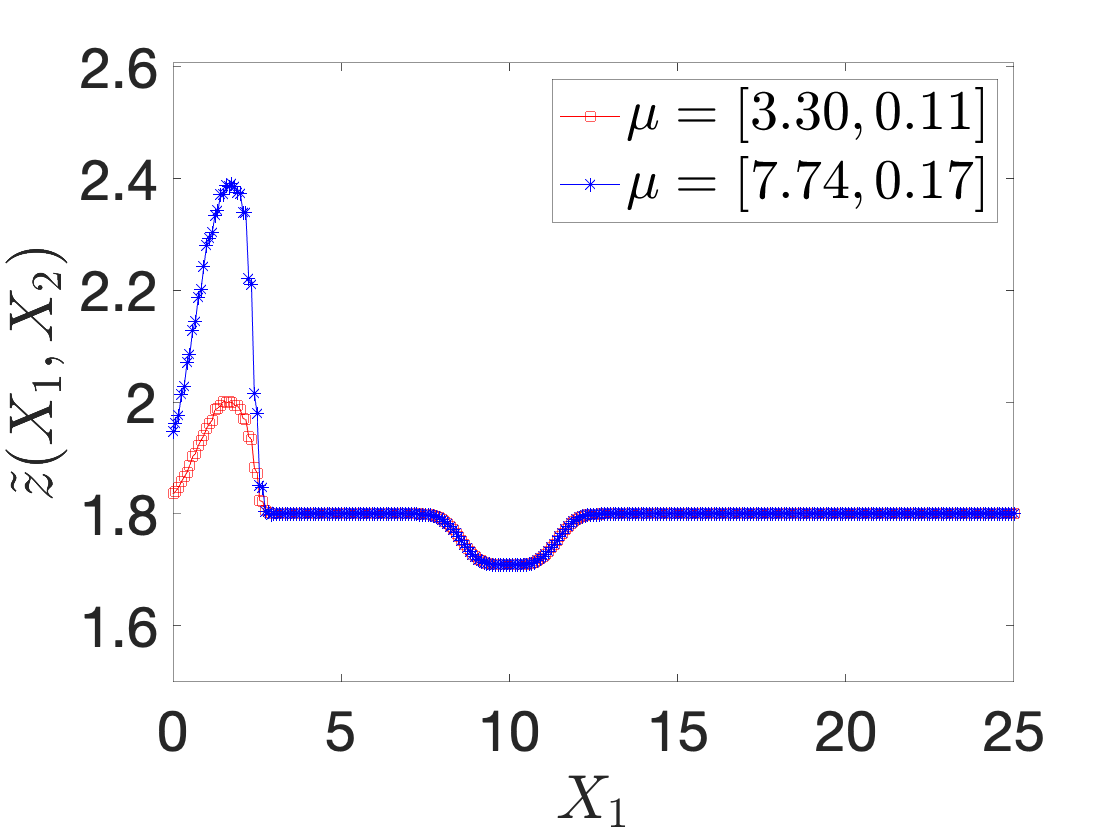}}
  ~~
 \subfloat[$X_2=1.5$] 
{  \includegraphics[width=0.32\textwidth]
 {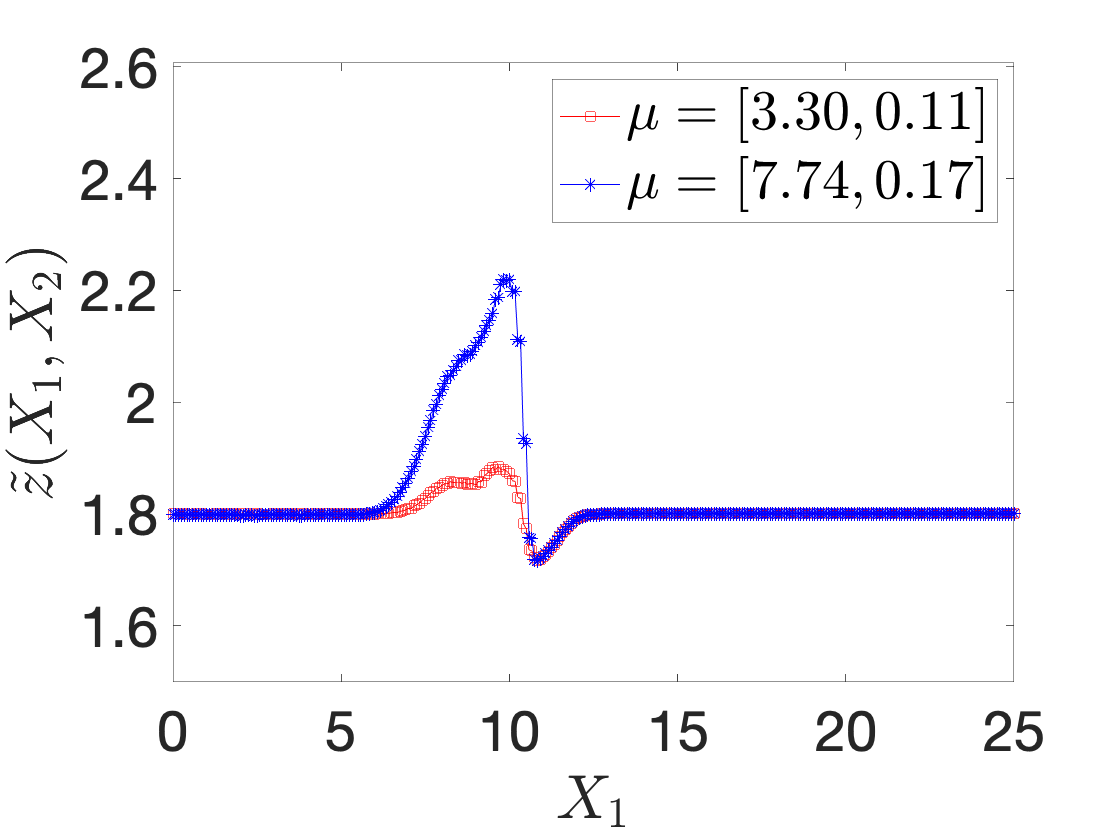}}
   ~~
 \subfloat[$X_2=3$] 
{  \includegraphics[width=0.32\textwidth]
 {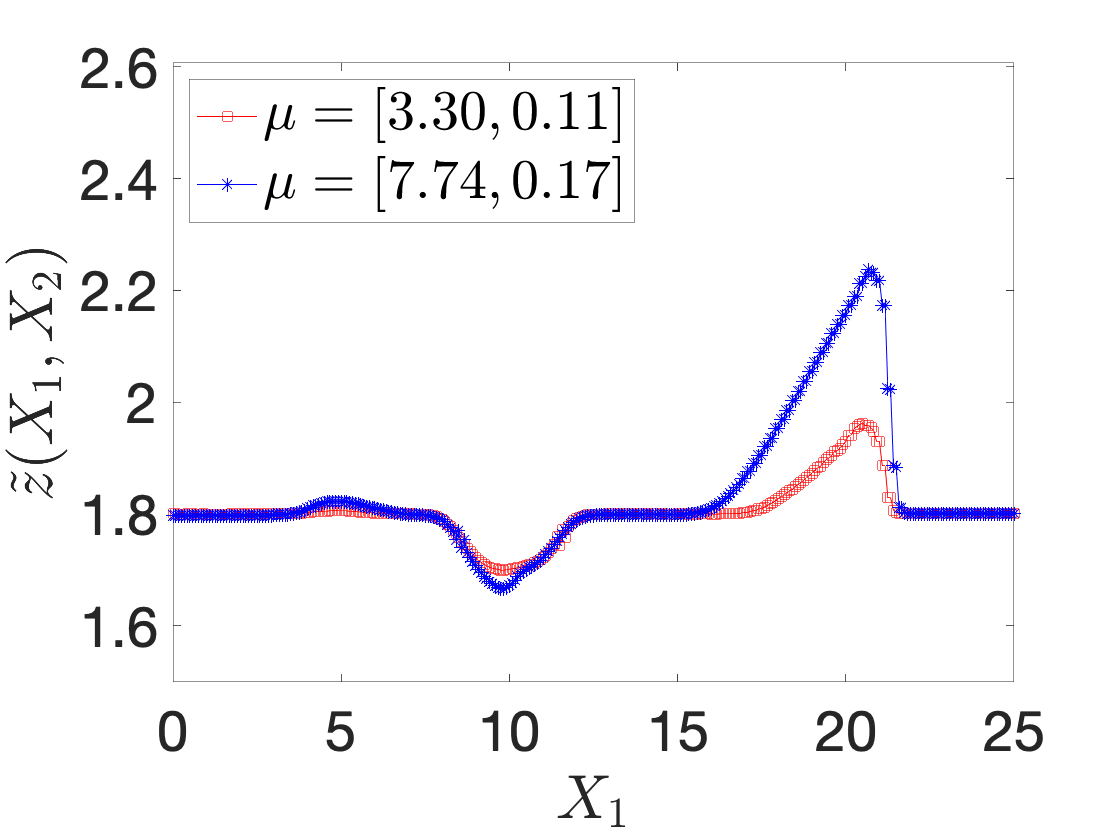}}
  
\caption{Shallow water equations;  space-time registration.
Behavior of the free surface $z$  for two values of $\mu$ and three values of the second coordinate.
(a-b-c) behavior in physical domain for $t=0.4,1.5,3$.
 (d-e-f): behavior  in reference domain for $X_2=0.4,1.5,3$.
}
 \label{fig:spacetimereg_sv2}
  \end{figure}

In Figure \ref{fig:spacetimereg_sv3}, we investigate the optimal reconstruction properties of the proposed approximation. Given $U_{\mu} \in \mathcal{M}$ and the POD space 
$\mathcal{Z}_{N} = {\rm span} \{  \zeta_n\}_{n=1}^N \subset \mathcal{X}$ obtained based on the mapped snapshots $\{ \widetilde{U}_{\mu^k}  \}_{k=1}^{n_{\rm train}}$, we define 
\begin{equation}
\label{eq:Uopt}
\widehat{U}_{\mu}^{\rm opt} := \Pi_{\mathcal{Z}_{N,\mu}}  {U}_{\mu},
\quad
{\rm where} \; 
\mathcal{Z}_{N,\mu} = 
{\rm span} \{  \zeta_n \circ \boldsymbol{\Phi}_{\mu}^{-1}  \}_{n=1}^N.
\end{equation} 
Figures \ref{fig:spacetimereg_sv3}(a)-(b)-(c) show the solution $ {U}_{\mu}$ and the approximation $\widehat{U}_{\mu}^{\rm opt}$ (black continuous line) for two values of $\mu$ and three time instants. 
We here report results for $N=3$.
We observe that we are able to obtain accurate reconstructions with an extremely low-dimensional representation.

\begin{figure}[h!]
\centering
\subfloat[$t=0.4$] 
{  \includegraphics[width=0.32\textwidth]
 {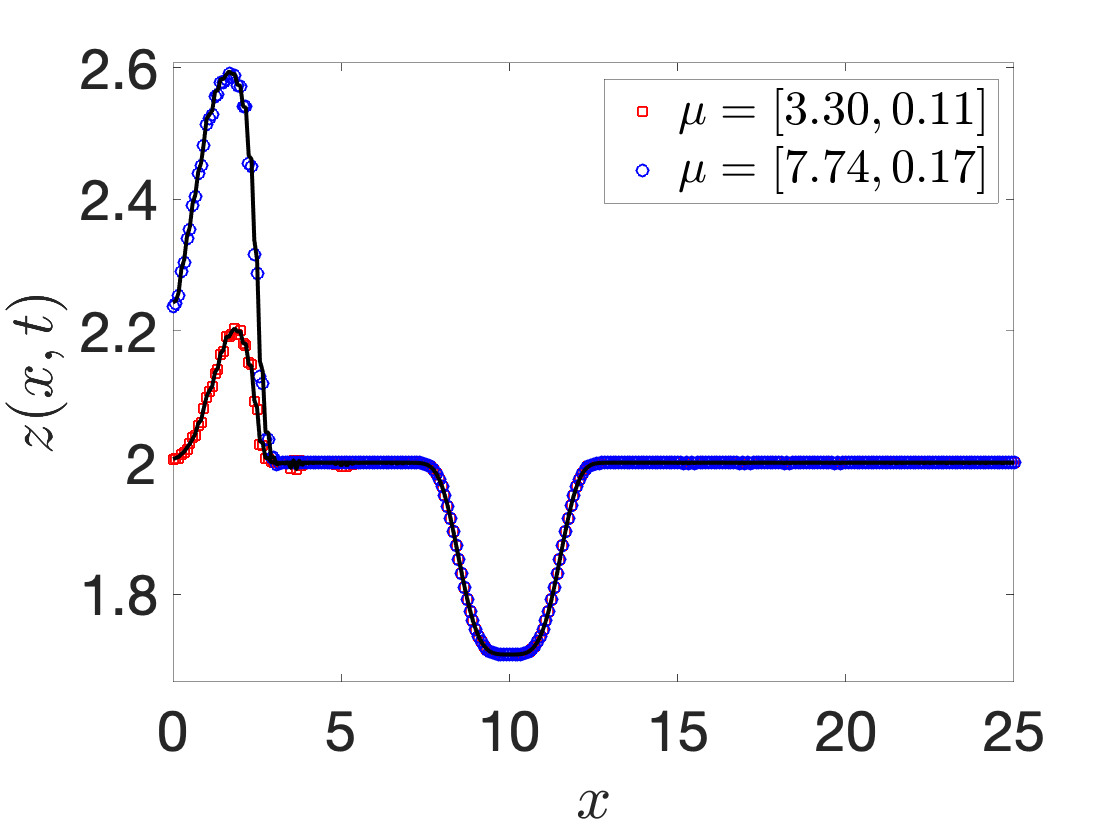}}
  ~~
 \subfloat[$t=1.5$] 
{  \includegraphics[width=0.32\textwidth]
 {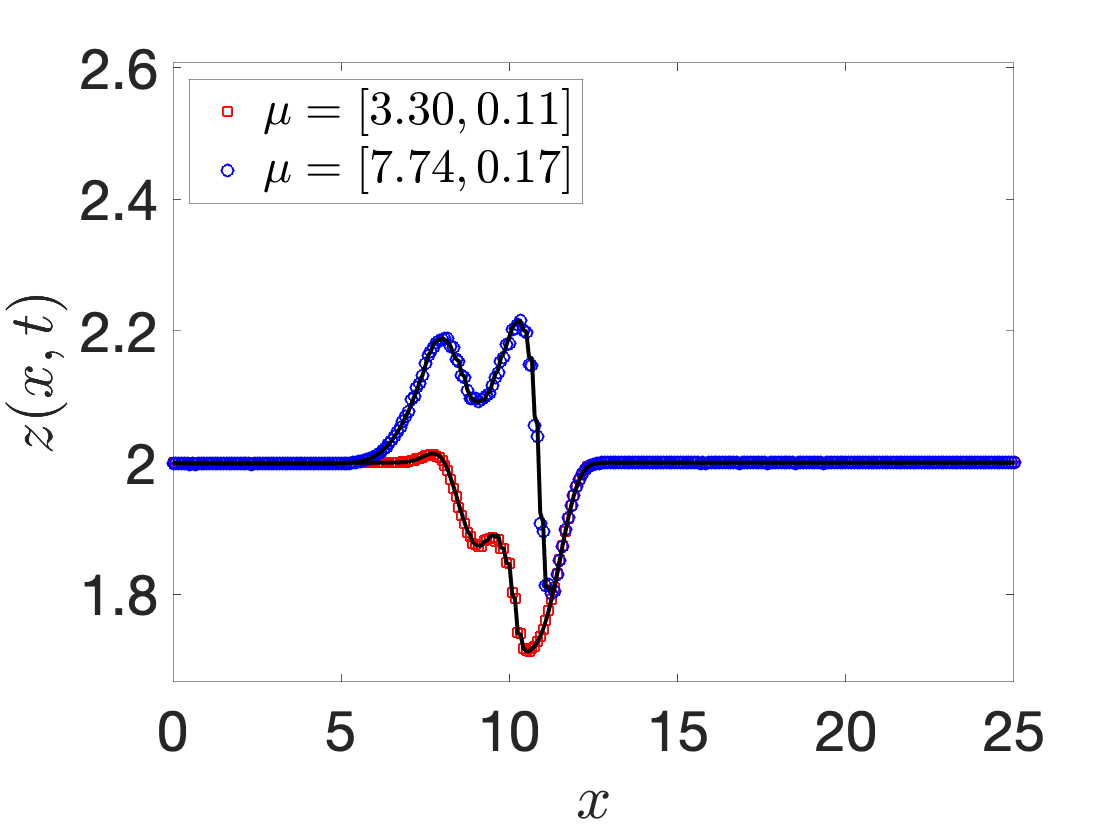}}
   ~~
 \subfloat[$t=3$] 
{  \includegraphics[width=0.32\textwidth]
 {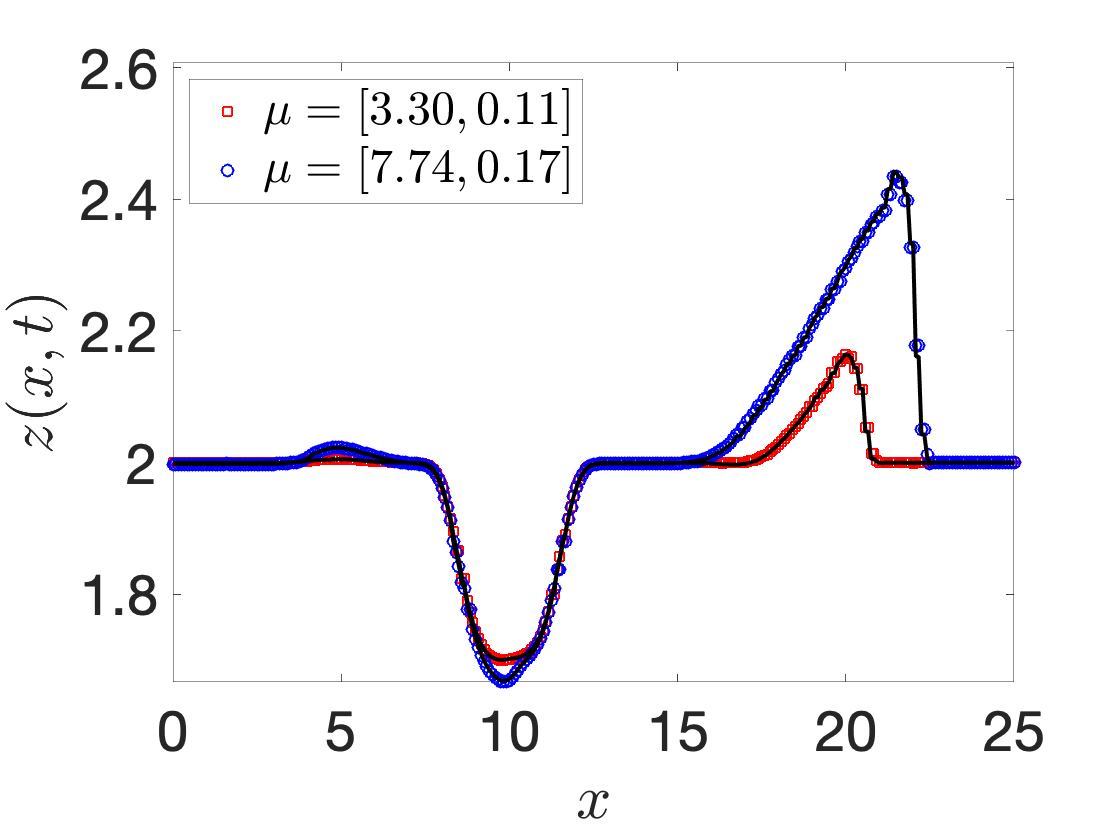}}
 
\caption{Shallow water equations;  space-time registration.
Behavior of the  free surface $z_{\mu}$ and of the optimal nonlinear reconstruction $\widehat{z}_{\mu}^{\rm opt} = ( \widehat{U}_{\mu}^{\rm opt} )_1 + b$ (black line)
(see \eqref{eq:Uopt} and \eqref{eq:bathymetry_saint_venant})
 for two values of $\mu$ and three time instants, with $N=3$.
}
 \label{fig:spacetimereg_sv3}
  \end{figure}  

\bibliographystyle{abbrv}	
\bibliography{all_refs}

\end{document}